\documentclass[11pt,twoside]{amsart}

\usepackage{enumerate}
\usepackage[hyperindex=true,plainpages=false,colorlinks=false,pdfpagelabels]{hyperref}

\theoremstyle{plain}
\newtheorem{The}{Theorem}[section]
\newtheorem*{TheP}{Prototype Result}
\newtheorem*{TheM}{Main Theorem}
\newtheorem{Pro}[The]{Proposition}
\newtheorem{Lem}[The]{Lemma}
\newtheorem{Cor}[The]{Corollary}
\newtheorem*{Cor*}{Corollary}

\theoremstyle{definition}

\theoremstyle{remark}

\newtheorem*{Rem*}{Remark}

\numberwithin{equation}{section}

\DeclareMathOperator{\End}{End}
\DeclareMathOperator{\Hom}{Hom}
\DeclareMathOperator{\GL}{GL}
\DeclareMathOperator{\SL}{SL}

\DeclareMathOperator{\SU}{SU}
\DeclareMathOperator{\Tr}{tr}             
\DeclareMathOperator{\Span}{Span}
\DeclareMathOperator{\Id}{Id}

\DeclareMathOperator{\Ad}{Ad}

\newcommand{\<}{<}
\renewcommand{\>}{>}
\DeclareMathOperator{\sff}{I\!I}
\newcommand{\HS}{H^{\S^3}}

\renewcommand{\Im}{\operatorname{Im}}
\renewcommand{\Re}{\operatorname{Re}}

\renewcommand{\ker}{\operatorname{ker}}
\newcommand{\im}{\operatorname{im}}
\newcommand{\image}{\operatorname{im}}

\DeclareMathOperator{\dbar}{\bar\partial}
\DeclareMathOperator{\del}{\partial}

\newcommand{\R}{\mathbb{R}}
\newcommand{\C}{\mathbb{C}}

\newcommand{\Z}{\mathbb{Z}}

\renewcommand{\H}{\mathbb{H}}  

\newcommand{\CP}{\mathbb{CP}}
\newcommand{\HP}{\mathbb{HP}}
\renewcommand{\S}{S}

\newcommand{\eprint}[1]{e--print: \href{http://#1}{\nolinkurl{#1}}}

\DeclareMathOperator{\Harm}{Harm}

\begin{document}

\title{Constrained Willmore Tori in the 4--Sphere}

\author{Christoph Bohle}

\address{Christoph Bohle\\
  Institut f\"ur Mathematik\\
  Technische Universit{\"a}t Berlin\\
  Stra{\ss}e des 17.\ Juni 136\\
  10623 Berlin\\
  Germany}

\email{bohle@math.tu-berlin.de}


\subjclass{53C42,53A30,53A05}

\date{\today}

\thanks{Author supported by DFG SPP 1154 ''Global Differential Geometry''.}

\begin{abstract} 
  We prove that a constrained Willmore immersion of a 2--torus into the
  conformal 4--sphere~$S^4$ is either of ``finite type'', that is, has a
  spectral curve of finite genus, or is of ``holomorphic type'' which means
  that it is super conformal or Euclidean minimal with planar ends in
  $\R^4\cong S^4\backslash \{ \infty \}$ for some point $\infty\in S^4$ at
  infinity. This implies that all constrained Willmore tori in $S^4$ can be
  constructed rather explicitly by methods of complex algebraic geometry.  The
  proof uses quaternionic holomorphic geometry in combination with integrable
  systems methods similar to those of Hitchin's approach~\cite{Hi} to the
  study of harmonic tori in~$S^3$.
\end{abstract}

\maketitle


\section{Introduction}
\label{sec:intro}

A conformal immersion of a Riemann surface is called a \emph{constrained
  Willmore surface} if it is a critical point of the Willmore functional
$\mathcal{W}=\int_M |\mathring\sff|^2 dA$ (with $\mathring \sff$ denoting the
trace free second fundamental form) under compactly supported infinitesimal
conformal variations, see~\cite{PS87,W,BPP02,BPP1}.  The notion of constrained
Willmore surfaces generalizes that of Willmore surfaces which are the critical
points of $\mathcal{W}$ under all compactly supported variations.  Because
both the functional and the constraint of the above variational problem are
conformally invariant, the property of being constrained Willmore depends only
on the conformal class of the metric on the target space. This suggests an
investigation within a M\"obius geometric framework like the quaternionic
projective model of the conformal 4--sphere used throughout the paper.

The space form geometries of dimension 3 and 4 occur in our setting as
subgeometries of 4--dimensional M\"obius geometry and provide several classes
of examples of constrained Willmore surfaces, including constant mean
curvature (CMC) surfaces in 3--dimensional space forms and minimal surfaces in
4--dimensional space forms.  See \cite{BPP1} for an introduction to
constrained Willmore surfaces including a derivation of the Euler--Lagrange
equation for compact constrained Willmore surfaces.

A prototype for our main theorem is the following result on harmonic maps, the
solutions to another variational problem on Riemann surfaces.
\begin{TheP}
  Let $f\colon T^2 \rightarrow S^2$ be a harmonic map from a 2--torus $T^2$ to
  the metrical 2--sphere $S^2$.  Then either
  \begin{itemize}
  \item $\deg(f)= 0$ and $f$ is of ``finite type'', i.e., has a spectral curve
    of finite genus, or
  \item $\deg(f)\neq 0$ and $f$ is conformal.
  \end{itemize}
\end{TheP}
That $\deg(f)\neq 0$ implies ``holomorphic type'' follows from a more general
result by Eells and Wood \cite{EW83}. That $\deg(f)= 0$ implies ``finite
type'' has been proven by Pinkall and Sterling \cite{PS89} and Hitchin
\cite{Hi}.  In contrast to the conformal case when $f$ itself is
(anti--)holomorphic, parametrizing a harmonic map $f$ of finite type involves
holomorphic functions on a higher dimensional torus, the Jacobian of the
\emph{spectral curve}, an auxiliary compact Riemann surface attached to $f$.

The main theorem of the paper shows that the same dichotomy of ``finite type''
versus ``holomorphic type'' can be observed in case of constrained Willmore
tori $f\colon T^2 \rightarrow S^4$ in the conformal 4--sphere $S^4$.

\begin{TheM}
  Let $f\colon T^2 \rightarrow S^4$ be a constrained Willmore immersion that
  is not Euclidean minimal with planar ends in $\R^4\cong S^4\backslash \{
  \infty \}$ for some point at infinity $\infty\in S^4$. Then either
  \begin{itemize}
  \item $\deg(\perp_f)= 0$ and $f$ is of ``finite type'', i.e., has a spectral
    curve of finite genus, or
  \item $\deg(\perp_f)\neq 0$ and $f$ is super conformal,
  \end{itemize}
  where $\deg(\perp_f)$ is the degree of the normal bundle $\perp_f$ of $f$
  seen as a complex line bundle.
\end{TheM}

Euclidean minimal tori with planar ends play a special role here since they
can have both topologically trivial and non--trivial normal bundle.
Constrained Willmore tori in the conformal 3--sphere~$S^3$ occur in our
setting as the special case of constrained Willmore immersions into~$S^4$ that
take values in a totally umbilic 3--sphere and, in particular, have trivial
normal bundle.

The main theorem implies that every constrained Willmore torus $f\colon T^2
\rightarrow S^4$ can be parametrized quite explicitly by methods of complex
algebraic geometry. If $f$ is of ``holomorphic type'', that is, if $f$ is
super conformal or Euclidean minimal with planar ends, then $f$ or its
differential is given in terms of meromorphic functions on the torus itself: a
super conformal torus is the twistor projection $\CP^3\rightarrow \HP^1$ of an
elliptic curve in $\CP^3$ and for a Euclidean minimal tori with planar ends
there is a point $\infty\in S^4$ such that the differential of $f\colon
T^2\backslash\{p_1,...,p_n\} \rightarrow \mathbb{R}^4=S^4\backslash\{\infty\}$
is the real part of a meromorphic 1--form with $2^{nd}$--order poles and no
residues at the ends $p_1$,...,$p_n$.  It should be noted that both super
conformal and Euclidean minimal tori with planar ends are Willmore and, by the
quaternionic Pl\"ucker formula \cite{FLPP01}, have Willmore energy
$\mathcal{W}=4\pi n$ for some integer~$n\geq 2$.

If $f$ is of ``finite type'', the algebraic geometry needed to parametrize the
immersion is more involved: the immersion is then not given by holomorphic
data on the torus itself, but can be interpreted as a periodic orbit of an
algebraically completely integrable system whose phase space contains as an
energy level the (generalized) Jacobian of a Riemann surface of finite genus,
the spectral curve.  This makes available the methods of ``algebraic
geometric'' or ``finite gap'' integration from integrable systems theory and
implies the existence of explicit parametrizations in terms of theta functions
as in the special case of CMC tori in space forms \cite{Bob}.  The algebraic
geometric reconstruction of general conformally immersed tori with finite
spectral genus from their spectral data will be addressed in a forthcoming
paper.

Our main theorem generalizes the following previous results:
\begin{itemize}
\item CMC tori in 3--dimensional space forms are finite type (Pinkall,
  Sterling~1989,~\cite{PS89}).
\item Constrained Willmore in $S^3$ are of finite type (Schmidt 2002,
  \cite{S02}).
\item Willmore tori in $S^4$ with topologically non--trivial normal bundle are
  super conformal or Euclidean minimal with planar ends (Leschke, Pedit,
  Pinkall 2003, \cite{LPP05}).
\end{itemize}
The fact that CMC tori are of finite type is closely related to the above
prototype result on harmonic tori in $S^2$, because CMC surfaces are
characterized by the harmonicity of their Gauss map $N\colon T^2 \rightarrow
S^2$.  In the early nineties this prototype result had been generalized to
harmonic maps from $T^2$ into various other symmetric target spaces
\cite{Hi,FPPS,BFPP} which led to the conjecture that Willmore tori as well
should be of finite type, because they are characterized by the harmonicity of
their conformal Gauss map or mean curvature sphere congruence.  This
conjecture remained open for more than a decade until Martin Schmidt, on the
last of over 200 pages of \cite{S02}, gave a proof that constrained Willmore
tori in $S^3$ are of finite type.

We investigate constrained Willmore tori by integrable systems methods similar
to those in Hitchin's study \cite{Hi} of harmonic tori in $S^3$.  These
provide a uniform, geometric approach to proving and generalizing the previous
results mentioned above.  The proof roughly consists of the following steps:
\begin{itemize}
\item Reformulation of the Euler--Lagrange equation describing constrained
  Willmore surface as a zero--curvature equation with spectral parameter.
  This zero--curvature formulation arises in the form of an associated family
  $\nabla^\mu$ of flat connections on a trivial complex rank 4 bundle which
  depends on a spectral parameter $\mu\in \mathbb{C}_*$.
\item Investigation of the holonomy representations $H^\mu\colon \Gamma
  \rightarrow \mathrm{SL}_4(\mathbb{C})$ that arise for the associated family
  $\nabla^\mu$ of flat connections of constrained Willmore tori.
\item Proof of the existence of a polynomial Killing field in case the
  holonomy $H^\mu$ is non--trivial.  This implies that a constrained Willmore
  torus with non--trivial holonomy representation has a spectral curve of
  finite genus and hence is of ``finite type''.
\item Proof that a constrained Willmore torus $f$ is of ``holomorphic type''
  if the family of holonomy representations $H^\mu$ of $\nabla^\mu$ is
  trivial.
\end{itemize}
In order to make the strategy of \cite{Hi} work for constrained Willmore tori
in $S^4$ we apply quaternionic holomorphic geometry \cite{FLPP01}, in
particular the geometric approach~\cite{BLPP} to the spectral curve based on
the Darboux transformation for conformal immersions into $S^4$. The main
application of quaternionic methods is in the investigation of which holonomy
representations $H^\mu\colon \Gamma \rightarrow \mathrm{SL}_4(\mathbb{C})$ are
possible for the associated family~$\nabla^\mu$ of constrained Willmore tori.
Understanding the possible holonomies is one of the mayor difficulties in
adapting Hitchin's method to the study of constrained Willmore tori.  The
reason is that, compared to the $\mathrm{SL}_2(\mathbb{C})$--holonomies
arising in the study of harmonic tori in~$S^3$, in case of holomorphic
families of $\mathrm{SL}_4(\mathbb{C})$--representations one has to cope with
a variety of degenerate cases of collapsing eigenvalues.

This difficulty can be handled by applying two analytic results of
quaternionic holomorphic geometry: the quaternionic Pl\"ucker formula
\cite{FLPP01} and the 1--dimensionality \cite{BPP2} of the spaces of
holomorphic sections with monodromy corresponding to generic points of the
spectral curve of a conformally immersed torus $f\colon T^2\rightarrow S^4$
with topologically trivial normal bundle.  The spectral curve as an invariant
of conformally immersed tori was first introduced, for immersions into
3--space, by Taimanov \cite{Ta98} and Grinevich, Schmidt~\cite{GS98}. It is
defined as the Riemann surface normalizing the Floquet--multipliers of a
periodic differential operator attached to the immersion $f$.  Geometrically
this Riemann surface can be interpreted~\cite{BLPP} as a space parameterizing
generic Darboux transforms of $f$.  In the following, the spectral curve will
be referred to as the \emph{multiplier spectral curve} $\Sigma_{mult}$ of $f$
in order to distinguish it from another Riemann surface that arises in our
investigation of constrained Willmore tori.

In Section~\ref{sec:associated} of the paper we review the quaternionic
projective approach to conformal surface theory in $S^4$ and introduce the
associated family $\nabla^\mu$ of flat connections of constrained Willmore
immersions into $S^4=\HP^1$.  This holomorphic family~$\nabla^\mu$ of flat
connections allows to study spectral curves of constrained Willmore tori by
investigating a holomorphic family of ordinary differential operators instead
of the holomorphic family of elliptic partial differential operators needed to
define the spectral curve of a general conformal immersion $f\colon
T^2\rightarrow S^4$ with trivial normal bundle.

In Section~\ref{sec:holonomy} we determine the types of holonomy
representations $H^\mu$ that are possible for the associated
family~$\nabla^\mu$ of constrained Willmore immersions $f\colon T^2\rightarrow
S^4$.  It turns out that there are two essentially different cases: either all
holonomies $H^\mu$, $\mu\in \C_*$ have~$1$ as an eigenvalue of multiplicity 4
or, for generic $\mu\in \C_*$, the holonomy $H^\mu$ has non--trivial, simple
eigenvalues.  The latter occurs only if the normal bundle is trivial and
allows to build a Riemann surface parametrizing the non--trivial eigenlines of
the holonomy.  In the following we call this Riemann surface the
\emph{holonomy spectral curve} $\Sigma_{hol}$ of $f$.

In Section~\ref{sec:asymptotics} we investigate the asymptotics of parallel
sections for $\mu\rightarrow 0$ and $\infty$. This will be essential for
proving the main theorem of the paper.  The asymptotics shows that the
holonomy spectral curve $\Sigma_{hol}$, whenever defined, essentially
coincides with the multiplier spectral curve $\Sigma_{mult}$. In particular,
Darboux transforms corresponding to points of the spectral curve are again
constrained Willmore and Willmore if $f$ itself is Willmore.

In Section~\ref{sec:proof}, the main theorem is proven by separately
discussing all possible cases of holonomy representations that occur for
constrained Willmore tori. For constrained Willmore tori with non--trivial
holonomy we prove the existence of a polynomial Killing field which implies
that $\Sigma_{hol}$ and hence $\Sigma_{mult}$ can be compactified by adding
points at infinity.  The proof of the theorem is completed by showing that
constrained Willmore tori with trivial holonomy are either super conformal or
Euclidean minimal with planar ends.

In Section~\ref{sec:harmonic} we discuss a special class of constrained
Willmore tori which is related to harmonic maps into $S^2$ and for which the
holonomies of the constrained Willmore associated family reduce to
$\SL(2,\C)$--representations.  This class includes CMC tori in $\R^3$ and
$S^3$, Hamiltonian stationary Lagrangian tori in $\C^2\cong \H$, and
Lagrangian tori with conformal Maslov form in $\C^2\cong \H$.  In case the
harmonic map $N\colon T^2\rightarrow S^2$ related to such constrained Willmore
torus~$f\colon T^2\rightarrow S^4$ is non--conformal, the above prototype
result implies that the map $N$ admits a spectral curve of finite genus.  We
show that this harmonic map spectral curve of $N$ coincides with the spectral
curve of the constrained Willmore immersion~$f$.

\section{Constrained Willmore Tori in $S^4$ and Their Associated Family}
\label{sec:associated}

A characteristic property of constrained Willmore surfaces in $S^4=\HP^1$ is
the existence of an associated family of flat connections depending on a
spectral parameter. This associated family of constrained Willmore surfaces is
an essential ingredient in the proof of the main theorem.  It is a direct
generalization of the associated family \cite{FLPP01,LPP05} of Willmore
surface in~$S^4$.  We show that parallel section of the associated family
$\nabla^\mu$ of flat connections of a constrained Willmore immersion $f$ give
rise to Darboux transforms of $f$ that are again constrained Willmore.

\subsection{M\"obius geometry of surfaces in the 4--sphere}
Throughout the paper we model 4--dimensional M{\"o}bius geometry as the
geometry of the quaternionic projective line $\HP^1$, see \cite{BFLPP02} for a
detailed introduction to the quaternionic approach to surface theory.  In
particular, we identify maps $f\colon M \rightarrow S^4$ from a Riemann
surface $M$ into the conformal 4--sphere with line subbundles $L\subset V$ of
a trivial quaternionic rank~2 vector bundle $V$ over $M$ equipped with a
trivial connection $\nabla$.
A map $f$ is a \emph{conformal immersion} if and only if its derivative
$\delta=\pi \nabla_{|L}\in\Omega^1\Hom(L,V/L)$, where $\pi\colon V \to V/L$
denotes the canonical projection, is nowhere vanishing and admits
$J\in\Gamma(\End(L))$ and $\tilde J\in\Gamma(\End(V/L))$ with $J^2=-\Id$ and
$\tilde J^2=-\Id$ such that
\begin{equation}\label{eq:holomorphic_curve_condition}
  *\delta=\delta J = \tilde J \delta
\end{equation}
with $*$ denoting the complex structure of $T^*M$, see Section~4.2 of
\cite{BFLPP02} for details.

A fundamental object of surfaces theory in the conformal 4--sphere $S^4$ is
the \emph{mean curvature sphere congruence} (or conformal Gauss map) of a
conformal immersion $f$. It is the unique congruence $S$ of oriented
2--spheres in $S^4$ that pointwise touches $f$ with the right orientation such
that the mean curvature of each sphere $S(p)$, with respect to any compatible
space form geometry, coincides with the mean curvature of the immersion $f$ at
the point $f(p)$ of contact.

In the quaternionic language, an oriented 2--sphere congruence is represented
by a complex structure on $V$, that is, a section $S\in \Gamma(\End(V))$
satisfying $S^2=-\Id$, with the 2--sphere at a point $p\in M$ corresponding to
the eigenlines of $S_p$.  Such a complex structure $S$ on $V$ gives rise to a
decomposition $\nabla=\partial+\dbar+A+Q$ of the trivial connection $\nabla$,
where $\partial$ and $\dbar$ are $S$--complex linear holomorphic and
anti--holomorphic structures and
\[ A=\tfrac14(S\nabla S+*\nabla S) \qquad \textrm{ and } \qquad
Q=\tfrac14(S\nabla S-*\nabla S).\] The so called \emph{Hopf fields} $A$ and
$Q$ of $S$ are tensor fields $A\in \Gamma(K\End_-(V))$ and $Q\in \Gamma(\bar K
\End_-(V))$, where $\End_-(V)$ denotes the bundle of endomorphisms of $V$ that
anti--commute with~$S$ and where we use the convention that a 1--form $\omega$
taking values in a quaternionic vector bundle (or its endomorphism bundle)
equipped with a complex structure $S$ is called of type $K$ or $\bar K$ if
$*\omega=S\omega$ or $*\omega = -S\omega$.

The mean curvature sphere congruence is characterized as the unique section
$S\in \Gamma(\End(V))$ with $S^2=-\Id$ that satisfies
\begin{equation}\label{eq:mcs_condition}
  SL=L,\qquad *\delta=S\delta=\delta S,\qquad \textrm{and}\qquad Q_{|L}=0,
\end{equation}
see Section~5.2 of \cite{BFLPP02}.  The first two conditions express that, for
every $p\in M$, the sphere $S_p$ touches the immersion $f$ at $f(p)$ with the
right orientation. This is equivalent to the property that $S$ induces the
complex structures $J$ and $\tilde J$ from
\eqref{eq:holomorphic_curve_condition} on the bundles $L$ and $V/L$. The third
condition singles out the mean curvature sphere congruence among all
congruences of touching spheres.  Given the first two conditions, the third
one is equivalent to $\image(A)\subset L$.

The Hopf fields $A$ and $Q$ of the mean curvature sphere congruence $S$
measure the local ``defect'' of the 2--sphere congruence $S$ from being
constant, that is, the defect of the immersion from being totally umbilic. A
conformal immersion $f$ is totally umbilic if both $A$ and $Q$ vanish
identically. In case only one of the Hopf fields vanishes identically the
immersion is called \emph{super conformal} and is the twistor projection of a
holomorphic curve in $\CP^3$, see Chapter~8 of~\cite{BFLPP02}.

The M\"obius invariant quantity measuring the global ``defect'' of $S$ from
being constant is the Willmore functional $\mathcal{W}$ which, for a conformal
immersion $f$ of a compact surface $M$, can be expressed in terms of the Hopf
fields by the formula
\begin{align}
  \label{eq:willmore_functional}
  \mathcal{W} = 2\int_M \< A\wedge * A \> - 2\pi \deg(\perp_f) = 2\int_M \<
  Q\wedge *Q\> + 2\pi \deg(\perp_f),
\end{align}
where $\< \>$ denotes $1/4$ of the real trace and $\deg(\perp_f)$ is the
degree of the normal bundle of the immersion~$f$.

\subsection{Euler--Lagrange equation of constrained Willmore surfaces}\label{sec:el_equation}
The following proposition shows how the \emph{Euler--Lagrange equation}
describing compact constrained Willmore surfaces can be expressed in terms of
the Hopf fields $A$ and $Q$ of the mean curvature sphere congruence~$S$.

\begin{Pro}
  A conformal immersion $f\colon M \rightarrow S^4$ of a compact Riemann
  surface $M$ is constrained Willmore if and only if there exists a 1--form
  $\eta\in \Omega^1(\mathcal{R})$ such that
  \begin{equation}
    \label{eq:def_constrained-willmore}
    d^{\nabla} (2{*}A + \eta ) =0,
  \end{equation}
  where $\mathcal{R}= \{ B\in \End(V) \mid \im(B) \subset L \subset
  \ker(B)\}$.
\end{Pro}
A proof of the Euler--Lagrange equation for constrained Willmore immersions of
compact surfaces can be found in \cite{BPP1}.  The form $\eta$ in
\eqref{eq:def_constrained-willmore} is the Lagrange multiplier of the
underlying constrained variational problem.  The vanishing of $\eta$
corresponds to the case of Willmore surfaces which are characterized by
$d^{\nabla} {*}A =0$, see Chapter~6 of \cite{BFLPP02}.

Equation \eqref{eq:def_constrained-willmore} is equivalent to
\begin{equation}   \label{eq:def_constrained-willmore2}
  d^{\nabla} (2{*}Q + \eta ) =0,
\end{equation}
because $d^{\nabla}{*}Q=d^{\nabla}{*}A$.  Every 1--form $\eta\in
\Omega^1(\mathcal{R})$ with \eqref{eq:def_constrained-willmore} satisfies
$\eta\in \Gamma(K\mathcal{R}_+)$, i.e.,
\begin{equation}
  \label{eq:eta_symmetry}
  *\eta=S\eta=\eta S.
\end{equation}
In fact, equation \eqref{eq:def_constrained-willmore} implies
$\delta\wedge(2{*}A + \eta )=0$ and hence $*\eta = S\eta$, because $*A=SA$.
Similarly, equation \eqref{eq:def_constrained-willmore2} implies $*\eta=\eta
S$.  Using $\nabla=\hat \nabla+ A + Q$, where $\hat \nabla=\partial+\dbar$ is
the $S$--commuting part of $\nabla$, we obtain the decomposition 
\begin{align}
  \label{eq:decomposed_el_eq}
  \underbrace{d^{\hat \nabla} \eta}_+ + \underbrace{ 2 S d^{\hat \nabla} A +
    A\wedge \eta + \eta \wedge Q}_- = d^{\nabla} (2{*}A + \eta ) = 0
\end{align}
of \eqref{eq:def_constrained-willmore} into $S$--commuting and anti--commuting
parts, as usual denoted by $\pm$.  This implies $d^{\hat \nabla} \eta=0$ which
is equivalent to $\eta\delta\in \Gamma(K^2\End_+(L))=\Gamma(K^2)$ being a
holomorphic quadratic differential. In particular, if $\eta$ does not vanish
identically it vanishes at isolated points only. Moreover, by
\eqref{eq:decomposed_el_eq} and the analogous decomposed version of
\eqref{eq:def_constrained-willmore2}, if $\eta\not \equiv 0$ and one of the
Hopf fields $A$ and $Q$ vanishes on some open set $U$, then both $A$ and $Q$
have to vanish on $U$, that is, on $U$ the immersion is totally umbilic.

In the following we denote by $A_\circ$ and $Q_\circ$ the 1--forms defined by
$2{*}A_\circ = 2{*}A + \eta$ and $2{*}Q_\circ=2{*}Q+ \eta$. Like the Hopf
fields $A$ and $Q$ they satisfy
\begin{gather*}
  \nabla S = 2{*}Q_\circ - 2{*}A_\circ, \\
  \im(A_\circ) \subset L \textrm{ and } L\subset \ker(Q_\circ), \\
  *A_\circ = SA_\circ \textrm{ and } *Q_\circ = Q_\circ S.
\end{gather*}
However, in contrast to the Hopf fields, the forms $A_\circ$ and $Q_\circ$ do
not anti--commute with $S$ if $\eta$ does not vanish identically.

\subsection{Uniqueness and non--uniqueness of the Lagrange multiplier $\eta$.}
For discussing the uniqueness of the Lagrange multiplier $\eta$ in the
Euler--Lagrange equation \eqref{eq:def_constrained-willmore} of constrained
Willmore surfaces we need the following quaternionic characterization of
isothermic surfaces, see e.g.\ \cite{HJ03,Boh03}: a conformal immersion
$f\colon M \rightarrow S^4$ is \emph{isothermic} if there is a non--trivial
1--form $\omega\in \Omega^1(\mathcal{R})$ with $d^\nabla \omega=0$, where
$\mathcal{R}= \{ B\in
\End(V) \mid \im(B) \subset L \subset \ker(B)\}$ as in
Section~\ref{sec:el_equation}.  As above one can prove that every closed
1--form $\omega \in \Omega^1(\mathcal{R})$ satisfies $\omega \in
\Gamma(K\mathcal{R}_+)$, that is,
\begin{equation}
  \label{eq:isothermic}
  *\omega = S\omega = \omega S,
\end{equation} and that the quadratic differential $\omega\delta \in
\Gamma(K^2\End_+(L))= \Gamma(K^2)$ is  holomorphic.
With this definition of isothermic surfaces, the following lemma is evident.

\begin{Lem}
  The Lagrange--multiplier $\eta$ occurring in the Euler--Lagrange equation
  \eqref{eq:def_constrained-willmore} of constrained Willmore surfaces is
  either unique or the surface is isothermic. In the latter case, the form
  $\eta$ is unique up to adding a closed form $\omega\in
  \Omega^1(\mathcal{R})$.
\end{Lem}

For isothermic surfaces that are not totally umbilic, the space of closed
forms in $\Omega^1(\mathcal{R})$ is real 1--dimensional, see e.g.\
\cite{Boh03}.  Examples of constrained Willmore surfaces for which the
Lagrange--multiplier $\eta$ is not unique are CMC surfaces in 3--dimensional
space forms, cf.\ Sections~\ref{sec:cmc_r3} and \ref{sec:cmc_s3}.  In fact,
CMC tori with respect to 3--dimensional space form subgeometries are the only
possible examples of constrained Willmore tori in the conformal 3--sphere with
non--unique Lagrange multiplier $\eta$, cf.~\cite{BPP02}.

Examples of constrained Willmore tori whose Lagrange multiplier~$\eta$ is
unique (namely $\eta\equiv 0$) are super conformal tori, see the discussion
following equation~\eqref{eq:decomposed_el_eq}.  Examples of constrained
Willmore tori in the 3--sphere which are non--isothermic (and hence not CMC
with respect to any space form subgeometry) can be obtained from Pinkall's
Hopf torus construction \cite{Pi}: a Hopf torus, the preimage of a closed
curve in $S^2$ under the Hopf fibration $S^3\rightarrow S^2$, is never
isothermic unless it is the Clifford torus. It is Willmore if the curve in
$S^2$ is elastic \cite{Pi} and it is constrained Willmore if the underlying
curve is generalized elastic \cite{BPP1}.

\subsection{Associated family $\nabla^\mu$ of flat connections of constrained
  Willmore surfaces in the 4--sphere}\label{sec:nabla_mu}
The fundamental tool in our study of constrained Willmore tori is the
\emph{associated family}
\begin{equation} \label{eq:nabla_mu} \nabla^\mu = \nabla +(\mu-1) \frac{1-i
    S}2 A_\circ+ (\mu^{-1}-1) \frac{1+i S}2 A_\circ,
\end{equation}
of complex connections on the complex rank~4 bundle $(V,i)$ which rationally
depends on the \emph{spectral parameter} $\mu \in \C_*$, where $(V,i)$ denotes
$V$ seen as a complex vector bundle by restricting the scalar field to
$\C=\Span_\R \{1,i\}$.  In the following we call a 1--form $\omega$ to be of
type $(1,0)$ or $(0,1)$ if $*\omega=\omega i$ or $*\omega=-\omega i$.  Setting
$A_\circ^{(1,0)}= \frac{1-i S}2 A_\circ$ and $A_\circ^{(0,1)}= \frac{1+i S}2
A_\circ$, formula \eqref{eq:nabla_mu} simplifies to
\[ \nabla^\mu = \nabla +(\mu-1) A_\circ^{(1,0)}+ (\mu^{-1}-1) A_\circ^{(0,1)}.
\]
\begin{Lem}
  The connection $\nabla^\mu$ is flat for every spectral parameter $\mu \in
  \C_*$.
\end{Lem}
\begin{proof}
  The curvature of $\nabla^\mu$ is
  \begin{multline*}
    R^{\nabla^\mu} = (\mu-1) d^\nabla A_\circ^{(1,0)} + (\mu^{-1}-1) d^\nabla
    A_\circ^{(0,1)} + (\mu-1)(\mu^{-1}-1) ( A_\circ^{(1,0)}\wedge
    A_\circ^{(0,1)} + A_\circ^{(0,1)} \wedge A_\circ^{(1,0)}).
  \end{multline*}
  From $d^\nabla*A_\circ= 0$ we obtain $d^\nabla A_\circ^{(1,0)} = d^\nabla
  A_\circ^{(0,1)} = \frac12 d^\nabla A_\circ = A_\circ \wedge A_\circ$ such
  that $R^{\nabla^\mu}=0$, because $(\mu-1)(\mu^{-1}-1)( A_\circ^{(1,0)}\wedge
  A_\circ^{(0,1)} + A_\circ^{(0,1)} \wedge A_\circ^{(1,0)})= (2-\mu-\mu^{-1})
  (A_\circ \wedge A_\circ)$.
\end{proof}

This associated family $\nabla^\mu$ of flat connections has the symmetry
\begin{equation}
  \label{eq:symmetry_nabla_mu}
  \nabla^{1/\bar \mu} = j^{-1} \nabla^\mu j \qquad \textrm{ for all } \mu
  \in \C_*
\end{equation}
with $j$ denoting the complex anti--linear endomorphism of $(V,i)$ given by
right--multiplication with the quaternion $j$. For every $\mu\in S^1\subset
\C_*$, the connection $\nabla^\mu$ is therefore quaternionic.

Another holomorphic family of flat complex connections is given by
\[ \tilde \nabla^\mu= \nabla +(\mu-1) Q_\circ^{(1,0)} + (\mu^{-1}-1) 
Q^{(0,1)}_\circ
\]
with $Q_\circ^{(1,0)}= Q_\circ \frac{1-i S}2$ and $Q_\circ^{(0,1)}= Q_\circ
\frac{1+i S}2$. Both families are in fact gauge equivalent:
\begin{align}
  \label{eq:gauge_trafo}
  \nabla^\mu = \big( (\mu + 1) - i (\mu-1) S\big) \circ \tilde \nabla^\mu
  \circ \big((\mu+1) - i (\mu-1) S\big)^{-1} \quad \textrm{ for all } \mu
  \in \C_*.
\end{align}
As a consequence we obtain that in case of super conformal Willmore surfaces
the connection $\nabla^\mu$ is trivial for every $\mu\in \C_*$, because either
$A\equiv 0$ or $Q\equiv 0$.

The family $\tilde \nabla^\mu$ of connections arises naturally by dualizing
the associated family
\begin{align}
  \label{eq:associated_perp}
  (\nabla^\perp)^\mu = \nabla + (\mu-1) (A_\circ^\perp)^{(1,0)} + (\mu^{-1}-1)
(A_\circ^\perp)^{(0,1)}
\end{align}
of the dual constrained Willmore surface $L^\perp\subset V^*$ seen as
connection on the complex bundle $(V^*,-i)$: the immersion $f^\perp$ has the
mean curvature sphere $S^\perp = S^*$ with Hopf fields $A^\perp = -Q^*$ and
$Q^\perp = -A^*$. The form $2{*}A_\circ^\perp := 2{*} A^\perp + \eta^\perp$
with $\eta^\perp := -\eta^*$ is closed which shows that $f^\perp$ is again
constrained Willmore.  The complex bundle $(V^*,-i)$ carries the usual complex
structure of the complex dual $(V,i)^*$ to $(V,i)$ when applying the canonical
identification between quaternionic and complex dual space, that is, the
identification between $\alpha\in V^*$ and its complex part $\alpha_\C\in
(V,i)^*$. The sign in $(V^*,-i)$ reflects the fact that $V^*$ is made into a
quaternionic right vector bundle by $\alpha\lambda := \bar \lambda \alpha$ for
$\alpha\in V^*$ and $\lambda\in \H$.

\subsection{Darboux transforms}
The following lemma is essential for the next sections as it provides a link
between the associated family $\nabla^\mu$ of flat connections and
quaternionic holomorphic geometry.  Recall that a conformal immersion $f\colon
M \rightarrow S^4\cong \HP^1$ induces a unique quaternionic holomorphic
structure \cite{FLPP01} on the bundle $V/L$ with the property that all
parallel sections of the trivial connection on the bundle $V$ project to
holomorphic sections: the complex structure on $V/L$ is $\tilde J$ from
\eqref{eq:holomorphic_curve_condition} while the holomorphic structure
$D\colon \Gamma(V/L) \rightarrow \Gamma(\bar K V/L)$ is defined by $D\pi =
(\pi\nabla)''$, where $\pi \colon V \rightarrow V/L$ is the canonical
projection and $()''$ denotes the $\bar K$--part with respect to $\tilde J$,
see \cite{BLPP} for details.  The projection
\begin{align}
  \label{eq:prolongation}
  \psi\in \Gamma(V) \quad \mapsto \quad \tilde \psi:= \pi \psi \in \Gamma(V/L)
\end{align}
induces a 1--1--correspondence between sections of $V$ with $\nabla\psi \in
\Omega^1(L)$ and holomorphic sections of $V/L$; the section $\psi$ with
$\nabla\psi \in \Omega^1(L)$ and $\tilde \psi:= \pi \psi$ is called the
\emph{prolongation} of the holomorphic section $\tilde \psi$.  Existence and
uniqueness of the prolongation $\psi$ for a given holomorphic section $\tilde
\psi$ immediately follows from the fact that $\delta=\pi\nabla_{|L}$ is
nowhere vanishing.  Flatness of $\nabla$ implies that $\nabla\psi \in
\Gamma(KL)$ for every $\psi\in \Gamma(V)$ with $\nabla\psi \in \Omega^1(L)$.

\begin{Lem}\label{lem:nabla_mu_parallel_prolonge_hol}
  Let $f\colon M \rightarrow S^4$ be a constrained Willmore immersion and
  denote by $\nabla^\mu$ its associated family of flat connections. Then every
  (local) $\nabla^\mu$--parallel section of $V$ is the prolongation of a
  (local) holomorphic section of $V/L$.
\end{Lem}
\begin{proof}
  The lemma is an immediate consequence of the fact that $A_\circ$ takes
  values in $L$, because every $\nabla^\mu$--parallel section $\psi$ of $V$
  satisfies
  \[ \nabla\psi = (1-\mu) A_\circ^{(1,0)} \psi + (1-\mu^{-1}) A_\circ^{(0,1)}
  \psi \in \Omega^1(L).  \vspace{-0.5cm} \] 
\end{proof}
 
A map $f^\sharp\colon M \rightarrow S^4$ is called a \emph{Darboux transform}
\cite{BLPP} of a conformal immersion $f\colon M \rightarrow S^4 \cong \HP^1$
if the corresponding line subbundle $L^\sharp\subset V$ is locally of the form
$L^\sharp=\psi\H$ for $\psi$ the prolongation of a nowhere vanishing
holomorphic section of $V/L$. In case $f$ is constrained Willmore we call
$f^\sharp$ a \emph{$\nabla^\mu$--Darboux transform} if there is $\mu\in \C_*$
such that the corresponding bundle is locally of the form $L^\sharp=\psi\H$
for $\psi$ a $\nabla^\mu$--parallel section whose projection to $V/L$ is
nowhere vanishing.

\begin{The}\label{the:nabla_dt_cw}
  Let $f\colon M \rightarrow S^4$ be a constrained Willmore immersion of a
  Riemann surface $M$. Every $\nabla^\mu$--Darboux transform $f^\sharp\colon M
  \rightarrow S^4$ of~$f$ is again constrained Willmore when restricted to the
  open subset of $M$ over which it is immersed.  In case $f$ is Willmore and
  $f^\sharp$ is a $\nabla^\mu$--Darboux transform (for $\nabla^\mu$ taken with
  $\eta\equiv 0$), then $f^\sharp$ is again Willmore where immersed.
\end{The}

\begin{proof}
  The proof uses notation and several results from~\cite{Boh03}.  If
  $f^\sharp$ is a Darboux transform of $f$, the corresponding line bundles
  satisfy $V=L\oplus L^\sharp$ and locally $L^\sharp$ admits a nowhere
  vanishing section $\psi$ with $\nabla\psi\in \Omega^1(L)$. In case $\psi$ is
  $\nabla^\mu$--parallel for $\mu\in \C\backslash\{ 0,1\}$, by definition of
  $\nabla^\mu$ we have
  \begin{equation}
    \label{eq:cw_preserved}
    \nabla \psi \, b +* \nabla\psi = (2* A_\circ) \psi = (2*A+\eta) \psi
  \end{equation}
  with $b\in \C$  defined by
  \begin{equation}
    \label{eq:b}
    b = \frac{2i}{1-\mu} - i = \frac{-2i}{1-\mu^{-1}} + i.   
  \end{equation}
  The endomorphisms $B$, $C\in \Gamma(\End(L^\sharp))$ in equation $(74)$ of
  \cite{Boh03} then satisfy $(B+C)\psi=\psi b$ and $B+C$ is parallel, because
  $b$ is constant. This shows that equation $(75)$ of \cite{Boh03} is
  satisfied for $D:=C$ such that $f^\sharp$ is again constrained Willmore.  In
  particular, if $f$ is Willmore and $\eta\equiv 0$, then $C\equiv 0$ and
  hence $D\equiv 0$ which proves that $f^\sharp$ is also Willmore.
\end{proof}

\subsection{Global Darboux transforms of conformal tori in the 4--sphere}\label{sec:global_DT}
In case the underlying surface $M$ is a torus $T^2=\C/\Gamma$, global Darboux
transforms of $f$ are obtained from prolongations of holomorphic sections with
\emph{monodromy} of $V/L$: these are sections $\psi\in \Gamma(\tilde V)$ of
the pullback $\tilde V$ of $V$ to the universal covering $\C$ of $T^2$ which
transform by
\begin{equation}
  \label{eq:hol_monodromy}
   \gamma^*\psi = \psi h_\gamma, \qquad  \gamma\in \Gamma
\end{equation}
for some \emph{multiplier} $h\in \Hom(\Gamma,\H^*)$ and have derivative
$\nabla\psi$ with values in the pullback $\tilde L$ of $L$.  Multiplying
$\psi$ by a quaternionic constant $\lambda\in \H_*$ yields the same Darboux
transform while the multiplier $h$ gets conjugated $\lambda^{-1}h\lambda$.
Because the group $\Gamma$ of deck transformations of the torus is abelian it
is therefore sufficient to consider prolongations of holomorphic section with
complex multiplier $h\in \Hom(\Gamma,\C_*)$.

For a constrained Willmore torus $f\colon T^2 \rightarrow S^4 \cong \HP^1$,
one can obtain global $\nabla^\mu$--Darboux transforms from holomorphic
sections with monodromy of $V/L$ whose prolongation $\psi \in \Gamma(\tilde
V)$ is $\nabla^\mu$--parallel for some $\mu\in \C_*$.  Every
$\nabla^\mu$--parallel section $\psi \in \Gamma(\tilde V)$ satisfying
\eqref{eq:hol_monodromy} for some $h\in \Hom(\Gamma,\C_*)$ gives rise to a
$\nabla^\mu$--parallel (complex) line subbundle of $V$ and vice versa.  Such
line subbundles correspond to 1--dimensional invariant subspaces of the
holonomy representation, i.e., to simultaneous eigenlines of $H^\mu_p(\gamma)$
for all $\gamma\in \Gamma$. These will be studied in the following section.

\section{Holonomy of Constrained Willmore Tori}
\label{sec:holonomy}

The characterization of constrained Willmore tori in~$S^4$ in terms of the
associated family $\nabla^\mu$ of flat connections puts us in a situation that
is quite familiar in integrable systems theory.  What one usually does when
encountering a family of flat connections over the torus is to investigate its
holonomy representations and, in particular, the eigenlines of the holonomy.
In general, investigating the holonomies of a family of flat connections on a
complex rank~4 bundle over the torus $T^2=\C/\Gamma$ is more involved than in
case of rank~2 bundles (like for harmonic tori in $S^2$ or $S^3$, see
\cite{Hi} or Section~\ref{sec:harmonic_spec} below) because one has to deal
with various possible configurations of collapsing eigenvalues.  In the
present section we show that for the associated family $\nabla^\mu$ of
constrained Willmore tori only few of these configurations do actually occur.

\subsection{Main result of the section}
Before stating the main result of the section we collect the relevant
properties of the holonomy representation for a family $\nabla^\mu$ of flat
connections on a surface $M$:
\begin{itemize}
\item Because all $\nabla^\mu$ are flat, for a fixed $p\in M$ and fixed
  $\mu\in \C_*$, the holonomy $H^\mu_p(\gamma)\in\GL_\C(V_p)$ depends only on
  the homotopy class of closed curves based at the point $p$.  The holonomy is
  thus a representation $\gamma\in \Gamma \mapsto
  H^\mu_p(\gamma)\in\GL_\C(V_p)$ of the group $\Gamma$ of deck
  transformations.
\item For fixed $p\in M$ and $\gamma\in \Gamma$, the holonomies
  $H^\mu_p(\gamma)$ depend holomorphically on the spectral parameter $\mu\in
  \C_*$.
\item The holonomies for different points on the torus are conjugated; the
  eigenvalues of $H^\mu_p(\gamma)$ are therefore independent of $p\in M$, only
  the eigenlines change when changing $p\in M$.
\end{itemize}
In case the underlying surface is a torus $T^2=\C/\Gamma$, the group $\Gamma$
of Deck transformations is abelian and the holonomies $H^\mu_p(\gamma_1)$ and
$H^\mu_p(\gamma_2)$ commute for all $\gamma_1$, $\gamma_2\in \Gamma$. For
fixed $p\in T^2$ and $\mu\in \C_*$, the eigenspaces of the holonomy
$H^\mu_p(\gamma_1)$ for one $\gamma_1\in \Gamma$ are thus invariant subspaces
of all other holonomies $H^\mu_p(\gamma_2)$, $\gamma_2\in \Gamma$. In
particular, simple eigenspaces are eigenspaces of all holonomies. More
generally, every eigenspace of $H^\mu_p(\gamma_1)$ for one $\gamma_1\in
\Gamma$ contains a simultaneous eigenline of $H^\mu_p(\gamma_2)$ for all
$\gamma_2\in \Gamma$.  The restriction of the holonomy representation
$H^\mu_p(\gamma)$ to such a simultaneous eigenline is a multiplier $h\in
\Hom(\Gamma,\C_*)$ which is the monodromy of the $\nabla^\mu$--parallel
sections $\psi\in \Gamma(\tilde V)$ whose value $\psi_p$ at $p$ is contained
in the eigenline (cf.\ Section~\ref{sec:global_DT}).

\begin{Pro} \label{prop:cases} Let $f\colon T^2 \rightarrow S^4$ be a
  constrained Willmore torus. The holonomy representations of the associated
  family $\nabla^\mu$ of $f$ belong to one of the following three cases:
  \begin{itemize}
  \item[I.] there is $\gamma\in \Gamma$ such that, away from isolated $\mu\in
    \C_*$, the holonomy $H^\mu_p(\gamma)$ has $4$ distinct eigenvalues which
    are non--constant as functions of $\mu$,
  \item[II.] all holonomies $H^\mu_p(\gamma)$ have a 2--dimensional common
    eigenspace with eigenvalue~1 and there is $\gamma\in \Gamma$ such that,
    away from isolated $\mu\in \C_*$, the holonomy $H^\mu_p(\gamma)$ has 2
    simple eigenvalues which are non--constant as functions of $\mu$, or
  \item[III.]  all $H^\mu_p(\gamma)$, $\gamma\in \Gamma$ have $1$ as an
    eigenvalue of multiplicity 4.  More precisely, either
    \begin{enumerate}
    \item[(a)] all holonomies are trivial, i.e., $H^\mu_p(\gamma)=\Id$ or
    \item[(b)] all holonomies are of Jordan--type with two $2\times2$
      Jordan--blocks.
    \end{enumerate}
  \end{itemize}
  If the immersions $f$ has topologically non--trivial normal bundle, it
  belongs to Case~III.
\end{Pro}
The proposition will be proven in Sections~\ref{proof:non--trivial}
and~\ref{proof:trivial}.  The cases of immersions with trivial and
non--trivial normal bundle are treated separately, because the situation in
both cases is quite different and so is the analysis needed in the proof.  In
the quaternionic model, the normal bundle of an immersion is the bundle
$\Hom_-(L,V/L)$, where ``$-$'' denotes the homomorphisms anti--commuting with
$J$ and $\tilde J$.  The degree of the normal bundle for an immersion of a
compact surface is
\[\deg(\perp_f) = \deg(\Hom_-(L,V/L)) = 2\deg(V/L)+\deg(K),\]
where the last equality holds because the differential $\delta$ of $f$ is a
nowhere vanishing section of $K\Hom_+(L,V/L)$. In particular, the normal
bundle of an immersed torus is trivial if and only if the induced quaternionic
holomorphic line bundle $V/L$ has degree zero.

All cases described in Proposition~\ref{prop:cases} do actually occur:
\begin{itemize}
\item examples for Case~I are Willmore tori with $\eta\equiv 0$ that are
  neither super conformal nor Euclidean minimal with planar ends for some
  point $\infty$ at infinity (see Corollary~\ref{cor:willmore}),
\item examples for Case~II are CMC tori in $\R^3$ (see
  Section~\ref{sec:harmonic}),
\item examples for Case~IIIa are super conformal tori (see
  Section~\ref{sec:nabla_mu}), and
\item examples for Case~IIIb are Euclidean minimal tori with planar ends for
  which the surfaces in the minimal surface associated family have
  translational periods (see~\cite{LPP05}).
\end{itemize}

\subsection{Non--trivial eigenvalues of $H^\mu$}\label{sec:non--trivial}
For fixed $\gamma\in \Gamma$, away from isolated spectral parameters $\mu\in
\C_*$ the eigenvalues of the holonomy $H^\mu(\gamma)$ are locally given by
holomorphic functions $\mu\mapsto \lambda(\mu)$: to see this note that 
\[\{ (\lambda,\mu)\in \C_*\times \C_* \mid f(\lambda,\mu)=0\} \qquad \textrm{
  with } \qquad f(\lambda,\mu)=\det(\lambda-H^\mu(\gamma))\] is a
1--dimensional analytic subset of $\C_*\times \C_*$ whose non--empty
intersection with $\C_*\times\{\mu\}$ for $\mu\in \C_*$ consists of up to 4
points. Denote by $X$ the Riemann surface normalizing this analytic set and by
$\mu\colon X \rightarrow \C_*$ its projection to the $\mu$--coordinate.  The
holomorphic function $\mu$ is then a branched covering whose number of sheets
is between $1$ and $4$ and coincides with the generic number of different
eigenvalues of $H^\mu(\lambda)$. Away from the branch points of $\mu$, the
different sheets of the analytic set are therefore locally graphs of
holomorphic functions $\mu\mapsto \lambda(\mu)$ describing the eigenvalues of
$H^\mu(\gamma)$.


If one of the local holomorphic functions $\mu\mapsto \lambda(\mu)$ that
describe the eigenvalues of the holonomy $H^\mu(\gamma)$ for some $\gamma\in
\Gamma$ is constant, that is, if there is $\lambda\in \C$ that is eigenvalue
of $H^\mu(\gamma)$ for all $\mu$ in an open subset of $\C_*$, then $\lambda$
is eigenvalue for all $\mu \in \C_*$ and the following lemma implies that
$\lambda\equiv 1$.

\begin{Lem} \label{lem:non-const} If $\lambda\in \C_*$ is an eigenvalue of
  $H^\mu(\gamma)$ for all $\mu$, then $\lambda=1$. The multiplicity of
  $\lambda=1$ as simultaneous eigenvalue of $H^\mu(\gamma)$ for all $\mu\in
  \C_*$ is even.
\end{Lem}
\begin{proof}
  The first statement follows from the fact that $H^{\mu}(\gamma)=\Id$ for
  $\mu=1$.  The second statement is a consequence of the quaternionic symmetry
  \eqref{eq:symmetry_nabla_mu} of $\nabla^\mu$ for $\mu \in S^1$. It implies
  that, for $\mu \in S^1$, the multiplicity of $1$ as an eigenvalue of
  $H^\mu_p(\gamma)$ is even. The same holds for all $\mu\in \C_*$, because the
  minimal kernel dimension of the holomorphic family $\mu\mapsto
  H^\mu_p(\gamma)-\Id$ of endomorphisms is generic and attained away from
  isolated points, see Proposition~\ref{pro:holomorphic_family} below.
\end{proof}

We denote by \emph{non--trivial eigenvalues} the eigenvalues that are not
equal to $1$ and therefore locally given by non--constant functions
$\mu\mapsto \lambda(\mu)$.  In the proof of Proposition~\ref{prop:cases} we
will see that non--trivial eigenvalues are generically simple, i.e., have
algebraic multiplicity~1.  The corresponding eigenlines are in the following
called \emph{non--trivial eigenlines}.

In the investigations of the present paper we will frequently apply this
proposition a proof of which can be found in \cite{BPP2} (see Proposition~3.1
there):

\begin{Pro}\label{pro:holomorphic_family}
  For a family of Fredholm operators that holomorphically depends on a
  parameter in a connected complex manifold $X$, the minimal kernel dimension
  is generic and attained away from an analytic subset $Y\subset X$.  In case
  $X$ is 1--dimensional, $Y$ is a set of isolated points and the holomorphic
  vector bundle defined by the kernels over $X\backslash Y$ holomorphically
  extends through the isolated points $Y$ with higher dimensional kernel. If
  the index of the operators is zero, the set of $x$ for which they are
  invertible is locally given as the vanishing locus of one holomorphic
  function.
\end{Pro}

\subsection{Proof of Proposition~\ref{prop:cases} in the non--trivial normal
  bundle case}\label{proof:non--trivial}
The proof in case of non--trivial normal bundle is analogous to that of
Lemma~3.1 in \cite{LPP05}.  We show that $1$ is an eigenvalue of
multiplicity~$4$ for all holonomies of $\nabla^\mu$.  Because the degree of
the quaternionic holomorphic line bundle $V^*/L^\perp\cong L^{-1}$ induced by
the dual constrained Willmore torus $f^\perp$ is $\deg(V^*/L^\perp) =
-\deg(V/L)$ and the holonomy representation of $(\nabla^\perp)^\mu$, by
\eqref{eq:gauge_trafo} and \eqref{eq:associated_perp}, is equivalent to the
dual representation of the holonomy of $\nabla^\mu$, we assume without loss of
generality (if necessary by passing to the dual surface $f^\perp$) that the
degree of the normal bundle and therefore of $V/L$ is negative.

If the multiplicity of $1$ as an eigenvalue was not $4$ for all holonomies, by
Lemma~\ref{lem:non-const}, there had to be a non--constant holomorphic map
$h\colon U\rightarrow \Hom(\Gamma,\C_*)$ defined on an open subset $U\subset
\C_*$ such that, for every $\mu\in U$, there is a non--trivial
$\nabla^\mu$--parallel section $\psi^\mu \in \Gamma(\tilde V)$ with
$H^\mu(\gamma)\psi^\mu=\psi^\mu h_\gamma^\mu$ for all $\gamma\in \Gamma$.
Projecting the $\psi^\mu$ to $V/L$ yields a family of holomorphic sections
with monodromy all of whose multipliers are different and which are therefore
linearly independent.  But the existence of such an infinite dimensional space
of holomorphic sections with monodromy of $V/L$ contradicts the
quaternionic Pl\"ucker formula with monodromy according to which
\[ \mathcal{W}(V/L) \geq -n \deg(V/L) \] in case there exists an
$n$--dimensional linear system with monodromy, see Appendix of~\cite{BLPP}.

The quaternionic symmetry \eqref{eq:symmetry_nabla_mu} of $\nabla^\mu$ for
$\mu\in S^1$ implies that in the Jordan--case all Jordan blocks are $2\times2$
(because they are $2\times2$ for $\mu\in S^1$ such that the holomorphic family
$H^\mu(\gamma)-\Id$ of endomorphisms squares to zero for all $\mu \in S^1$ and
hence everywhere). \qed

\subsection{The multiplier spectral curve}\label{sec:multiplier_spec} 
The proof of Proposition~\ref{prop:cases} in the trivial normal bundle case
requires ideas from quaternionic holomorphic geometry related to the spectral
curve of conformally immersed tori $f\colon T^2 \rightarrow S^4$ with trivial
normal bundle.  The spectral curve of $f$, in the following also called the
\emph{multiplier spectral curve} $\Sigma_{mult}$, is the Riemann surface
normalizing its \emph{spectrum}, the complex analytic set that consists of
all complex multipliers
\[h\in \Hom(\Gamma,\C_*) \cong \C_*\times \C_*\] for which there is a
non--trivial holomorphic section with monodromy $h$ of~$V/L$,
cf.~\cite{BLPP,BPP2}.  The idea of defining a spectral curve for conformal
immersions is due to Taimanov \cite{Ta98}, Grinevich, and Schmidt \cite{GS98}
who give a slightly different (but equivalent, cf.~\cite{Boh03}) definition of
the spectral curve for immersions $f\colon T^2\rightarrow \R^3$ which is based
on the Euclidean concept of Weierstrass representation.

In order to justify the definition of $\Sigma_{mult}$ one has to verify that
the possible multipliers form a 1--dimensional complex analytic set. In
\cite{BPP2} this is proven by asymptotic analysis of a holomorphic family of
elliptic operators. In addition it is shown that $\Sigma_{mult}$ has one or
two ends (depending on whether its genus is infinite or finite) and one or two
connected components each of which contains an end. Moreover, the minimal
vanishing order of the functions describing the spectrum is one at generic
points which implies that, away from isolated points $\sigma\in
\Sigma_{mult}$, the space of holomorphic sections with monodromy $h^\sigma$ of
$V/L$ is complex 1--dimensional.

Because the kernels of a holomorphic 1--parameter family of elliptic operators
form a holomorphic vector bundle which holomorphically extends through the
isolated points with higher dimensional kernel, see
Proposition~\ref{pro:holomorphic_family}, we obtain \cite{BPP2} a unique line
subbundle $\mathcal{L}$ of the trivial bundle $\Sigma_{mult}\times
\Gamma(\widetilde{V/L})$ equipped with the $C^\infty$--topology each fiber
$\mathcal{L}_\sigma$ of which is contained in (and generically coincides with)
the space of holomorphic sections with monodromy $h^\sigma$ of~$V/L$.  This
defines~\cite{BLPP} a map
\[ F\colon T^2\times \Sigma_{mult} \rightarrow \CP^3, \qquad (p,\sigma)
\mapsto \psi^\sigma(p)\C,
\] where $\psi^\sigma$ denotes the prolongation of a non--trivial element of
$\mathcal{L}_\sigma$.  For a fixed point $p\in T^2$ on the torus, the map
$\sigma \in \Sigma_{mult} \rightarrow F(p,\sigma)$ is holomorphic. For fixed
$\sigma\in \Sigma_{mult}$ in the spectral curve, the twistor projection of
$p\mapsto F(p,\sigma)$ to $\HP^1$ is a \emph{singular Darboux transform} of
$f$, that is, a Darboux transform defined away from the finitely many points
$p$ at which $\psi^\sigma$ is contained in $L$.

The set of possible multipliers is invariant under complex conjugation,
because multiplying a holomorphic section with monodromy
$h\in\Hom(\Gamma,\C_*)$ by the quaternion $j$ yields a holomorphic section
with monodromy $\bar h$. By lifting the map $h\mapsto \bar h$ to the
normalization $\Sigma_{mult}$ we obtain an anti--holomorphic involution
$\rho\colon\Sigma_{mult} \rightarrow \Sigma_{mult}$ which has no fixed points
because \[F(p,\rho(\sigma)) = F(p, \sigma)j.\]

Let $\varphi^\sigma$ be a nowhere vanishing local holomorphic section of the
bundle $\mathcal{L}\rightarrow \Sigma_{mult}$, i.e., $\varphi^\sigma$ is a
family of holomorphic sections in $\Gamma(\widetilde{V/L})$ which
holomorphically depends on $\sigma$ and satisfies
\[ \gamma^*\varphi^\sigma= \varphi^\sigma h^\sigma_\gamma \] for all
$\gamma\in \Gamma$.  Taking the derivative $\tfrac{\partial}{\partial x}$ with
respect to an arbitrary chart $x$ of $\Sigma_{mult}$ yields
\[ \gamma^*\tfrac{\partial \varphi^\sigma}{\partial x}= \tfrac{\partial
  \varphi^\sigma}{\partial x} h^\sigma_\gamma + \varphi^\sigma \tfrac{\partial
  h^\sigma_\gamma}{\partial x}\] such that $\varphi^{\sigma}$ and
$\tfrac{\partial \varphi^\sigma}{\partial x}$ span a 2--dimensional linear
system with Jordan monodromy of $V/L$.  The following lemma shows that this
2--dimensional linear system is generically the only linear system with Jordan
monodromy and eigenvalue $h^\sigma$ of $V/L$ (like $\varphi^\sigma$
generically spans the space of holomorphic sections with monodromy $h^\sigma$
of $V/L$).

\begin{Lem}\label{lem:linear_sys_with_monodromy}
  For a generic point $\sigma\in \Sigma_{mult}$ in the multiplier spectral
  curve of an immersed torus $f\colon T^2\rightarrow S^4$ with trivial normal
  bundle, there is a unique 2--dimensional linear systems with Jordan
  monodromy and eigenvalues $h^\sigma$ of $V/L$.
\end{Lem}

\begin{proof}
  Firstly, we prove the existence of a point $\sigma_0\in \Sigma_{mult}$ that
  admits only one 2--dimensional linear system with Jordan monodromy and
  eigenvalue $h^{\sigma_0}$. Secondly, we show how, using
  Proposition~\ref{pro:holomorphic_family}, this implies that generic points
  $\sigma\in \Sigma_{mult}$ admit a unique 2--dimensional linear system with
  Jordan monodromy and eigenvalue $h^\sigma$.

  Let $\sigma_0\in\Sigma_{mult}$ be a point that normalizes a regular point of
  the spectrum, the analytic set of possible multipliers of holomorphic
  sections with monodromy. The spectrum is then locally given as the vanishing
  locus of a holomorphic function that vanishes to first order at
  $h^{\sigma_0}$.  Because the minimal vanishing order of holomorphic
  functions describing the spectrum is greater or equal to the dimension of
  the space of holomorphic sections with monodromy, the space of holomorphic
  sections with monodromy $h^{\sigma_0}$ is 1--dimensional.  By Lemma~4.9 of
  \cite{BPP2} we can assume that $\sigma_0$ is chosen such that non--trivial
  holomorphic sections with monodromy $h^{\sigma_0}$ have no zeros.  Denote by
  $\nabla$ the quaternionic connection of $V/L$ rendering this space of
  holomorphic sections parallel.  Then $d^\nabla$ makes $KV/L$ into a
  quaternionic holomorphic line bundle of degree~$0$. Because $\nabla$ is
  flat, it maps holomorphic sections with monodromy of $V/L$ to holomorphic
  sections with monodromy of $KV/L$ such that the spectrum of $V/L$ is
  included in the spectrum of $KV/L$.  This shows that both spectra coincide,
  because the spectrum of a quaternionic holomorphic line bundle of degree~0
  over a torus is a 1--dimensional analytic set that is either irreducible or
  has two irreducible components interchanged under complex conjugation, see
  \cite{BPP2}.  In particular, there is a local holomorphic function
  describing the spectrum of $KV/L$ that vanishes to first order at
  $h^{\sigma_0}$ such that the space of holomorphic sections with monodromy
  $h^{\sigma_0}$ of $KV/L$ is also 1--dimensional.  This implies the
  uniqueness of the 2--dimensional linear system with Jordan monodromy
  $h^{\sigma_0}$ of $V/L$: let $\varphi_1$, $\varphi_2$ be holomorphic
  sections of $\widetilde{V/L}$ and $t \in \Hom(\Gamma,\C)$ such that
  \[\gamma^*\varphi_1=\varphi_1
  h^{\sigma_0}_\gamma\qquad \textrm{ and } \qquad\gamma^*\varphi_2=\varphi_2
  h^{\sigma_0}_\gamma+ \varphi_1 t_\gamma h^{\sigma_0}_\gamma\] for all
  $\gamma\in \Gamma$.  Then $\varphi_1$ is unique up to scaling, because it is
  a holomorphic section with monodromy $h^{\sigma_0}$, and $\varphi_2$ is
  unique up to scaling and adding a multiple of $\varphi_1$, because
  $\nabla\varphi_2$ is a holomorphic section with monodromy $h^{\sigma_0}$ of
  $KV/L$.

  Similar as in Section~2.3 of \cite{BPP2} one can check that 2--dimensional
  linear systems with Jordan monodromy correspond to solutions to
  \[ D_\omega \varphi_1 =0 \qquad \textrm{ and }\qquad D_\omega \varphi_2 +
  (\varphi_1 \eta)'' =0, \tag{$*$}\] where $\omega$, $\eta\in
  \Hom(\Gamma,\C)\cong \Harm(\C/\Gamma,\C)$ and $D_{\omega}\varphi_1 =
  D\varphi_1 + (\varphi_1 \omega)''$: given a solution $\varphi_1$,
  $\varphi_2\in \Gamma(V/L)$ to $(*)$,
  \[ (\tilde \varphi_1, \tilde \varphi_2) = (\varphi_1,\varphi_2) e^{\int
    \omega}
  \begin{pmatrix} 1 & \int\eta \\ 0 & 1 \end{pmatrix}\] is a 2--dimensional
  linear system with Jordan monodromy and eigenvalue $h$ of $V/L$, where
  $h_\gamma = e^{\int_\gamma\omega}$. Clearly, non--trivial solutions to $(*)$
  can only exists if $h$ with $h_\gamma = e^{\int_\gamma\omega}$ belongs to
  the spectrum. Denote by $\tilde\Sigma$ the ``logarithmic spectral curve'',
  the normalization of the space of $\omega$ for which $h$ belongs to the
  spectrum. We consider now the holomorphic family
  \[ D_{\omega,\eta} \begin{pmatrix} \varphi_1\\\varphi_2
  \end{pmatrix}=
  \begin{pmatrix}
    D_\omega \varphi_1 \\ D_\omega \varphi_2 + (\varphi_1 \eta)''
  \end{pmatrix}
  \]
  of elliptic operators parametrized over $\tilde\Sigma\times
  (\Harm(\C/\Gamma,\C)\backslash\{0\})\cong \tilde \Sigma\times
  (\C^2\backslash\{0\})$. The fact that for every $h$ in the spectrum there is
  at least one 2--dimensional linear system with Jordan monodromy and
  eigenvalue $h$ implies that for every $\omega\in \tilde \Sigma$ there is at
  least a line worth of $\eta\in (\Harm(\C/\Gamma,\C)\backslash\{0\})$ for
  which $D_{\omega,\eta}$ has a non--trivial kernel.  On the other hand, we
  have proven above the existence of a multiplier $h$ that admits a unique
  linear system with Jordan monodromy such that there exists $\omega\in \tilde
  \Sigma$ admitting a unique (up to scale) $\eta$ with
  $\ker(D_{\omega,\eta})\neq \{0\}$.  Because $D_{\omega,\eta}$ has index~$0$,
  Proposition~\ref{pro:holomorphic_family} implies that the set of
  $\omega,\eta$ for which $D_{\omega,\eta}$ has a non--trivial kernel is a
  non--empty 2--dimensional analytic set which projects to a 1--dimensional
  analytic subset of $\tilde \Sigma\times \CP^1$.  To complete the proof we
  have to show that, apart from components of the form $\{\omega\}\times
  \CP^1$, the normalization of this 1--dimensional analytic set is a graph
  over $\tilde \Sigma$. Assume this was not the case.  Then, the normalization
  has one component $X$ that is neither of the form $\{\omega\}\times \CP^1$
  nor a part of the graph over $\tilde \Sigma$ that corresponds to the
  ``generic'' 2--dimensional linear systems with Jordan monodromy described
  above (before the statement of the lemma).

  The projection to $\tilde \Sigma$ would map this additional component $X$
  onto a connected component of $\tilde \Sigma$ (because, for every Riemann
  surface $Y$, the image of a connected component of the normalization of a
  1--dimensional analytic subset of $Y \times \CP^1$ under the projection to
  $Y$ is either a point or a connected component of $Y$). But this is
  impossible, because every connected component of $\tilde \Sigma$ contains a
  regular point at which the corresponding holomorphic sections with monodromy
  of $V/L$ are nowhere vanishing such that, as seen above, there is a unique
  2--dimensional linear system with Jordan monodromy belonging to the
  respective eigenvalue.
\end{proof}

\subsection{Proof of Proposition~\ref{prop:cases} in the trivial normal bundle
  case}\label{proof:trivial}
We fix $\gamma\in \Gamma$ for which $H^\mu(\gamma)$ generically has the
maximal number of different eigenvalues such that the branched covering
$\mu\colon X \rightarrow \C_*$ has the maximal number of sheets, where $X$ as
in Section~\ref{sec:non--trivial} denotes the Riemann surface normalizing the
1--dimensional analytic subset of $\C_*\times \C_*$ given by
$f(\lambda,\mu)=0$ with $f(\lambda,\mu)=\det(\lambda-H^\mu(\gamma))$.  Recall
that the number of sheets is 4 in case all eigenvalues of $H^\mu(\gamma)$ are
distinct and 1, 2,~or~3 if the discriminant of the characteristic polynomial
of $H^\mu_p(\tilde \gamma)$ vanishes identically for all $\tilde \gamma\in
\Gamma$.

In case the number of sheets is $4$ we are in Case~I of the above list: away
from isolated parameters $\mu$ the holonomy $H^\mu_p(\gamma)$ has then 4
different eigenvalues which, by Lemma~\ref{lem:non-const}, are non--constant
as functions of $\mu$.

If the number of sheets is 1 we are in Case~III of the above list: for
every $\mu \in \C_*$ the eigenvalue of the $\SL(4,\C)$--holonomy is then a
fourth root of unity and hence equal to $1$, because $\nabla^{\mu=1}$ is
trivial. As in the non--trivial normal bundle case, the statement about the
Jordan holonomy is an immediate consequence of the quaternionic
symmetry~\eqref{eq:symmetry_nabla_mu} of $\nabla^\mu$ for $\mu\in S^1$.

If the number of sheets is 3 we are in Case~II of the above list: away from a
discrete set of points, the dimension of the generalized eigenspaces
$\ker((\lambda -H^\mu(\gamma))^2)$ is constant on connected components of
$\Sigma$, cf.\ Proposition~\ref{pro:holomorphic_family}. The Riemann surface
$\Sigma$ is thus the disconnected sum of one sheet that corresponds to a
double eigenvalue of $H^\mu(\gamma)$ and a hyper--elliptic surface that
parametrizes its simple eigenvalues.  The quaternionic
symmetry~\eqref{eq:symmetry_nabla_mu} implies that for generic $\mu \in \S^1$
the holonomy $H^\mu(\gamma)$ has 2 simple eigenvalues which are complex
conjugate and one real eigenvalue of geometric multiplicity~2. The
corresponding eigenspaces are invariant under all holonomies $H^\mu(\tilde
\gamma)$, $\tilde \gamma\in \Gamma$. In particular, because there is no
holonomy with 4 different eigenvalues, for all $\tilde\gamma\in \Gamma$ the
restriction of $H^\mu(\tilde \gamma)$ to the 2--dimensional eigenspace of
$H^\mu(\gamma)$ is a multiple of identity.  As explained in
Section~\ref{sec:multiplier_spec}, there is only a discrete set of complex
multipliers $h\in \Hom(\Gamma,\C_*)$ for which the space of holomorphic
sections with monodromy $h$ of $V/L$ has dimension greater or equal~2.
Because $\nabla^\mu$--parallel sections that correspond to simultaneous
eigenlines of the holonomy project to holomorphic sections with monodromy of
$V/L$, see Lemma~\ref{lem:nabla_mu_parallel_prolonge_hol}, we obtain that the
double eigenvalues of $H^\mu(\tilde \gamma)$, $\tilde \gamma\in \Gamma$ are
locally constant as function of $\mu$ and hence, by Lemma~\ref{lem:non-const},
equal to $1$. The same lemma shows that the simple eigenvalues are
non--constant as functions of $\mu$ such that we are in Case~II of the list.

To complete the proof of the proposition it remains to show that the branched
covering $\mu\colon \Sigma\rightarrow \C_*$ cannot be 2--sheeted.  It is
impossible that the Riemann surface $\Sigma$ is the disconnected sum of two
sheets that correspond to a simple and a triple eigenvalue, respectively,
because this would contradict the quaternionic
symmetry~\eqref{eq:symmetry_nabla_mu} for $\mu \in S^1$. Thus, if $\Sigma$ is
a 2--sheeted branched covering, by Proposition~\ref{pro:holomorphic_family}
the generalized eigenspaces of the holonomies $H^\mu(\gamma)$ define a rank~2
bundle over $\Sigma$.

In case all holonomies $H^\mu(\tilde \gamma)$, $\tilde \gamma\in \tilde
\Gamma$ are diagonizable, every vector in this rank~2 bundle is an eigenvector
of $H^\mu(\tilde \gamma)$ for all $\tilde \gamma\in \tilde \Gamma$, because
otherwise there would be $\tilde \gamma\in \tilde \Gamma$ for which
$H^\mu(\tilde \gamma)$ has four distinct eigenvalues such that we were in
Case~I. Hence, every fiber of this rank~2 bundle gives rise to a
2--dimensional space of holomorphic section with monodromy of $V/L$. But this
is impossible, because the eigenvalues of $H^\mu(\gamma)$ are non--constant as
function of $\mu$ while higher dimensional spaces of holomorphic section with
monodromy $h$ of $V/L$ can only exists for isolated $h\in \Hom(\Gamma,\C_*)$,
see Section~\ref{sec:multiplier_spec}.

We can therefore chose $\gamma\in \Gamma$ for which $H^\mu(\gamma)$
generically has two double eigenvalues with geometric multiplicity~1 and
algebraic multiplicity~2. All holonomies $H^\mu(\tilde \gamma)$,
$\tilde\gamma\in \Gamma$ then leave the rank~2 bundle defined by the
generalized eigenspaces of $H^\mu(\gamma)$ invariant and all their
restrictions to this rank~2 bundle have a double eigenvalue.  In other words,
taking projections to $V/L$ of the $\nabla^\mu$--parallel sections which
correspond to the rank~2 bundle, a generic point $\sigma\in \Sigma$ gives rise
to a 2--dimensional linear system with Jordan monodromy and eigenvalue
$h^\sigma$ of $V/L$.

By Lemma~\ref{lem:linear_sys_with_monodromy}, for generic $\sigma$ such linear
system is unique such that if $\psi^\mu$ denotes a holomorphic family of
$\nabla^\mu$--parallel sections with monodromy of $V$, the sections
$\frac{\partial \psi^\mu}{\partial \mu}$ are also $\nabla^\mu$--parallel
(because generically the projections to $V/L$ of $\psi^\mu$ and
$\frac{\partial \psi^\mu}{\partial \mu}$ span the unique linear system with
Jordan monodromy belonging to the corresponding multiplier, see the discussion
before the statement of Lemma~\ref{lem:linear_sys_with_monodromy}). Taking the
derivative of $\nabla^\mu \psi^\mu=0$ with respect to $\mu$ then yields
$A^{(1,0)}_\circ\psi^\mu- \mu^{-2} A^{(0,1)}_\circ\psi^\mu=0$ such that
$A_\circ\psi^\mu=0$. Hence, all $\psi^\mu$ are constant sections of $V$ and
all holonomies are trivial such that the number of sheets of $\mu\colon
\Sigma\rightarrow \C_*$ is 1.  This completes the proof, because it shows that
the number of sheets of $\mu\colon \Sigma\rightarrow \C_*$ can never be~2.
\qed

\subsection{The holonomy spectral curve}
We show that a constrained Willmore torus that belongs to Case~I or II of
Proposition~\ref{prop:cases} gives rise to a Riemann surface parametrizing the
non--trivial eigenlines of the holonomies $H^\mu_p(\gamma)$, $\gamma\in
\Gamma$.

\begin{Lem}\label{lem:holonomy_curve}
  Let $A(\mu)$ be a family of complex $n\times n$--matrices holomorphically
  depending on $\mu\in U$ in a connected open subset $U\subset \C$ with the
  property that eigenvalues which are non--constant as functions of $\mu$ are
  generically simple, i.e., have algebraic multiplicity~1. Then there exists a
  unique Riemann surface $\Sigma$ with holomorphic maps $\mu\colon \Sigma
  \rightarrow U$ and $\mathcal{E}\colon \Sigma \rightarrow \CP^{n-1}$ such
  that
  \begin{itemize}
  \item[a)] for every $\sigma\in \Sigma$, the line $\mathcal{E}(\sigma)$ is an
    eigenline of $A(\mu(\sigma))$ and
  \item[b)] for generic $\mu_1\in U$, the set $\mathcal{E}(\mu^{-1}\{\mu_1\})$
    is the set of eigenlines of $A(\mu_1)$ that belong to non--constant
    eigenvalues.
  \end{itemize}
\end{Lem}

\begin{proof}
  Denote by $g(\lambda,\mu)= \det(\lambda - A(\mu))$ the characteristic
  polynomial of $A(\mu)$. After splitting of linear factors $(\lambda
  -\lambda_j)$ belonging to eigenvalues that are constant in $\mu$ we obtain a
  function $\tilde g(\lambda,\mu)$ whose discriminant, as a polynomial in
  $\lambda$, does not vanish identically.  We define $\Sigma$ as the Riemann
  surface normalizing the 1--dimensional analytic set \[ \{ (\lambda,\mu) \in
  \C \times U \mid \tilde g(\lambda,\mu)=0 \}.\] The projection $\mu\colon
  \Sigma\rightarrow U$ is then a branched covering and $\Sigma$ is the unique
  Riemann surface equipped with holomorphic functions $\mu$ and $\lambda$ such
  that, for generic $\mu_1\in U$, the non--constant eigenvalues of $A(\mu_1)$
  are given by $\lambda(\sigma_j)$ with $\sigma_1$, ..., $\sigma_{\tilde n}
  \in \mu^{-1}\{\mu_1\}$.

  The family of matrices $\lambda(\sigma)\Id -A(\mu(\sigma))$ depends
  holomorphically on $\sigma\in \Sigma$ and generically, away from a set of
  isolated points, has a 1--dimensional kernel which is an eigenline of
  $A(\mu(\sigma))$. Because $\Sigma$ is complex 1--dimensional, the line
  bundle $\mathcal{E}(\sigma) = \ker(\lambda(\sigma)-A(\mu(\sigma)))$ defined
  over generic points extends holomorphically through the isolated points with
  a higher dimensional kernel, cf.~Proposition~\ref{pro:holomorphic_family}.

  By construction, $\Sigma$ with $\mu$ and $\mathcal{E}$ satisfies a) and b) in
  the statement of the lemma. The uniqueness of $\Sigma$ with $\mu$ and
  $\mathcal{E}$ follows from the above uniqueness property of $\Sigma$ with
  $\mu$ and $\lambda$, because the eigenline map $\mathcal{E}$ of $A(\mu)$
  allows to recover the holomorphic function $\lambda$ describing the
  eigenvalues.
\end{proof}

For constrained Willmore tori belonging to Case~I or II of
Proposition~\ref{prop:cases}, the preceding lemma allows to define a Riemann
surface, in the following called the \emph{holonomy spectral curve}
$\Sigma_{hol}$, that parametrizes the non--trivial eigenlines of the
holonomies $H^\mu(\gamma)$: we define $\Sigma_{hol}$ as the 2 or 4--sheeted
branched covering of $\C_*$ obtained from Lemma~\ref{lem:holonomy_curve}
applied to $A(\mu) = H^\mu_p(\gamma)$ with fixed $p\in T^2$ and $\gamma \in
\Gamma$ for which $H^\mu_p(\gamma)$ has the maximal number of non--trivial
eigenvalues. Because the holonomies for different $\gamma\in \Gamma$ commute,
the uniqueness part of Lemma~\ref{lem:holonomy_curve} shows that
$\Sigma_{hol}$ is independent of the choice of $\gamma$.  Moreover, the
Riemann surface $\Sigma_{hol}$ does not depend on the choice of $p$, because
changing the point $p$ on the torus amounts to conjugate the holonomy (but the
Riemann surface $\Sigma$ in the proof of Lemma~\ref{lem:holonomy_curve} is
defined purely in terms of eigenvalues of $A(\mu)$).  What does depend on the
point $p\in T^2$ of the torus is the eigenline curve $\mathcal{E}_p\colon
\Sigma_{hol}\rightarrow \CP^3$.

For every point $\sigma\in \Sigma_{hol}$, the line $\mathcal{E}_p(\sigma)$ is
invariant under the holonomy representation $\gamma\in \Gamma\mapsto
H^{\mu(\sigma)}_p(\gamma)$. This defines a holomorphic map $h\colon
\Sigma_{hol}\rightarrow \Hom(\Gamma,\C_*)$ whose image is contained in the set
of multipliers of holomorphic sections with monodromy of $V/L$.  This map
lifts to a holomorphic map $\iota \colon \Sigma_{hol} \rightarrow
\Sigma_{mult}$ which turns out to be injective and almost surjective, see
Theorem~\ref{the:biholomorphicity}.  By definition, the eigenline curve
$\mathcal{E}$ is related to the map $F$ defined in
Section~\ref{sec:multiplier_spec} by
$\mathcal{E}_p(\sigma)=F(\iota(\sigma),p)$ for all $\sigma\in \Sigma_{hol}$.
The map $F$ can therefore be seen as a generalization of the holonomy
eigenline curve $\mathcal{E}$ of the constrained Willmore associated family
$\nabla^\mu$ to arbitrary conformal immersions $f\colon T^2\rightarrow S^4$
with trivial normal bundle.

The map $\iota \colon \Sigma_{hol} \rightarrow \Sigma_{mult}$ interchanges the
fixed point free, anti--holomorphic involutions on $\Sigma_{hol}$ and
$\Sigma_{mult}$: under $\iota$, the involution $\rho$ on $\Sigma_{hol}$
induced by the symmetry \eqref{eq:symmetry_nabla_mu} via
\[ \mathcal{E}_{\rho(\sigma)} = \mathcal{E}_\sigma j \] corresponds to the
involution $\rho$ on $\Sigma_{mult}$. This fixed points free involution $\rho$
on $\Sigma_{hol}$ covers the involution $\mu\mapsto 1/\bar \mu$ of the
$\mu$--plane.

\section{The Asymptotics of $\nabla^\mu$--parallel Sections}
\label{sec:asymptotics}

The proof of the main theorem in Section~\ref{sec:proof} requires some control
over the asymptotic behavior of $\nabla^\mu$--parallel sections for
$\mu\rightarrow 0$ or $ \infty$. This is provided by
Proposition~\ref{pro:asymptotic} of the present section.  As an immediate
application of Proposition~\ref{pro:asymptotic} we show that the holonomy
spectral curve essentially coincides with the multiplier spectral curve in
case they are both defined, that is, for constrained Willmore tori belonging
to Cases~I and II of Proposition~\ref{prop:cases}.

\subsection{Main result of the section}
Because of the symmetry \eqref{eq:symmetry_nabla_mu} it is sufficient to
understand the asymptotic behavior of parallel sections for
$\mu\rightarrow\infty$.  We approach this problem by investigating the
sections $\psi \in \Gamma(\tilde V)$ of the pullback $\tilde V$ of $V$ to the
universal covering $\C$ of $T^2=\C/\Gamma$ that satisfy
\begin{equation}
  \label{eq:asymptotic_DT}
  \nabla\psi\in \Omega^1(\tilde L) \qquad \textrm{ and } \qquad (A_\circ
  \psi)^{(1,0)}=0,
\end{equation}
where $\tilde L$ denotes the pullback of $L$ to the universal covering.
Equation~\eqref{eq:asymptotic_DT} is an asymptotic version of
$\nabla^\mu\psi^\mu=0$ for $\mu\rightarrow\infty$; examples of solutions to
\eqref{eq:asymptotic_DT} are the prolongations of holomorphic sections that
are suitable limits of $\nabla^\mu$--parallel sections for $\mu \rightarrow
\infty$.

The following proposition summarizes Lemmas~\ref{lem:fundamental_nonconstant}
and \ref{lem:fundamental_constant} below.

\begin{Pro}\label{pro:asymptotic}
  Let $f\colon T^2 \rightarrow S^4$ be a constrained Willmore torus. If $f$ is
  neither super conformal nor Euclidean minimal with planar ends and
  $\eta\equiv 0$, then the complex dimension of the space of solutions to
  \eqref{eq:asymptotic_DT} is at most two.
\end{Pro}

For the rest of the section we assume that $f$ is not super conformal and, in
particular, $A_\circ$ does not vanish identically.  Because $A_\circ$ takes
values in $L$, its rank is then at most one.  Away from its isolated zeros the
rank of $A_\circ$ is one and the kernel bundle $\check L=\ker(A_\circ)$
smoothly extends through the isolated zeroes to a line bundle over $T^2$.  To
see this note that, because $\delta\eta$ is a holomorphic quadratic
differential, either $\eta\not\equiv 0$ such that $\eta$ and hence $A_\circ$
have no zeroes at all, or $\eta\equiv 0$ such that $A_\circ = A$ itself is
holomorphic, see Proposition~22 of \cite{BFLPP02}.

Denote by $U$ the open set of points where $A_{|L}$ does not vanish or,
equivalently, where $V=L\oplus\check L$.  The set $U$ is non--empty:
otherwise, by Lemma~22 of \cite{BFLPP02}, the Hopf field $A$ had to vanish
identically which is impossible because then $\eta\equiv 0$ (see the
discussion following \eqref{eq:decomposed_el_eq}) such that $A_\circ\equiv 0$.
If $\check L=\ker(A_\circ)$ is constant the subset $U$ is dense, because it is
the complement of the set of points where the immersion $f$ goes through
$\check L$.

Let $\psi$ be a solution to \eqref{eq:asymptotic_DT} defined on $U$.  Using
$d^\nabla A_\circ = 2 A_\circ\wedge A_\circ$, differentiation of $*A_\circ
\psi = -A_\circ \psi i$ yields
\[ -* A_\circ \wedge \nabla\psi = -2 A_\circ \wedge A_\circ \psi i + A_\circ
\wedge \nabla\psi i. \] Because $\nabla\psi$ takes values in $L$ and
$(A_\circ)_{|L} = A_{|L}$ this implies
\[ A_\circ\wedge (-S\nabla\psi + \nabla\psi i - 2 A_\circ \psi i)=0. \] On the
open set $U$ where $A_{|L}$ has no zeros, every form $\alpha\in \Gamma(KL)$
with $A_\circ \wedge \alpha =0$ has to vanish identically.  Since $\nabla\psi
\in \Gamma(KL)$, on $U$ every solution $\psi$ to \eqref{eq:asymptotic_DT}
satisfies
\begin{equation}
  \label{eq:zero_one_part}
  (\nabla\psi)^{(0,1)} = A_\circ \psi. 
\end{equation}

\subsection{The case that   \texorpdfstring{$A_\circ \not \equiv 0 $ and $\check
    L=\ker(A_\circ)$}{$ker(A_0)$}  is non--constant}
If $\check L=\ker(A_\circ)$ is non--constant, the corresponding map into the
4--sphere is called a 2--step \emph{B\"acklund transformation} of $L$ (see
e.g.~\cite{Boh03} for a detailed discussion of B\"acklund transformations).

\begin{Lem}\label{lem:2step}
  Let $f\colon M\rightarrow S^4$ be a constrained Willmore immersion with
  $A_\circ\not \equiv 0$ for which $\check L = \ker(A_\circ)$ is
  non--constant.  The corresponding map $\check f$ into $S^4$ is then conformal
  and, on the open set $U$ where $V=L\oplus \check L$, it is constrained
  Willmore and admits a 1--form $\check \eta$ with $\im(\check \eta)\subset
  \check L \subset \ker(\check\eta)$ such that the form $2{*}\check
  Q_\circ=2{*} \check Q + \check \eta$ is closed and satisfies $2{*}\check
  Q_\circ=2{*}A_\circ$.
\end{Lem}

The last formula shows $L= \im(\check Q_\circ)$ which, in the language of
\cite{Boh03}, implies that $L$ is a 2--step backward B\"acklund transformation
of $\check L$.

\begin{proof}
  For every $\varphi\in \Gamma(\check L)$ we have $*A_\circ\varphi=0$ and
  therefore
  \[ 0=d^\nabla({*}A_\circ \varphi)=\underbrace{d^\nabla({*}A_\circ)}_{=0}
  \varphi- {*}A_\circ \wedge \nabla\varphi.\] Hence $\check f$ is a (possibly
  branched) conformal immersion, because $A_\circ \wedge \check \delta = 0$
  with $\check \delta = \pi_{V/\check L} \nabla_{|\check L}$ denoting the
  derivative of $\check f$. With respect to the splitting $L\oplus \check L$
  on $U$, the connection $\nabla$ and the mean curvature sphere $S$ of $L$ can
  be written as
\begin{equation}
  \label{eq:split_connection_mcs_L}
  \nabla=
  \begin{pmatrix}
    \nabla^L & \check \delta \\ \delta & \check \nabla
  \end{pmatrix} \qquad \textrm{ and } \qquad 
  S= 
  \begin{pmatrix}
    J & B \\ 0 & \tilde J
  \end{pmatrix},
\end{equation}
where $J$ and $\tilde J$ are complex structures on $L$ and $\check L$ with
$*\delta = \delta J = \tilde J \delta$ and where $B\in \Gamma(\Hom(\tilde
L,L))$ with $JB+B\tilde J=0$. The derivative of $S$ is
\begin{equation*}
  \nabla S = 
  \begin{pmatrix}
    \nabla^L J -B\delta & \nabla B + \check \delta\tilde J  - J \check \delta \\
    0 & \check \nabla \tilde J + \delta B
  \end{pmatrix}.
\end{equation*}
The mean curvature sphere condition $Q_{|L}=0$ now becomes that $\nabla^L J
-B\delta$ is left $K$ (and right $\bar K$) with respect to $J$. Moreover,
because $(A_\circ)_{|\check L}=0$ and $(Q_\circ)_{|L}=0$, the identity $\nabla
S = 2{*} Q_\circ-2{*}A_\circ$ implies
\begin{equation}
  \label{eq:A_Q_dagger}
  2{*}A_\circ=
  \begin{pmatrix}
    -\nabla^L J + B\delta & 0\\ 
    0 & 0
  \end{pmatrix}
  \qquad \textrm{ and } \qquad
   2{*}Q_\circ= 
  \begin{pmatrix}
    0 & \nabla B + \check \delta\tilde J  - J \check \delta \\
    0 & \check \nabla \tilde J + \delta B
  \end{pmatrix}.
\end{equation}
From $A_\circ \wedge \check \delta = 0$ we obtain $*\check \delta = - J \check
\delta$ (because $\nabla^L J-B\delta$ is right--$\bar K$) and the mean
curvature sphere of $\check L$ is
\begin{equation}
  \label{eq:mcs_tilde}
  \check S = 
  \begin{pmatrix}
  -J & 0 \\ \check B & -\check J
  \end{pmatrix},
\end{equation}
where $\check J$ is the complex structure on $\check L$ with $*\check \delta =
-\check \delta \check J$ and where $\check B\in \Gamma(\Hom(L,\check L))$ with
$\check J \check B+\check B J=0$.  Now
\begin{equation*}
   \nabla \check S = 
  \begin{pmatrix}
    -\nabla^L J + \check \delta \check B & 0 \\
    \nabla \check B + \check J \delta - \delta J & -\check \nabla \check J -
    \check B \check \delta
  \end{pmatrix}
\end{equation*}
and the condition $\im(\check A)\subset \check L$ that $\check S$ is the mean
curvature sphere of $\check L$ becomes that $-\nabla^L J + \check \delta
\check B$ is left $K$ (and right $\bar K$) with respect to $J$.

The Hopf field $\check Q$ of $\check L$ is given by
\begin{equation}
  \label{eq:tilde_Q}
  2{*}\check Q = \frac12(\nabla\check S + \check S {*} \nabla \check S) =   
  \begin{pmatrix}
    -\nabla^L J + \check \delta \check B & 0 \\
     * & 0 
  \end{pmatrix}.
\end{equation}

The condition that $S$ is mean curvature sphere of $L$ is equivalent to
$(\nabla^L J)'' = B\delta$ with $()''$ denoting the $\bar K$--part with
respect to $J$. Similarly, that $\check S$ is mean curvature sphere of $\check
L$ is equivalent to $(\nabla^L J)'' = \check \delta\check B$.  This implies
\begin{equation}
  \label{eq:B_and_tilde_B}
  B \delta = \check \delta \check B.
\end{equation}
By \eqref{eq:A_Q_dagger}, \eqref{eq:tilde_Q} and \eqref{eq:B_and_tilde_B}, the
1--form $\check \eta:= 2{*}A_\circ -2{*}\check Q $ satisfies $\im(\check \eta)
\subset \check L \subset \ker(\check \eta)$ and $2{*}\check Q_\circ=
2{*}\check Q+\check \eta$ is closed because
\begin{equation}\label{eq:bt}
  2{*}\check Q_\circ =   2{*}A_\circ.
\end{equation} 
This shows that, on $U$, $\check L$ is constrained Willmore.
\end{proof}

\begin{Lem} \label{lem:fundamental_nonconstant} Let $f\colon T^2 \rightarrow
  S^4$ be a constrained Willmore immersion with $A\not \equiv 0$, $Q\not
  \equiv 0$ for which $\check L = \ker(A_\circ)$ is non--constant.  Then, the
  space of solutions to \eqref{eq:asymptotic_DT} defined on the universal
  covering of $T^2$ is at most (complex) 2--dimensional.  In case $f$ is
  Willmore with $\eta\equiv 0$ but neither super conformal nor Euclidean
  minimal with planar ends, the space of solutions to \eqref{eq:asymptotic_DT}
  is 0--dimensional if $AQ\equiv 0$ and 1--dimensional if $AQ\not\equiv 0$.
\end{Lem}
Examples of Willmore immersions with $AQ\equiv 0$ are Willmore surfaces
contained in a 3--sphere $S^3$ and minimal surfaces in the metrical 4--sphere
or 4--dimensional hyperbolic space, see Chapter~10 of \cite{BFLPP02}. In case
of Willmore surfaces with $AQ\not\equiv 0$, the unique solution to
\eqref{eq:asymptotic_DT} corresponds to a 4--step Willmore B\"acklund
transformation.

\begin{proof} 
  We first prove that, on the open set $U$ where $L\oplus \check L$, a local
  solution $\psi$ to \eqref{eq:asymptotic_DT} satisfies $\check S\psi= -\psi
  i$ where $\check S$ is the mean curvature sphere of $\check L$, see
  \eqref{eq:mcs_tilde}.  Writing $\psi = (\psi_1,\psi_2)$ with respect to
  $L\oplus \check L$, the equation $*A_\circ \psi = -A_\circ \psi i$ implies
  $J\psi_1 = \psi_1 i$. By \eqref{eq:mcs_tilde} we thus have $\check S\psi +
  \psi i \in \Gamma(\check L)$.  Using \eqref{eq:bt} we obtain
  \begin{multline}\label{eq:fund}
    \nabla(\check S\psi + \psi i) = (2{*}\check Q_\circ - 2{*}\check
    A_\circ)\psi +\check S \nabla\psi + \nabla\psi i \\
    = (2{*}A_\circ - 2{*}\check A_\circ)\psi +\check S \nabla\psi +
    \nabla\psi i = - 2{*}\check A_\circ\psi + \check B \nabla\psi \in
    \Omega^1(\check L),
  \end{multline}
  where the last identity holds because, by \eqref{eq:zero_one_part}, we have
  $2*A_\circ\psi-*\nabla\psi+\nabla \psi i=0$.  Because $\check L$ is
  non--constant, this implies $\check S\psi + \psi i=0$ or, equivalently,
  $\check S \psi = -\psi i$.

  If $f$ is Willmore with $\eta\equiv0$, equation \eqref{eq:A_Q_dagger}
  implies $B=0$ because $S$ anti--commutes with $A=A_\circ$ and therefore, by
  \eqref{eq:B_and_tilde_B}, $\check B=0$. Plugging this and $\check S\psi +
  \psi i=0$ into \eqref{eq:fund} yields $\check A \psi=0$, that is, $\psi$ is
  a section of the complex line bundle $\{ v\in \ker(\check A) \mid \check S v
  = -v i \}$.  The space of solutions to \eqref{eq:asymptotic_DT} is thus at
  most 1--dimensional, because for any two solutions $\psi$, $\varphi$ there
  is a non--empty open set $U' \subset U$ and a complex function $g$ on $U'$
  such that $\psi = \varphi g$.  Hence, $\nabla \psi=\nabla\varphi g+ \varphi
  dg$ and, taking the projection $\pi$ to $V/L$, we have $(\pi\varphi) dg=0$
  which shows that $g$ is constant (because the holomorphic section
  $\pi\varphi$ of $V/L$ vanishes at isolated points only) and $\psi$ and
  $\varphi$ are linearly dependent on $T^2$ because they are linearly
  dependent on the open subset $U'\subset U$.

  If $\eta\equiv 0$ and $AQ\equiv0$, then every section $\psi$ solving
  \eqref{eq:asymptotic_DT} has to vanish identically: because $Q\not\equiv 0$,
  the 2--step (forward) B\"acklund transformations $\check L =\ker(A)$ is at
  the same time a 2--step (backward) B\"acklund transformations, that is,
  $\check L=\hat L = \im(Q)$.  Hence $\ker(\check A)=L$ by Theorem~8 of
  \cite{BFLPP02}.  Because $f$ is immersed, a section $\psi\in \Gamma(L)$ with
  $\nabla\psi \in \Omega^1(L)$ has to vanish identically. This proves the
  statement about the Willmore case with $\eta\equiv0$.

  The rest of the proof deals with the case that $\eta\not \equiv 0$.  Assume
  that, on $U$, there are two (complex) linearly independent solutions $\psi$
  and $\varphi$ to \eqref{eq:asymptotic_DT}. Then, on an open and dense subset
  $U'\subset U$ both section are pointwise linearly independent: if there were
  not, there had to be an open set $\tilde U$ and a complex function $g$ on
  $\tilde U$ such that $\psi=\varphi g$.  But this is impossible, because, by
  the same argument as above, $g$ then had to be constant and prolongations of
  holomorphic sections that are linearly dependent on an open set are linearly
  dependent everywhere.  We now prove that on the set $U$ there cannot be more
  solutions to \eqref{eq:asymptotic_DT} than the complex 2--dimensional space
  spanned by $\psi$ and $\varphi$. Assume $\tilde \psi$ was another solution.
  Then, on the open set $U'$ where $\psi$, $\varphi$ pointwise span the
  subbundle $\{v\in V \mid \check S v = - v i\}$, there would be complex
  valued function $g_1$, $g_2$ such that $\tilde \psi = \psi g_1 + \varphi
  g_2$. The functions $g_1$ and $g_2$ are holomorphic, because, by
  \eqref{eq:zero_one_part}, all solutions to \eqref{eq:asymptotic_DT} are
  holomorphic with respect to the complex holomorphic structure
  $(\nabla-A_\circ)^{(0,1)}$.  Taking the projection of $\nabla\tilde \psi =
  \nabla\psi g_1 + \nabla \varphi g_2+ \psi dg_1 + \varphi dg_2$ to $V/L$
  shows that if one of the functions $g_1$ and $g_2$ is constant, the other
  has to be constant as well.  To prove the claim we have to show that $g_1$
  and $g_2$ are both constant.  Assume that this was not the case, i.e., that
  both functions are non--constant.  The projection of $\nabla\tilde\psi$ to
  $V/L$ then yields $(\pi \psi) = (\pi \varphi) h$ with $h$ the meromorphic
  function defined by $h dg_1 = -dg_2$.  This would force $h$ to be constant
  and $\psi$ and $\varphi$ to be linearly dependent, because the quotient of
  two holomorphic sections of the quaternionic holomorphic line bundle $V/L$
  with non--trivial Hopf field $Q\not \equiv 0$ has to be constant if it is
  complex holomorphic (recall that, by the discussion following
  \eqref{eq:decomposed_el_eq}, the Hopf field $Q$ cannot vanish on any open
  subset of $U$ because $\eta\not\equiv 0$ and $A$ is nowhere vanishing on
  $U$).  Hence, both $g_1$ and $g_2$ have to be constant and the space of
  local solutions to \eqref{eq:asymptotic_DT} defined on the open set $U'$ is
  2--dimensional, because on $U'$ every solution $\tilde \psi$ is linearly
  dependent to $\psi$ and $\varphi$.

  Because prolongations of holomorphic sections are uniquely determined by
  their values on an open set this implies that the space of global solutions
  to \eqref{eq:asymptotic_DT} defined on the universal covering of $T^2$ is at
  most (complex) 2--dimensional.
\end{proof}

\subsection{The case that \texorpdfstring{$A_\circ \not \equiv 0 $ and $\check
    L=\ker(A_\circ)$}{$ker(A_0)$} is constant}
The following lemma is the analogue to Lemma~\ref{lem:fundamental_nonconstant}
in the case that $\check L =\ker(A_\circ)$ is constant. It should be noted
that every constrained Willmore torus in~$S^4$ for which $\check L
=\ker(A_\circ)$ is constant belongs to Case~II or III of
Proposition~\ref{prop:cases}, because a constant section of $\ker(A_\circ)$ is
$\nabla^\mu$--parallel for every $\mu\in \C_*$.  A detailed discussion of
constrained Willmore surfaces with constant $\check L =\ker(A_\circ)$ is given
in Section~\ref{sec:harmonic}.

\begin{Lem}\label{lem:fundamental_constant}
  For a constrained Willmore torus $f\colon T^2\rightarrow S^4$ with the
  property that $\eta\not\equiv 0$ and $\check L =\ker(A_\circ)$ is constant,
  the space of solutions to \eqref{eq:asymptotic_DT} is (complex)
  2--dimensional.
\end{Lem}

The case that $\eta\equiv 0$ and $\check L=\ker(A_\circ)$ is constant which is
excluded from the lemma corresponds to minimal surfaces with planar ends in
the Euclidean space $\R^4=S^4\backslash\{\infty\}$ defined by $\infty =\check
L = \ker(A)$, see Section~\ref{sec:moebius_char}.

\begin{proof}
  As in the proof of Lemma~\ref{lem:2step}, on the open set $U$ where
  $V=L\oplus \check L$ the connection $\nabla$ and the mean curvature sphere
  $S$ of $L$ take the form
  \begin{equation}\label{eq:S_checkLconst}
    \nabla=
    \begin{pmatrix}
      \nabla^L & 0 \\ \delta & \check \nabla
    \end{pmatrix} \qquad \textrm{ and } \qquad S=
    \begin{pmatrix}
      J & B \\ 0 & \tilde J
    \end{pmatrix},
  \end{equation}
  where $J$ and $\tilde J$ are complex structures on $L$ and $\check L$ with
  $*\delta = \delta J = \tilde J \delta$ and where $B\in \Gamma(\Hom(\check L
  ,L))$ with $JB+B\tilde J=0$. The derivative of $S$ is
  \begin{equation}
    \nabla S = 
    \begin{pmatrix}
      \nabla^L J -B\delta & \nabla B  \\
      0 & \check \nabla \tilde J + \delta B
    \end{pmatrix}
  \end{equation}
  and, by $\nabla S= 2{*}Q_\circ - 2{*}A_\circ$, the form $A_\circ$ is given by
  \begin{equation} \label{eq:A_harm} 2{*}A_\circ =
    \begin{pmatrix}
      - \nabla^L J +B\delta & 0  \\
      0 & 0
    \end{pmatrix}.
  \end{equation}
  Let $\psi=(\psi_1,\psi_2)$ be a solution to \eqref{eq:asymptotic_DT} defined
  on $U$.  Then $J\psi_1=\psi_1 i$ and therefore $(\nabla^L J)\psi_1 =
  \nabla^L \psi_1 i - J \nabla^L \psi_1$.  Hence, taking the $(0,1)$--part of
  $\nabla\psi = \nabla^L\psi_1$ yields
  \[ (\nabla \psi)^{(0,1)} = \tfrac12(\nabla^L \psi_1 + J \nabla^L \psi_1 i )
  = \tfrac12 J (\nabla^L J) \psi_1. \] On the other hand, on $U$, equation
  \eqref{eq:asymptotic_DT} implies \eqref{eq:zero_one_part} such that by
  \eqref{eq:A_harm}
  \[ (\nabla \psi)^{(0,1)} = \tfrac12 J (\nabla^L J) \psi_1 - \tfrac12 J B
  \delta\psi_1. \] Together, the last two equations imply that solutions to
  \eqref{eq:asymptotic_DT} with $\psi_1\not\equiv 0$ can only exist if
  $B\equiv 0$ which is impossible because, by \eqref{eq:A_harm}, $B\equiv 0$
  is equivalent to $A_\circ$ anti--commuting with $S$ which again is
  equivalent to $\eta\equiv 0$.
\end{proof}

\subsection{Relation between holonomy and multiplier spectral curve}
As an application of Proposition~\ref{pro:asymptotic} we show now that the
holonomy spectral curve $\Sigma_{hol}$ essentially coincides with the
multiplier spectral curve $\Sigma_{mult}$ provided they are both defined. This
is the case for constrained Willmore tori belonging to Case~I or II of
Proposition~\ref{prop:cases}.

\begin{The} \label{the:biholomorphicity} Let $f\colon T^2\rightarrow S^4$ be a
  constrained Willmore torus for which both $\Sigma_{hol}$ and $\Sigma_{mult}$
  are defined. The holomorphic map $\iota \colon \Sigma_{hol} \rightarrow
  \Sigma_{mult}$ is an injective immersion whose image is $\Sigma_{mult}$ with
  finitely many points removed.  The multipliers of the removed points or
  their conjugates belong to holomorphic sections whose prolongations solve
  \eqref{eq:asymptotic_DT}.  In particular, all but finitely many points in
  $\Sigma_{mult}$ give rise to (possibly singular) Darboux transforms which
  are again constrained Willmore where they are immersed.  In case $f$ is
  Willmore, all Darboux transforms belonging to points of $\Sigma_{mult}$ are
  again Willmore.
\end{The}
 
As suggested by the theorem, we will in the following not distinguish between
$\Sigma_{hol}$ and its image $\iota(\Sigma_{hol})$ under $\iota\colon
\Sigma_{hol}\rightarrow \Sigma_{mult}$. Theorem~\ref{the:mainA} below shows
that $\Sigma_{mult}\backslash \Sigma_{hol}$ consists of at most four points.

\begin{Cor}\label{cor:biholomorphicity_willmore_AQ_null}
  If $f\colon T^2\rightarrow S^4$ is a Willmore torus with $\eta\equiv0$ and
  $AQ\equiv0$, then $\Sigma_{hol}=\Sigma_{mult}$ in case they are both
  defined.
\end{Cor}
\begin{proof}
  This follows from Lemma~\ref{lem:fundamental_nonconstant} according to which
  the space of solutions to \eqref{eq:asymptotic_DT} is 0--dimensional if
  $AQ\equiv 0$.
\end{proof}

\begin{proof}[Proof of Theorem~\ref{the:biholomorphicity}]
  If the map $\iota \colon \Sigma_{hol} \rightarrow \Sigma_{mult}$ is not
  injective, there is a non--empty open subset of $\Sigma_{mult}$ all
  points of which have several preimages. In particular, there is
  $\sigma\in\Sigma_{mult}$ that has two preimages belonging to different
  parameters $\mu_1$ and $\mu_2\in \C_*$ and for which the space of
  holomorphic section with monodromy $h^\sigma$ is 1--dimensional, see
  Section~\ref{sec:multiplier_spec}.  The prolongation $\psi$ of such a
  holomorphic section satisfies
  \[ \nabla \psi = (1-\mu_l) (A_\circ\psi)^{(1,0)} + (1-\mu_l^{-1}) (
  A_\circ\psi)^{(0,1)} \] for $l=1$ and $2$. Because $\mu_1\neq \mu_2$ this
  implies $A_\circ\psi=0$ which yields
  \[ 0 = d^\nabla({*}A_\circ \psi ) = -{*}A_\circ \wedge \nabla\psi. \]
  Because $\nabla\psi$ takes values in $KL$ and $(A_\circ)_{|L}=A_{|L}$ is
  right $\bar K$ and does not vanish on the non--empty open set $U$ where
  $V=L\oplus \check L$, this forces $\psi$ to be constant on $U$ and hence
  everywhere. But this contradicts the assumption that the multiplier
  $h^\sigma$ is non--trivial such that the map $\iota$ is injective and in
  particular unbranched.

  The image $\iota(\Sigma_{hol})$ is an open subset of $\Sigma_{mult}$.  If
  $\Sigma_{mult}$ is not connected it has two connected components which get
  interchanged by $\rho$ and, because $\iota$ is compatible with the
  involutions~$\rho$ the image $\iota(\Sigma_{hol})$ intersects both
  components.  In order to prove the lemma we show that the boundary of
  $\iota(\Sigma_{hol})$ in $\Sigma_{mult}$ consist of points whose multipliers
  or their conjugates belong to global solutions to \eqref{eq:asymptotic_DT}.
  Assume there is a sequence of points $\sigma_n\in \iota(\Sigma_{hol})$
  converging to $\sigma_0\in \Sigma_{mult}\backslash \iota(\Sigma_{hol})$.
  The corresponding sequence of parameters $\mu_n\in \C_*$ cannot be contained
  in a bounded subset of $\C_*$: otherwise we could assume by passing to a
  subsequence that $\mu_n$ converges to some $\mu_0\in \C_*$. But then
  $\nabla^{\mu_n}\psi^{\sigma_n}=0$ would imply $\nabla^{\mu_0}
  \psi^{\sigma_0}=0$ which contradicts the assumption that $\sigma_0$ is not
  contained in $\iota(\Sigma_{hol})$; here $\psi^\sigma$ denotes the
  prolongation of a local holomorphic section defined near $\sigma_0$ of the
  line bundle $\mathcal{L}$ from Section~\ref{sec:multiplier_spec}, that is,
  every $\psi^\sigma$ is the prolongation of a holomorphic section
  $\pi\psi^\sigma$ with monodromy $h^\sigma$ of $V/L$ and $\sigma\mapsto
  \pi\psi^\sigma$ depends holomorphically on $\sigma$.

  Hence $\mu_n$ is not contained in a bounded subset of $\C_*$ and, by
  passing to a subsequence and possibly applying the anti--holomorphic
  involutions $\rho$ of $\Sigma_{hol}$ and $\Sigma_{mult}$, we can assume that
  $\mu_n$ converges to $\infty$. Because
  \[ \nabla\psi^{\sigma_n} = (1-\mu_n) (A_\circ\psi^{\sigma_n})^{(1,0)} +
  (1-\mu_n^{-1}) (A_\circ\psi^{\sigma_n})^{(0,1)} \] for all $n$, while
  $\nabla\psi^{\sigma_n} \rightarrow \nabla\psi^{\sigma_0}$ and $
  A_\circ\psi^{\sigma_n} \rightarrow A_\circ\psi^{\sigma_0}$ for $n\rightarrow
  \infty$, we obtain
  \[ (A_\circ\psi^{\sigma_0})^{(1,0)}=0. \] This shows that for every point
  $\sigma$ in the boundary of $\iota(\Sigma_{hol})$ the multiplier $h^\sigma$
  or $\bar h^\sigma= h^{\rho(\sigma)}$ admits a holomorphic section of $V/L$
  whose prolongation solves \eqref{eq:asymptotic_DT}.  By
  Proposition~\ref{pro:asymptotic}, there are at most two multipliers
  belonging to solutions to \eqref{eq:asymptotic_DT}, because holomorphic
  sections of $V/L$ with different monodromies are linearly independent over
  $\C$. The boundary of $\iota(\Sigma_{hol})$ thus consists of finitely many
  points and hence coincides with its complement
  $\Sigma_{mult}\backslash\iota(\Sigma_{hol})$.

  By Theorem~\ref{the:nabla_dt_cw}, the (possibly singular) Darboux
  transforms corresponding to points of $\Sigma_{hol}$ are constrained
  Willmore or Willmore, respectively. In the Willmore case, by continuity this
  holds for all points of $\Sigma_{mult}$.
\end{proof}

Theorem~\ref{the:biholomorphicity} immediately implies that the spectral curve
$\Sigma_{mult}$ of a constrained Willmore torus that belongs to Case~I or II
has finite genus: in the infinite genus case the spectral curve of a conformal
torus in $S^4$ has only one end asymptotic to two planes joined by an infinite
number of handles, see Theorem~4.1 of \cite{BPP2}.  But this is impossible for
a constrained Willmore torus of Case~I or II, because its spectral curve
$\Sigma_{mult}$ has two ends one of which is an end of $\{\sigma \in
\Sigma_{hol} |\mid \mu(\sigma)|<1\}$ while the other is an end of $\{\sigma
\in \Sigma_{hol}\mid |\mu(\sigma)|>1\}$.

In Section~\ref{sec:proof} we give a more geometric proof of the fact that
constrained Willmore tori of Case~I and II have finite spectral genus by
showing the existence of a polynomial Killing field.  Moreover, we prove that
constrained Willmore tori of Case~III are super conformal or Euclidean minimal
with planar ends.

\subsection{Constant Darboux transforms in $\Sigma$}
For conformally immersed tori $f\colon T^2\rightarrow S^4$ with trivial normal
bundle, the normalization map $h\colon \Sigma_{mult}\rightarrow
\Hom(\Gamma,\C_*)$ of the spectrum has a special multiple point at the trivial
multiplier $h=1 \in \Hom(\Gamma,\C_*)$. This singularity is characteristic for
spectral curves of quaternionic holomorphic line bundles of degree~0 that are
induced by immersed tori.  Among the points desingularizing this singularity
one is especially interested in the points corresponding to constant Darboux
transforms of $f$. In the following we investigate the number of such points
in case of constrained Willmore tori in $S^4$.  For general conformal
immersions $f\colon T^2\rightarrow S^4$ very few is known about this number.

\begin{Lem}
  Let $f\colon T^2\rightarrow S^4$ be a constrained Willmore immersion of a
  torus which belongs to Case~I or II of Proposition~\ref{prop:cases}.
  \begin{itemize}
  \item If $\check L = \ker(A_\circ)$ is non--constant, the only points in
    $\Sigma_{mult}$ that correspond to constant Darboux transforms are the 2
    or 4 points $\mu^{-1}(\{1\})\subset \Sigma_{hol}$.
  \item If $\check L = \ker(A_\circ)$ is constant, then
    $\mu^{-1}(\{1\})\subset \Sigma_{hol}$ consists of 2 points corresponding
    to constant Darboux transforms.  Moreover, every constant Darboux
    transform that corresponds to a point in $\Sigma_{mult} \backslash
    \mu^{-1}(\{1\})$ is contained in $\check L =\ker(A_\circ)$.
  \end{itemize}
\end{Lem}

\begin{proof}
  Because $\nabla^{\mu=1}$ is trivial, the points in the fiber
  $\mu^{-1}(\{1\})\subset \Sigma_{hol}$ correspond to constant Darboux
  transforms. The set $\mu^{-1}(\{1\})$ is invariant under the fixed point
  free involution $\rho$ and therefore consists of 2 or 4 points. For
  immersions whose holonomy representation belongs to Case~II, for example if
  $\check L = \ker(A_\circ)$ is constant, $\Sigma_{hol}$ is a 2--fold branched
  covering of $\C_*$ and $\mu^{-1}(\{1\})$ consists of 2 points.  In case the
  Darboux transform corresponding to a point $\Sigma_{hol}\backslash
  \mu^{-1}(\{1\})$ is constant it has to be contained in $\check L
  =\ker(A_\circ)$ which is only possible if $\check L$ is constant.

  A Darboux transform corresponding to a point in $\Sigma_{mult}\backslash
  \Sigma_{hol}$ is never constant unless $\check L =\ker(A_\circ)$ is
  constant: such Darboux transform has to be contained in $\check L$, because
  by Theorem~\ref{the:biholomorphicity} it solves \eqref{eq:asymptotic_DT}
  and, on a non--empty open set $U$, \eqref{eq:zero_one_part} such that
  $\check L$ is constant on $U$ and hence everywhere.  (In the following
  section we will see that $\Sigma_{mult}\backslash \Sigma_{hol}=\emptyset$ if
  $\check L$ is constant.)
\end{proof}

For some special cases, using results from the following two sections enables
us to give more precise information about the number of points in
$\Sigma_{mult}$ that correspond to constant Darboux transforms:
\begin{enumerate}[a)]
\item All minimal tori in the metrical 4--sphere $S^4$ that are not  super
  conformal belong to Case~I and the fiber $\mu^{-1}(\{1\})$ consists of 4
  points corresponding to constant Darboux transforms: the fact that they
  belong to Case~I follows from Corollary~\ref{cor:willmore}. The fiber
  $\mu^{-1}(\{1\})$ consists of 4 points, because for all $\mu\in S^1$ the
  holonomy of $\nabla^\mu$ is contained in $\SU(4)$ such that $\mu\colon
  \Sigma_{hol}\rightarrow \C_*$ is unbranched over the unit circle.
\end{enumerate}

If $\im(Q_\circ)$ is constant and the immersion is neither super conformal nor
Euclidean minimal with planar ends, by Proposition~\ref{pro:cases_harmonic}
its holonomy belongs to Case~II and, by Lemma~\ref{lem:2by2_to_4by4}, the two
points $\mu^{-1}(\{1\})\subset \Sigma_{hol}$ correspond to constant Darboux
transforms contained in $\im(Q_\circ)$. If $\ker(A_\circ)$ is non--constant,
no other points in $\Sigma_{mult}$ correspond to constant Darboux transforms.
If $\ker(A_\circ)$ is as well constant, there is at most one other pair of
points corresponding to constant Darboux transforms in $\Sigma_{mult}$
(following from the fact that a constant Darboux transform has to be contained
in $\ker(A_\circ)$ and that parallel sections for different $\mu$ of the
connections $\nabla^\mu$ on $V/L$ defined in
\eqref{eq:associated_willmoreconnection} are linearly independent over~$\C$).
\begin{enumerate}[b)]
\item For CMC tori in $\R^3$, the Lagrange multiplier $\eta$ can be chosen
  such that $\ker(A_\circ)=\im(Q_\circ)$, see Section~\ref{sec:cmc_r3}.  This
  shows that, apart from the 2 points $\mu^{-1}(\{1\})$ there are no other
  points in $\Sigma_{mult}$ that correspond to constant Darboux transforms.
\item[c)] In case of CMC tori in $S^3$ with mean curvature
  $H^{S^3}=\cot(\alpha/2)$, formula \eqref{eq:mc_cmc_s3} shows that for the
  right choice of $\eta$ (namely $\rho=\pm1/2$ in Section~\ref{sec:cmc_s3}),
  the 4 points $\mu^{-1}(\{1\})$ and $\mu^{-1}(\{e^{i\alpha}\})$ are the only
  points corresponding to constant Darboux transforms in $\im(Q_\circ)$ and
  $\ker(A_\circ)$ respectively.
\end{enumerate}

\section{The Main Theorem and its Proof}
\label{sec:proof}
We prove the main theorem of the paper by separately dealing with all possible
cases of holonomy representations that occur for the associated family of
constrained Willmore tori. For Cases~I and II of Proposition~\ref{prop:cases}
we prove the existence of a polynomial Killing field $\xi$, a family of
sections of $\End_\C(V)$ which is polynomial in $\mu$ and, for all $\mu\in
\C_*$, satisfies $\nabla^\mu\xi(\mu,.)=0$.  Because $\xi$ commutes with all
holonomies, its existence implies that $\Sigma_{hol}$ and hence
$\Sigma_{mult}$ can be compactified by adding points at infinity.  For
Case~IIIa of Proposition~\ref{prop:cases} we prove the existence of a
polynomial family of $\nabla^\mu$--parallel sections of the complex rank
4--bundle $(V,i)$ itself and for Case~IIIb the existence of a nil--potent
polynomial Killing field $\xi$.  Investigating the asymptotics of such
polynomial families of sections reveals that Cases~IIIa and IIIb correspond to
immersions that are super conformal or Euclidean minimal with planar ends.

\subsection{Main theorem of the paper}\label{sec:main_th}
The following is a more detailed formulation of the main theorem stated in the
introduction. Its proof will be given in Sections~\ref{sec:hitchin_trick} to
\ref{sec:proof_main}.

\begin{The}\label{the:mainA}
  Let $f\colon T^2 \rightarrow S^4$ be a constrained Willmore immersion from a
  torus into the conformal 4--sphere $S^4$. Then one of the following holds:
  \begin{enumerate}
  \item[I.] The holonomy spectral curve $\Sigma_{hol}$ can be compactified to
    a 4--fold branched covering of $\CP^1$ by adding points over $\mu=0$ and
    $\mu=\infty$.  The two ends of the spectral curve $\Sigma_{mult}$
    correspond to branch points of $\mu$ over $\mu=0$ and $\mu=\infty$. The
    complement $\Sigma_{mult}\backslash \Sigma_{hol}$ of the holonomy spectral
    curve inside the multiplier spectral curve consists of at most four
    points.
  \item[II.] The holonomy and multiplier spectral curves coincide
    $\Sigma_{mult}=\Sigma_{hol}$ and can be compactified to a 2--fold branched
    covering of $\CP^1$ by adding one point over $\mu=0$ and one over
    $\mu=\infty$.
  \item[III.] The immersion $f$ is super conformal or minimal with planar ends
    in the Euclidean space $\R^4=S^4\backslash \{\infty\}$ defined by some
    point $\infty\in S^4$ at infinity. More precisely:
    \begin{itemize}
    \item $f$ belongs to Case~IIIa in Proposition~\ref{prop:cases} if and only
      if it is super conformal or $f\colon T^2 \backslash\{p_1,...,p_n\}
      \rightarrow \R^4\cong S^4\backslash\{\infty\}$ is an algebraic Euclidean
      minimal immersion with planar ends (with algebraic meaning that the
      closed form $*df$ has no periods such that $f$ is the real part of a
      meromorphic null immersion from $T^2$ into $\C^4$).
    \item $f$ belongs to Case~IIIb in Proposition~\ref{prop:cases} if and only
      if $f$ is a non--algebraic Euclidean minimal immersion with planar ends
      (that is, $*df$ has periods).
    \end{itemize} 
  \end{enumerate}
  If the normal bundle $\perp_f$ of the immersion $f$ is topologically trivial
  and $f$ is not Euclidean minimal with planar ends, it belongs to the
  ``finite type'' Cases~I and II.  If the normal bundle $\perp_f$ is
  non--trivial, then $f$ is of ``holomorphic type'' and belongs to Case~III.
\end{The}

It should be noted that for an isothermic constrained Willmore torus $f\colon
T^2\rightarrow S^4$ which belongs to Case~I or II in Theorem~\ref{the:mainA},
changing the form $\eta\in \Omega^1(\mathcal{R})$ in the Euler--Lagrange
equation \eqref{eq:def_constrained-willmore} changes $\Sigma_{hol}$ which
amounts to changing the meromorphic function $\mu$ on the multiplier spectral
curve $\Sigma_{mult}$.  For different choices of $\eta$, the holonomy curve
might then change between Cases~I and II of the theorem.  This happens for
example in case of minimal tori in $S^3$, see Section~\ref{sec:cmc_s3}.

\begin{Cor}\label{cor:willmore}
  Let $f\colon T^2 \rightarrow S^4$ be a Willmore torus with $\eta\equiv 0$
  that is not Euclidean minimal with planar ends. Then:
  \begin{itemize}
  \item $f$ belongs to Case~I if and only if $\deg(\perp_f)=0$ and
  \item $f$ belongs to Case~III and is super conformal if and only if
    $\deg(\perp_f)\neq0$.
  \end{itemize}
\end{Cor}
\begin{proof}
  For immersions belonging to Case~II or III, one can check as in the proof of
  Lemma~\ref{lem:caseIIIa} below that the fiber $\mathcal{V}_\infty$ over
  $\mu=\infty$ of the holomorphic vector bundle $\mathcal{V}$ constructed in
  Lemma~\ref{lem:case_polynomial_conserved_quantity} is a space of solutions
  to \eqref{eq:asymptotic_DT} of dimension greater or equal~$2$.  By
  Lemma~\ref{lem:fundamental_nonconstant}, such bundle $\mathcal{V}$ cannot
  exist for a Willmore torus $f$ with $\eta\equiv 0$ that is neither super
  conformal nor Euclidean minimal.
\end{proof}

\begin{Cor}\label{cor:irreducible}
  If the holonomy spectral curve $\Sigma_{hol}$ of a constrained Willmore
  torus that belongs to Case~I or II in Theorem~\ref{the:mainA} coincides with
  $\Sigma_{mult}$, the spectral curve $\Sigma_{mult}=\Sigma_{hol}$ is
  irreducible, that is, has a single connected component.
\end{Cor}

\begin{proof}
  Since $\Sigma_{hol}=\Sigma_{mult}$ has only two ends, cf.\
  Section~\ref{sec:multiplier_spec}, it can be compactified by glueing in a
  single point over $\mu=0$ and $\infty$, respectively.  Because the branch
  order of the meromorphic function $\mu$ at the two added points is then 3 in
  Case~I and 1 in Case~II, the total space of the branched covering $\mu
  \colon \Sigma_{hol}\rightarrow \C_*$ is connected.
\end{proof}

Irreducibility of the spectral curve automatically holds
\begin{itemize}
\item in Case~II, because then always $\Sigma_{hol}=\Sigma_{mult}$, and
\item for Willmore immersions $f$ with $\eta\equiv 0$ and $AQ\equiv 0$, see
  Corollary~\ref{cor:biholomorphicity_willmore_AQ_null}.
\end{itemize}
The condition $AQ=0$ is satisfied for all Willmore tori in the conformal
3--sphere $S^3$ and for minimal tori in the metrical 4--sphere or in
hyperbolic 4--space, see Chapter~10 of \cite{BFLPP02}.

\begin{Cor}\label{cor:small_willmore_energy}
  Let $f\colon T^2\rightarrow S^4$ be a constrained Willmore torus whose
  Willmore functional is bounded by $\mathcal{W}< 8\pi$.  Then:
  \begin{itemize}
  \item $f$ belongs to Case~I if and only if $\check L = \ker(A_\circ)$ is
    non--constant and
  \item $f$ belongs to Case~II if and only if $\check L = \ker(A_\circ)$ is
    constant.
  \end{itemize}
  In the latter case $\im(Q_\circ)$ is also constant and $\ker(A_\circ)\neq
  \im(Q_\circ)$.
\end{Cor}
\begin{proof}
  A constrained Willmore torus $f$ with $\mathcal{W}< 8\pi$ cannot belong to
  Case~III of Theorem~\ref{the:mainA}, because super conformal and Euclidean
  minimal immersions with planar ends have Willmore energy $\mathcal{W}\geq
  8\pi$.  Moreover, $f$ has trivial normal bundle and the quaternionic
  Pl\"ucker formula \cite{FLPP01} implies that the space $H^0(V/L)$ of
  holomorphic sections with trivial monodromy is quaternionic 2--dimensional.
  All holomorphic sections of $V/L$ with trivial monodromy are thus
  projections of constant sections of $V$. This shows that $f$ belongs to
  Case~II if and only if $\ker(A_\circ)$ is constant, because $\nabla^\mu$
  with $\mu\in \C_*\backslash \{1\}$ has parallel sections with trivial
  monodromy if and only if $\ker(A_\circ)$ is constant.

  The fact that $\im(Q_\circ)$ is also constant in case $\ker(A_\circ)$ is
  constant follows by passing to the dual constrained Willmore
  surface~$f^\perp$ because, by \eqref{eq:gauge_trafo} and
  \eqref{eq:associated_perp}, the immersion $f$ belongs to Case~II if and only
  if its dual immersion $f^\perp$ belongs to Case~II.
  Lemma~\ref{lem:both_constant} implies that $\ker(A_\circ)\neq \im(Q_\circ)$,
  because otherwise $f$ had to be CMC in $\R^3$ or Euclidean minimal with
  planar ends which is impossible if $\mathcal{W}< 8\pi$.
\end{proof}

\subsection{Hitchin trick}\label{sec:hitchin_trick}
In the following we construct, for each case in the holonomy list of
Proposition~\ref{prop:cases}, a family of sections of either $(V,i)$ or
$\End_\C(V)$ that is polynomial in the spectral parameter $\mu$ and has the
property that the evaluation at every $\mu\in \C_*$ is $\nabla^\mu$--parallel.
To do this we use the following idea from \cite{Hi}: although $\nabla^\mu$ has
poles at $\mu=0$ and $\mu=\infty$, the holomorphic families of elliptic
operators
\begin{equation}
  \label{eq:holomorphic_families}
  \begin{split}
    \partial^\mu & = (\nabla^\mu)^{(1,0)}= \nabla^{(1,0)} + (\mu-1)
    A_\circ^{(1,0)}
    \\
    \dbar^\mu & = (\nabla^\mu)^{(0,1)}= \nabla^{(0,1)} + (\mu^{-1}-1)
    A_\circ^{(0,1)}
  \end{split}
\end{equation}
obtained by taking the $(1,0)$-- and $(0,1)$--parts of the associated family
$\nabla^\mu$ of flat connections extend to $\mu = 0$ or $\mu = \infty$,
respectively.  Because elliptic operators on compact manifolds are Fredholm,
Proposition~\ref{pro:holomorphic_family} applies to the holomorphic families
$\partial^\mu$ and $\dbar^\mu$. We use this to show for each of the cases in
Proposition~\ref{prop:cases} that the family of spaces of
$\nabla^\mu$--parallel sections of either $(V,i)$ or $\End_\C(V)$ gives rise
to a holomorphic vector bundles over $\CP^1$ whose fiber over generic points
$\mu\in \C_*\subset \CP^1$ coincides with the space of $\nabla^\mu$--parallel
sections.  A polynomial family of parallel sections can then be obtained as a
meromorphic section with a single pole at $\mu=\infty$ of this holomorphic
vector bundle over~$\CP^1$.

\subsection{Case by case study}\label{sec:case_by_case} 
A \emph{polynomial Killing field} \cite{FPPS,BFPP} for the constrained
Willmore associated family $\nabla^\mu$ is a polynomial family
\[\xi(\mu,p) = \sum_{l=0}^d \xi_l(p) \mu^l \] 
of sections of $\End_\C(V)$ with coefficients $\xi_l\in \Gamma(\End_\C(V))$
whose value at every $\mu\in \C_*$ is $\nabla^\mu$--parallel, that is,
satisfies
\begin{equation}
  \label{eq:killing_field}
  \nabla^\mu \xi(\mu,\underline{\phantom{x}}) = 0
\end{equation}
which is equivalent to the Lax--type equation
\[ \nabla\xi = \big[(1-\mu)A_\circ^{(1,0)}+(1-\mu^{-1})A^{(0,1)}_\circ,
\xi\big].\] 

\begin{Lem}\label{lem:caseI} 
  A constrained Willmore torus that belongs to \textbf{Case~I} of
  Proposition~\ref{prop:cases} admits a polynomial Killing field $\xi$ which,
  for generic $\mu$, has four different eigenvalues.
\end{Lem}

\begin{proof}
  Away from isolated spectral parameters $\mu\in \C_*$, the space of
  $\nabla^\mu$--parallel sections of the bundle $\End_\C(V)$ is
  4--dimensional, because parallel endomorphisms are parametrized by the
  elements in the fiber $\End_\C(V)_p$ over one point $p\in T^2$ that commute
  with all $H^\mu_p(\gamma)$, $\gamma\in \Gamma$. In order to prove the
  existence of a polynomial Killing field we construct now a holomorphic
  rank~4 subbundle $\mathcal{U}$ of the trivial bundle $\CP^1\times
  \Gamma(\End_\C(V)$) whose fiber $\mathcal{U}_\mu$ over generic $\mu \in
  \C_*$ coincides with the space of $\nabla^\mu$--parallel sections.  As
  sketched in Section~\ref{sec:hitchin_trick}, we do this by applying
  Proposition~\ref{pro:holomorphic_family} to the operators
  \eqref{eq:holomorphic_families} acting on sections of $\End_\C(V)$. The
  existence of $\mathcal{U}$ then follows by proving that there is one $\mu\in
  \C_*$ for which the kernel of $\dbar^\mu$ (or, which is equivalent by
  \eqref{eq:symmetry_nabla_mu}, the kernel of
  $\partial^\mu=j^{-1}\dbar^{1/\bar \mu} j$) coincides with the space of
  $\nabla^\mu$--parallel sections.

  Assume this was not the case. For all $\mu$ the dimension of the space of
  $\dbar^\mu$--holomorphic endomorphisms is then higher than~4. For generic
  $\mu\in \C_*$, the bundle $V$ is a direct sum $V=E_{h_1} \oplus ...  \oplus
  E_{h_4}$ of $\nabla^\mu$--parallel subbundles whose parallel sections have
  different multipliers $h_1$,..., $h_4$.  Therefore, $\End_\C(V) =
  \oplus_{kl} \, E_{h_k} E_{h_l}^{-1}$ and the fact that there are more
  $\dbar^\mu$--holomorphic than $\nabla^\mu$--parallel endomorphisms implies
  that for some $k\neq l$ there exists a non--trivial holomorphic homomorphism
  between the complex holomorphic line bundles $E_{h_k}$ and $E_{h_l}$ of
  degree~0, that is, the bundles $E_{h_k}$ and $E_{h_l}$ are holomorphically
  equivalent and represent the same point in the Jacobi variety of~$T^2$. In
  other words, over generic points $\mu\in \C_*$ on the 4--sheeted branched
  covering $\Sigma_{hol}\rightarrow \C_*$ the holomorphic map $h\colon
  \Sigma_{hol}\rightarrow \Hom(\Gamma,\C_*)$ takes 4 different values
  corresponding to 4 different gauge equivalence classes of flat complex line
  bundles, but the induced map into the Jacobian, i.e., the space of isomorphy
  classes of complex holomorphic line bundles, takes the same value on at
  least 2 of the 4 points over $\mu$.  Because locally, away from the set
  $B=\{\sigma\in \Sigma_{hol} \mid \mu(\sigma) \textrm{ is singular value of }
  \mu \}$, for every pair of sheets of the covering $\Sigma_{hol}\backslash B
  \rightarrow \C_*$ the multipliers belonging to the pair of points over the
  same $\mu$'s are either all holomorphically equivalent or holomorphically
  equivalent for isolated $\mu$'s only, the non--empty set
  \begin{multline*}
    U = \{ \sigma\in \Sigma_{hol} \backslash B \mid \textrm{for all } \sigma'
    \textrm{ near } \sigma \textrm{ there is }
    \sigma''\neq \sigma' \textrm{ with } \mu(\sigma')=\mu(\sigma'') \\
    \textrm{ for which } h^{\sigma'} \textrm{ and } h^{\sigma''} \textrm{ give
      rise to isomorphic holomorphic line bundles} \}
  \end{multline*}
  is open and closed and hence a connected component of $\Sigma_{hol}
  \backslash B$.

  The assumption that $U$ is non--empty implies $U=\Sigma_{hol} \backslash B$:
  this is clear in case that the Riemann surface $\Sigma_{mult}$ and
  therefore, by Theorem~\ref{the:biholomorphicity}, $\Sigma_{hol}$ is
  connected.  If $\Sigma_{mult}$ is not connected it has two connected
  components and the assumption $U\neq \emptyset$ and $U\neq \Sigma_{hol}
  \backslash B$ implies that $U$ coincides with one of the two connected
  components of $\Sigma_{hol} \backslash B$. Each of the two components of
  $\Sigma_{mult}$ has one end which, by Lemma~5.2 of \cite{BPP2}, can be
  parametrized by $x\in D^*$ in a punctured disc $D^*=\{x\in \C_*\mid
  |x|<\epsilon\}$ such that either $h^{\sigma(x)}(\gamma)= \exp((a_0+x^{-1})
  \gamma + b(x) \bar \gamma)$ or $h^{\sigma(x)}(\gamma)= \exp(a(x)\gamma +
  (b_0+ x^{-1}) \bar \gamma)$ for all $x\in D^*$ and $\gamma\in \Gamma$ where
  $a(x)$ respectively $b(x)$ are holomorphic in $x=0$ (note that we use here
  that a reducible spectral curve has finite genus, see Theorem~4.1 of
  \cite{BPP2}).  The first kind of end cannot be contained in $U$, because
  then the holomorphic map $U\rightarrow T^*=\C/\Gamma'$ given by
  $\sigma\mapsto b(\sigma) \textrm{ mod } \Gamma'$, where $\Gamma'$ is the
  lattice
  \[\Gamma'=\{c\in \C \mid -\bar c \gamma + c \bar \gamma \in 2\pi i \Z
  \textrm{ for all } \gamma\in \Gamma\}\] and $b(\sigma)$ satisfies
  $h^\sigma(\gamma) = \exp(a(\sigma)\gamma+b(\sigma)\bar \gamma)$ for all
  $\gamma\in \Gamma$, would descend to the $\mu$--plane and could be extended
  to a holomorphic map $\CP^1\rightarrow T^*$ which had to be constant. To see
  that the second kind of end cannot be contained in $U$ we use that the map
  $x\mapsto \mu(\sigma(x))$ is a 2--fold covering of a punctured neighborhood
  of $\mu=0$ or $\mu=\infty$ and introduce another coordinate $y$ at the end
  defined by $\mu=y^2$ or $\mu=y^{-2}$ respectively. Because the multipliers
  $h^{\sigma(y)}$ and $h^{\sigma(-y)}$ are holomorphically equivalent this
  implies $\frac{1}{x(y)} = \frac{1}{x(-y)} + c$ for $c \in \Gamma'$ which is
  impossible, because the residues of both sides of the equation have opposite
  signs.

  We prove now by a consideration for $\mu \in S^1\subset \C_*$, that the
  assumption $U\neq \emptyset$ and hence $U=\Sigma_{hol} \backslash B$ leads
  to a contradiction.  For this, we have to deal with two different cases:
  either for all $\sigma\in \mu^{-1}(S^1)$ the multipliers $h^\sigma$ and
  $h^{\rho(\sigma)}$ are holomorphically equivalent or for all $\sigma_1\in
  \mu^{-1}(S^1)$ there is $\sigma_2\neq \rho(\sigma_1)$ with
  $\mu(\sigma_1)=\mu(\sigma_2)$ such that $h^{\sigma_1}$ and $h^{\sigma_2}$ as
  well as $h^{\rho(\sigma_1)}$ and $h^{\rho(\sigma_2)}$ are holomorphically
  equivalent.

  To prove that the first case is impossible we write the multiplier
  $h^\sigma$ of $\sigma\in \mu^{-1}(S^1)$ as
  \[ h^\sigma(\gamma) = \exp(a\gamma+ b\bar\gamma) \] where $a$, $b\in \C$ are
  unique up to the action of $c\in \Gamma'$ by $(a,b)\mapsto (a-\bar c, b+c)$.
  The multiplier of $\rho(\sigma)$ is then $h^{\rho(\sigma)}(\gamma)=\exp(\bar
  b \gamma + \bar a \bar \gamma)$ and, because $h^\sigma$ and
  $h^{\rho(\sigma)}$ give rise to isomorphic holomorphic line bundles, $\bar
  a= b+c$ for some $c\in \Gamma'$.  But this implies
  $h^{\rho(\sigma)}(\gamma)=\exp((a-\bar c) \gamma + (b+c) \bar \gamma)$ such
  that $h^\sigma=h^{\rho(\sigma)}$ for all $\sigma\in \mu^{-1}(S^1)$ which is
  impossible because in Case~I over generic $\mu \in \C_*$ all multipliers
  are different.

  To prove that the second case is impossible, we chose two curves
  $\sigma_1(t)$ and $\sigma_2(t)$ in $\mu^{-1}(S^1)\subset \Sigma_{hol}$ with
  $\mu(\sigma_1(t))= \mu(\sigma_2(t))$ for all $t$ and $\mu(\sigma_1(0))=
  \mu(\sigma_2(0))=1$ such that $h^{\sigma_1(t)}$ and $h^{\sigma_2(t)}$ are
  holomorphically equivalent. Then there are unique complex functions
  $a_1(t)$, $a_2(t)$ and $b(t)$ with $a_1(0)=a_2(0)=b(0)$ such that
  \[ h^{\sigma_1(t)}(\gamma) = \exp(a_1(t) \gamma+ b(t) \bar \gamma) \qquad
  \textrm{ and } \qquad h^{\sigma_2(t)}(\gamma) = \exp(a_2(t) \gamma+ b(t)
  \bar \gamma) \] for all $t$.  Because $h^{\rho(\sigma_1(t))}$ and
  $h^{\rho(\sigma_2(t))}$ are holomorphically equivalent as well there is a
  constant $c\in \Gamma'$ such that $a_2(t) = a_1(t) -\bar c$.  Evaluating at
  $t=0$ implies $c=0$. As a consequence, $h^{\sigma_1(t)}=h^{\sigma_2(t)}$ for
  all $t$ which is impossible.

  Hence $U=\emptyset$ and the minimal kernel dimension of $\dbar^\mu$ (and
  $\partial^\mu$) on $\End_\C(V)$ is 4.
  Proposition~\ref{pro:holomorphic_family} applied to the operators
  \eqref{eq:holomorphic_families} acting on sections of $\End_\C(V)$ thus
  yields the existence of a holomorphic rank~4 subbundle $\mathcal{U}$ of the
  trivial bundle $\CP^1\times \Gamma(\End_\C(V))$ whose fiber over a generic
  point $\mu\in \C_*$ coincides with the space of $\nabla^\mu$--parallel
  endomorphisms.  A global meromorphic section of $\mathcal{U}$ with a single
  pole at $\mu=\infty$ is then a polynomial Killing field $\xi$.
\end{proof}

In Cases~II and III of Proposition~\ref{prop:cases}, the space of
$\nabla^\mu$--parallel sections with trivial monodromy for $\mu \in \C_*$ is
at least a (complex) 2--dimensional subspace of the finite dimensional space
$\mathcal{H}=\{ \psi \in \Gamma(V)\mid \nabla\psi \in \Omega^1(L)\}$ of
prolongations of holomorphic sections with trivial monodromy of $V/L$.

\begin{Lem}\label{lem:case_polynomial_conserved_quantity}
  For a constrained Willmore torus that belongs to \textbf{Case~II or III} of
  Proposition~\ref{prop:cases}, there is a holomorphic vector subbundle
  $\mathcal{V}$ of the trivial bundle $\CP^1\times \mathcal{H}$ whose fiber
  $\mathcal{V}_\mu$ over generic $\mu \in \C_*\subset \CP^1$ coincides with
  the space of $\nabla^\mu$--parallel sections with trivial monodromy.  The
  bundle $\mathcal{V}$ has rank~2 in Case~II and~IIIb and rank~4 in Case~IIIa.
\end{Lem}
\begin{proof}
  We prove the existence of $\mathcal{V}$ by applying
  Proposition~\ref{pro:holomorphic_family} to the operators
  \eqref{eq:holomorphic_families} acting on the finite dimensional space
  $\mathcal{H}$. For doing this it remains to check that, over generic $\mu\in
  \C_*$, the space of $\dbar^\mu$--holomorphic sections (or, which is
  equivalent by \eqref{eq:symmetry_nabla_mu}, the space of
  $\del^\mu$--anti--holomorphic sections) contained in $\mathcal{H}$ coincides
  with the space of $\nabla^\mu$--parallel sections.

  This is immediately clear in Case~IIIa of Proposition~\ref{prop:cases}
  because, for every $\mu\in \C_*$, the connection $\nabla^\mu$ is then
  trivial and so is the induced holomorphic structure $\dbar^\mu =
  (\nabla^\mu)^{(0,1)}$.  Therefore, all $\dbar^\mu$--holomorphic sections are
  $\nabla^\mu$--parallel and $\mathcal{V}$ is a holomorphic bundle of rank~4.
  In Case~II, for generic $\mu\in \C_*$, the trivial bundle $V$ over the torus
  has a splitting $V=E_1 \oplus E_{h_1} \oplus E_{h_2}$ into
  $\nabla^\mu$--parallel subbundles with the property that the connection
  induced by $\nabla^\mu$ on the rank 2 bundle $E_1$ is trivial while parallel
  sections of the line bundles $E_{h_l}$, $l=1,2$ have nontrivial monodromy
  $h_l\in \Hom(\Gamma,\C_*)$.  As above, every $\dbar^\mu$--holomorphic
  section of the trivial subbundle $E_1$ is $\nabla^\mu$--parallel.  Every
  $\dbar^\mu$--holomorphic section of $E_{h_l}$ is of the form $\psi_l f_l$,
  where $\psi_l$ is a $\nabla^\mu$--parallel section of $E_l$ and $f_l$ is a
  holomorphic complex function with monodromy $h^{-1}_l$. But such a section
  $\psi_l f_l$ is never contained in $\mathcal{H}$: if the quotient of two
  holomorphic sections of $V/L$ is complex, it is constant unless the Hopf
  field $Q$ is trivial (which is impossible, because a torus of Case~II cannot
  be super conformal).  Thus, in Case~II, the holomorphic subbundle
  $\mathcal{V}$ of the trivial bundle $\CP^1\times \mathcal{H}$ has rank~2.

  Proving the existence of $\mathcal{V}$ in Case~IIIb of
  Proposition~\ref{prop:cases} is slightly more involved: for generic $\mu\in
  S^1\subset \C_*$ there is a $\nabla^\mu$--parallel section $\psi$ with
  trivial monodromy and a parallel section $\varphi$ together with $t\in
  \Hom(\Gamma,\C)$ such that $\gamma^*\varphi = \varphi + \psi \, t_\gamma$
  for all $\gamma\in \Gamma$.  Every $\dbar^\mu$--holomorphic section $\tilde
  \psi$ is then of the form $\tilde \psi = \psi (f_1 + j f_2) + \varphi (g_1 +
  j g_2)$ where $f_1, f_2, g_1$ and $g_2$ are complex holomorphic functions.
  Such a section has trivial monodromy if and only if $g_1$, $g_2$ are
  constant and $\gamma^* f_1 = f_1 - g_1\, t_\gamma$ and $\gamma^* f_2 = f_2 -
  g_2\, \bar t_\gamma$ for all $\gamma\in \Gamma$. By the same argument as in
  Case~II, the section $\tilde \psi$ is in $\mathcal{H}$ if and only if $f_1$
  and $f_2$ are constant and $g_1=g_2=0$.  This shows that in Case~IIIb the
  bundle $\mathcal{V}$ has rank~2.
\end{proof}

\begin{Lem}\label{lem:caseII}
  A constrained Willmore torus that belongs to \textbf{Case~II} of
  Proposition~\ref{prop:cases} admits a polynomial Killing field $\xi$ which,
  for generic $\mu$, has a 2--dimensional kernel and two different
  non--trivial eigenvalues.
\end{Lem}
\begin{proof}
  The existence of $\xi$ is proven by similar arguments as in Case~I. For
  generic $\mu\in \C_*$, the trivial bundle $V$ has a splitting $V=E_1 \oplus
  E_{h_1} \oplus E_{h_2}$ into $\nabla^\mu$--parallel subbundles.  The space
  of $\nabla^\mu$--parallel sections of $\End_\C(V)$ is now 6--dimensional,
  because $E_1$ is a trivial subbundle.  Using
  Proposition~\ref{pro:holomorphic_family} we construct a subbundle
  $\mathcal{U}$ of the trivial bundle $\CP^1\times \Gamma(\End_\C(V))$ whose
  fiber $\mathcal{U}_\mu$ over generic $\mu$ coincides with the space of
  $\nabla^\mu$--parallel sections. For this we have to check that generically
  the space of $\dbar^\mu$--holomorphic sections of $\End_\C(V)$ (or,
  equivalently, that of $\partial^\mu$--anti--holomorphic sections) is also
  6--dimensional.  Assuming that this was not the case would force the
  multipliers $h_1^\mu$ and $h_2^\mu$ to be holomorphically equivalent for
  every $\mu$ which is impossible because, as in Case~I, this would imply
  $h_1^\mu=h_2^\mu$ for all $\mu\in S^1$.

  A global meromorphic section $\xi$ of the holomorphic rank~6 bundle
  $\mathcal{U}$ with a single pole at $\mu=\infty$ is a polynomial Killing
  field. For the reconstruction of $\Sigma_{hol}$ from $\xi$ it is preferable
  to have a polynomial Killing field $\tilde \xi$ that, for every $\mu$,
  vanishes on the trivial $\nabla^\mu$--parallel subbundle~$E_1$. We construct
  such $\tilde \xi$ by using the holomorphic rank~2 bundle $\mathcal V$
  defined in Lemma~\ref{lem:case_polynomial_conserved_quantity} as well as the
  corresponding bundle $\mathcal{V}^\perp$ that belongs to the dual
  constrained Willmore immersion $f^\perp$. The latter is the holomorphic
  rank~2 subbundle of $\mathcal{H}^\perp=\{ \alpha \in \Gamma(V^*)\mid
  \nabla\alpha \in \Omega^1(L^\perp)\}$ whose fiber over generic $\mu\in \C_*$
  is the space of $(\nabla^\perp)^\mu$--parallel sections. The existence of
  $\mathcal{V}^\perp$ follows from
  Lemma~\ref{lem:case_polynomial_conserved_quantity} because, by
  \eqref{eq:gauge_trafo}, the connections $(\nabla^\perp)^\mu$ on the bundle
  $(V^*,-i)$ are gauge equivalent to the connections $(\tilde
  \nabla^\perp)^\mu = \nabla + (\mu-1)( Q^\perp_\circ)^{(1,0)}+
  (\mu^{-1}-1)(Q_\circ^\perp)^{(0,1)}$ dual to $\nabla^\mu$ such that all its
  holonomies have $1$ as an eigenvalue of geometric multiplicity~$2$.  Let
  $\psi_1$, $\psi_2$ be two linearly independent meromorphic sections of
  $\mathcal{V}$ and $\alpha^1$, $\alpha^2$ two linearly independent
  meromorphic sections of the holomorphic bundle $\tilde{\mathcal{V}}^\perp$
  that is image of $\mathcal{V}^\perp$ under the gauge transformation
  of~\eqref{eq:gauge_trafo}.

  For generic $\mu \in \C_*$, the sections $\psi_1(\mu,.)$, $\psi_2(\mu,.)$ of
  $(V,i)$ are pointwise linearly independent and span the rank~2 subbundle
  $E_1=\ker(H^\mu_p(\gamma)-\Id)\subset (V,i)$ where $\gamma\in
  \Gamma\backslash\{0\}$.  Similarly, for generic $\mu$, the sections
  $\alpha^1(\mu,.)$, $\alpha^2(\mu,.)$ of $(V^*,-i)$ are pointwise linearly
  independent and span the rank~2 subbundle $\ker((H^\mu_p(\gamma))^*-\Id)=
  \im(H^\mu_p(\gamma)-\Id)^\perp$ for $\gamma\in \Gamma\backslash\{0\}$, where
  we use the identification of Section~\ref{sec:nabla_mu} between $(V^*,-i)$
  and the complex dual space of $(V,i)$. The meromorphic family of
  $2\times2$--matrices $g_{kl}(\mu)= \<\alpha^k_\C(\mu,.), \psi_l(\mu,.)\>$
  (with $\alpha^k_\C$ denoting the complex part of the $\alpha^l$) is
  therefore invertible for generic $\mu$ and \[ \tilde \xi(\mu,p) :=
  \xi(\mu,p) - \sum_{kl} \xi(\mu,p)(\psi_k(\mu,p))\cdot g_{kl}^{-1}(\mu)\cdot
  \alpha^l_\C(\mu,p) \] defines a polynomial Killing field that, for every
  $\mu$, vanishes on the trivial subbundle~$E_1$.
\end{proof}

\begin{Lem}\label{lem:caseIIIa}
  If a constrained Willmore torus belongs to \textbf{Case~IIIa} of
  Proposition~\ref{prop:cases} it is super conformal or Euclidean minimal with
  planar ends and $\eta\equiv 0$.
\end{Lem}
\begin{proof}
  In Case~IIIa of Proposition~\ref{prop:cases}, the bundle $\mathcal V$
  constructed in Lemma~\ref{lem:case_polynomial_conserved_quantity} has
  rank~4.  Because a local holomorphic section $\psi^\mu$ of $\mathcal{V}$
  defined in a neighborhood $U$ of $\mu=\infty$ satisfies $\nabla^\mu
  \psi^\mu=0$ for all $\mu\in U\cap\C_*$, the fiber $\mathcal{V}_\infty$ over
  $\mu=\infty$ is a 4--dimensional space of solutions to
  \eqref{eq:asymptotic_DT}.  Thus, by Proposition~\ref{pro:asymptotic}, an
  immersion $f$ belonging to Case~IIIa has to be super conformal or Euclidean
  minimal with planar ends and $\eta\equiv 0$.
\end{proof}

\begin{Lem}\label{lem:caseIIIb}
  If a constrained Willmore torus belongs to \textbf{Case~IIIb} of
  Proposition~\ref{prop:cases} it is Euclidean minimal with planar ends and
  $\eta\equiv 0$.
\end{Lem}
\begin{proof}
  Firstly, we prove the existence of a nilpotent polynomial Killing field
  $\xi$ for the associated family $\nabla^\mu$ of a constrained Willmore torus
  that belongs to Case~IIIb.  We do this by using the holomorphic rank~2
  bundle $\mathcal V$ defined in
  Lemma~\ref{lem:case_polynomial_conserved_quantity} as well as the
  corresponding bundle $\mathcal{V}^\perp$ that belongs to the dual
  constrained Willmore immersion $f^\perp$ (which exists because, by
  \eqref{eq:gauge_trafo}, the connections $(\nabla^\perp)^\mu$ on the bundle
  $(V^*,-i)$ also have Jordan--holonomy): we define $\xi$ by
  \[\xi(\mu,p) =  \psi(\mu,p) \alpha_\C(\mu,p),\] 
  where $\psi$ is a non--trivial meromorphic section of $\mathcal{V}$ and
  $\alpha_\C$ the complex part of a non--trivial meromorphic section $\alpha$
  of the holomorphic bundle $\tilde{\mathcal{V}}^\perp$ that is image of
  $\mathcal{V}^\perp$ under the gauge transformation
  of~\eqref{eq:gauge_trafo}.  The Killing field $\xi$ is polynomial in $\mu$
  if $\psi$ and $\alpha$ are chosen holomorphic on $\C$ with a single pole at
  $\mu = \infty$.

  The polynomial Killing field $\xi$ thus constructed is nilpotent: for
  generic $\mu \in \C_*$, at every point $p\in T^2$ the elements of
  $\mathcal{V}_\mu\subset\Gamma(\mathcal{V})$ span the rank~2 subbundle
  $\im(R^\mu_p(\gamma))=\ker(R^\mu_p(\gamma))\subset (V,i)$ with
  $R^\mu_p(\gamma)= H^\mu_p(\gamma)-\Id$ denoting the nilpotent part of the
  Jordan--holonomy around a non--trivial cycle $\gamma\in
  \Gamma\backslash\{0\}$.  Similarly, the elements of $(\tilde
  {\mathcal{V}}^\perp)_\mu\subset\Gamma(\mathcal{V}^*)$ generically span the
  subbundle $\im((R^\mu_p(\gamma))^*)=\ker((R^\mu_p(\gamma))^*)$ of $(V^*,-i)$
  which, under the identification of Section~\ref{sec:nabla_mu} between
  $(V^*,-i)$ and the complex dual space to $(V,i)$, coincides with the
  subbundle $\im(R^\mu_p(\gamma))^\perp =\ker(R^\mu_p(\gamma))^\perp$.

  We show now that the existence of $\xi$ leads to a contradiction unless $f$
  is Euclidean minimal with planar ends and $\eta\equiv 0$.  Because $\xi$
  does not vanish identically, on the non--empty open set $U$ on which
  $V=L\oplus \check L$ with $\check L=\ker(A_\circ)$, see
  Section~\ref{sec:asymptotics}, it takes the form $\xi = \sum_{l=0}^d \xi^l
  \mu^l$ with highest coefficient $\xi^d\in \Gamma_U(\End_\C(V))$ that is not
  identically zero.  Comparing coefficients in $\nabla^\mu \xi=0$ implies that
  $\xi^d\in \Gamma_U(\End_\C(V))$ satisfies
  \begin{align}
    \label{eq:coeff}
    [ A_\circ^{(1,0)} ,\xi^d ] = 0 \quad \textrm{ and } \quad \nabla\xi^d = -
    [ A_\circ^{(1,0)},\xi^{(d-1)} ] + [ A_\circ^{(0,1)}, \xi^d ].
  \end{align}

  Assume now that $\check L=\ker(A_\circ)$ is non--constant.  The bundles $L$
  and $\check L$ can then be trivialized by nowhere vanishing sections
  $\psi\in \Gamma(L)$ and $\check \psi \in \Gamma(\check L)$ with $J\psi =
  \psi i$ and $\check J \check \psi =\check \psi i$, where $\check J$ is the
  complex structure on $\check L$ occurring in the mean curvature sphere
  $\check S$ of $\check L$, see \eqref{eq:mcs_tilde}.  With respect to the
  frame $\psi$, $\psi j$, $\check \psi$, $\check \psi j$ on $U$ we then have
  \begin{align}
    \label{eq:constant_xi}
    A_\circ^{(1,0)} =
    \begin{pmatrix}
      0 & a \, dz & 0 & 0 \\ 0 & 0 & 0 & 0 \\ 0 & 0 & 0 & 0 \\ 0 & 0 & 0 & 0
    \end{pmatrix} \qquad \textrm{ and therefore } \qquad \xi^d=
    \begin{pmatrix}
      \xi_{11} & \xi_{12} & \xi_{13} & \xi_{14} \\
      0 & \xi_{11} & 0 & 0 \\
      0 & \xi_{32} & \xi_{33} & \xi_{34} \\
      0 & \xi_{42} & \xi_{43} & \xi_{44}
    \end{pmatrix},
  \end{align}
  where the form of $\xi^d$ follows from the first part of \eqref{eq:coeff}
  because $a$ doesn't vanish on~$U$.  The fact that $\xi$ is nilpotent implies
  that $\xi_{11}$ vanishes identically.

  As in the proof of Lemma~\ref{lem:caseIIIa}, the elements in the fiber
  $\mathcal{V}_\infty$ over $\infty$ of the rank~2 bundle $\mathcal{V}$ are
  solutions to \eqref{eq:asymptotic_DT}.  Because an immersion that belongs to
  Case~IIIb is not super conformal and $\check L$ is assumed to be
  non--constant such that the immersion is not Euclidean minimal with
  $\eta\equiv 0$, by Proposition~\ref{pro:asymptotic} the space of solutions
  to \eqref{eq:asymptotic_DT} is at most 2--dimensional and therefore
  coincides with $\mathcal{V}_\infty\subset \Gamma(V)$.  By construction of
  $\xi$, its highest order coefficient $\xi^d$ vanishes on the elements in
  $\mathcal{V}_\infty$ and pointwise takes values in the span of the sections
  in $\mathcal{V}_\infty$.

  As we have seen in the proof of Lemma~\ref{lem:fundamental_nonconstant},
  there is a dense open subset of $U$ on which the solutions to
  \eqref{eq:asymptotic_DT} span the $-i$--eigenspace of $\check S$. This
  eigenspace is the complex span of $\check \psi$ and $\psi - \check \psi
  \frac k2 b$ with $b$ given by $\check B\psi = \check \psi j b$ for $\check
  B$ as in \eqref{eq:mcs_tilde}. The function $b$ is nowhere vanishing,
  because $\check B$ vanishes at a point $p\in U$ if and only if $\eta$
  vanishes at $p$ (because, by \eqref{eq:A_Q_dagger}, $(A_\circ)_{|p}$
  commutes with $S_p$ iff $B_p=0$ which, by \eqref{eq:B_and_tilde_B}, is
  equivalent to $\check B_p=0$) but $\eta$ has no zeros at all since on a
  torus the holomorphic quadratic differential $\delta\eta$ has none (if it
  had one, $\eta$ had to vanish identically which is impossible by
  Lemma~\ref{lem:fundamental_nonconstant}, because in the Willmore case with
  $\eta \equiv 0$ the space of solutions to \eqref{eq:asymptotic_DT} is at
  most 1--dimensional unless the immersion is Euclidean minimal).

  That $\xi^d$ vanishes on the sections of $\mathcal{V}_\infty$ and therefore
  on the span of $\check S$ thus implies
  \begin{align*}
    \xi^d =
    \begin{pmatrix}
      0 & \xi_{12} & 0 & 0 \\
      0 & 0 & 0 & 0 \\
      0 & \xi_{32} & 0 & 0 \\
      0 & \xi_{42} & 0 & 0 \\
    \end{pmatrix}.
  \end{align*}
  In particular, because $\xi^d\check\psi j = 0$, the second part of
  \eqref{eq:coeff} yields
  \begin{align*}
    -\xi^d\, \check\delta\check\psi j = (\nabla\xi^d)\check \psi j = - [
    A^{(1,0)}_\circ, \xi^{(d-1)} ]\check \psi j + [ A^{(0,1)}_\circ, \xi^d
    ]\check \psi j,
  \end{align*}
  where $\check \delta$ denotes the derivative $\check \delta=
  \pi_L\nabla_{|\check L}$ of $\check L$ with $\pi_L$ the projection to
  $L\cong V/\check L$. Because the right hand side of this equation takes
  values in $L$ we obtain $\xi_{32}=\xi_{42}=0$ and, because $\xi^d$ takes
  values in the $-i$--eigenbundle of $\check S$, $\xi_{12}=0$. This is
  impossible as it contradicts the assumption that $\xi^d$ does not vanish
  identically on $U$.

  The assumption that $\check L=\ker(A_\circ)$ is non--constant thus leads to
  a contradiction such that $\check L=\ker(A_\circ)$ has to be constant.
  Proposition~\ref{pro:cases_harmonic} below shows that a constrained Willmore
  torus which is not super conformal and has constant $\check L=\ker(A_\circ)$
  belongs to Case~III if and only if it is Euclidean minimal with planar ends.
  (Alternatively, one could directly prove, by similar arguments as in case of
  non--constant $\check L=\ker(A_\circ)$, that if $\check L=\ker(A_\circ)$ is
  constant the polynomial Killing field constructed above can only exists if
  $\eta\equiv 0$ and the immersion is Euclidean minimal with planar ends.)
\end{proof}

\subsection{Proof of the main theorem}\label{sec:proof_main}
For a constrained Willmore torus that belongs to Case~I or~II of
Proposition~\ref{prop:cases}, Lemmas~\ref{lem:caseI} and \ref{lem:caseII}
imply the existence of a polynomial Killing field~$\xi$. This gives rise to a
Riemann surface parametrizing the non--trivial eigenlines of~$\xi$, see
Lemma~\ref{lem:holonomy_curve}. Because for every $\mu\in \C_*$ the polynomial
Killing field $\xi$ commutes with all holonomies of $\nabla^\mu$, the
eigenlines of $\xi$ are also eigenlines of the holonomies and the uniqueness
part of Lemma~\ref{lem:holonomy_curve} implies that the Riemann surface
constructed using $\xi$ coincides with the holonomy spectral curve
$\Sigma_{hol}$.  The fact that $\xi$ is polynomial in $\mu$ implies that
$\Sigma_{hol}$ has finite genus and can be compactified by adding points over
$\mu=0$ and $\mu=\infty$.

By Theorem~\ref{the:biholomorphicity}, the holonomy spectral curve
$\Sigma_{hol}$ is a dense subset of the multiplier spectral curve
$\Sigma_{mult}$ which has two ends, see Section~\ref{sec:multiplier_spec}.
Thus, two of the points needed to compactify $\Sigma_{hol}$ correspond to the
two ends of the multiplier spectral curve~$\Sigma_{mult}$ while the other
added points are contained in the complement $\Sigma_{mult}\backslash
\Sigma_{hol}$ of the holonomy spectral curve inside the multiplier spectral
curve.  Because the two ends are interchanged under the anti--holomorphic
involution $\rho$ which covers the involution $\mu\mapsto 1/\bar \mu$, one of
the ends corresponds to a point over $\mu=0$ and the other to a point over
$\mu=\infty$.  To see that both ends of $\Sigma_{mult}$ are branch points of
$\mu$, note that when $\sigma$ tends to one of the ends the corresponding
multipliers go to infinity in $\Hom(\Gamma,\C^2) \cong T^2\times \R^2$ which
cannot happen on a single sheet of the covering $\mu\colon
\Sigma_{hol}\rightarrow \C_*$ only, because the holonomies $H^\mu(\gamma)$
itself can be represented as $\SL(4,\C)$--matrices.

For immersions belonging to Case~II of Proposition~\ref{prop:cases} this shows
that the two points corresponding to the ends of $\Sigma_{mult}$ are the only
points needed to compactify $\Sigma_{hol}$, because $\mu\colon
\Sigma_{hol}\rightarrow \C_*$ is a 2--fold branched covering.  In particular,
$\Sigma_{hol}$ then coincides with $\Sigma_{mult}$.  In Case~I, the map
$\mu\colon \Sigma_{hol}\rightarrow \C_*$ is a 4--fold branched covering and,
because $\mu$ is branched at the two ends of $\Sigma_{mult}$, the complement
$\Sigma_{mult}\backslash \Sigma_{hol}$ of the holonomy spectral curve inside
the multiplier curve consists of at most four points.

As we have seen in Section~\ref{sec:nabla_mu}, all super conformal tori belong
to Case~IIIa of Proposition~\ref{prop:cases}.  Euclidean minimal tori with
planar ends belong to Case~III, because the holonomy of $\nabla^\mu$ for
$\mu\in S^1$ is of Jordan type with eigenvalue~1 and off--diagonal part
related to the translational periods of $*df$, see \cite{LPP05}.

Conversely, Lemma~\ref{lem:caseIIIa} shows that a constrained Willmore torus
that belongs to Case~IIIa of Proposition~\ref{prop:cases} is either super
conformal or Euclidean minimal with planar ends and Lemma~\ref{lem:caseIIIb}
shows that Case~IIIb is only possible for Euclidean minimal tori with planar
ends.

The fact that all constrained Willmore tori with non--trivial normal bundle
$\perp_f$ belong to Case~III follows from Proposition~\ref{prop:cases}.  A
constrained Willmore torus with trivial normal bundle $\perp_f$ that is not
Euclidean minimal with planar ends belongs to Case~I or II, because otherwise
it had to be super conformal which is impossible, because super conformal tori
have non--trivial normal bundle, see \eqref{eq:willmore_functional}.\qed

\section{The Harmonic Case}
\label{sec:harmonic}

We discuss a special class of constrained Willmore surfaces related to
harmonic maps to~$S^2$. It includes CMC surfaces in $\R^3$ and $S^3$,
Hamiltonian stationary Lagrangian surfaces in $\C^2\cong \H$, and Lagrangian
surfaces with conformal Maslov form in $\C^2\cong \H$.  Constrained Willmore
tori of this class belong to Cases~II or~III of Theorem~\ref{the:mainA}. More
precisely, they belong to Case~II if and only if the appendant harmonic map to
$S^2$ is non--conformal. Then, the harmonic map itself admits a spectral curve
\cite{PS89,Hi} which is shown to coincide with the constrained Willmore
spectral curve studied above.

\subsection{Main theorem of the section}
The following theorem will be proven in Section~\ref{sec:proof_the_b}.

\begin{The}\label{the:mainB}
  If a conformal immersion $f\colon M\rightarrow S^4$ of a Riemann surface $M$
  admits a point $\infty\in S^4$ at infinity for which one factor of the
  (Euclidean) Gauss map \[ M \rightarrow Gr^+(2,4)=S^2\times S^2\] of $f$
  seen as an immersion into $\mathbb{R}^4= S^4\backslash \{\infty\}$ is
  harmonic, then $f$ is constrained Willmore.  If $M=T^2$ is a torus there is
  a Lagrange multiplier $\eta$ such that $f$ belongs~to Case~II or~III of
  Theorem~\ref{the:mainA}, depending on whether the harmonic factor is
  non--conformal or conformal. In the non--conformal case, the harmonic map
  spectral curve coincides with the constrained Willmore spectral curve.
\end{The}

As pointed out by Fran~Burstall~\cite{Bu07}, the property that the Gauss map
of an immersion into~$\R^4$ has a harmonic factor is equivalent to
holomorphicity or anti--holomorphicity of its mean curvature vector, see
Section~\ref{sec:euclidean}. The first part of Theorem~\ref{the:mainB} has
been generalized by Burstall~\cite{Bu07} who proved that every immersion into
a 4--dimensional space form with (anti--)holomorphic mean curvature vector is
constrained Willmore.

It should be noted that, conversely, every constrained Willmore torus $f\colon
T^2\rightarrow S^4$ that belongs to Case~III of Theorem~\ref{the:mainA} (and
hence is super conformal or Euclidean minimal with planar ends) admits a point
$\infty \in S^4$ at infinity for which the Gauss map of $f$ seen as an
immersion into $\R^4=S^4\backslash\{\infty\}$ has a conformal factor, see
Section~\ref{sec:euclidean}.  It is also worth noting that (following from
Corollary~\ref{cor:small_willmore_energy} together with
Lemma~\ref{lem:harmonic}) every constrained Willmore torus that belongs to
Case~II and has Willmore functional $\mathcal{W}< 8\pi$ admits a point $\infty
\in S^4$ for which a factor of the Euclidean Gauss map is harmonic and
non--conformal.

\subsection{Surfaces in Euclidean 4--space $\R^4=\H$}\label{sec:euclidean}
The immersions studied in this section come with a distinguished
point~$\infty\in S^4$ at infinity in the conformal 4--sphere. In the following
we therefore work with a fixed trivialization $V\cong \H^2$ and identify
$\R^4=\H$ with $\HP^1\backslash \{\infty\}$ via $x\in \H\mapsto [ (x,1)]$.
The relation between the (M\"obius geometric) mean curvature sphere congruence
of an immersion $f\colon M \rightarrow \H= \HP^1\backslash\{\infty\}$ and its
(Euclidean) Gauss map is as follows (see Chapter~7 of \cite{BFLPP02} for
details): the Gauss map of a conformal immersion $f\colon M \rightarrow \H$ is
represented by the \emph{left} and \emph{right normal vectors} $N$, $R\colon M
\rightarrow S^2\subset \Im\H$ of $f$ which are defined by
\[ *df = N df = - df R. \] The mean curvature vector $\mathcal{H} = \frac12
\Tr (\sff)$ can be expressed in terms of $N$ and $R$ by
\begin{equation}
  \label{eq:mcvector_and_left_right_normals}
  dN'= \frac12(dN -N{*}dN)= -df H \qquad \text{and} \qquad dR'= \frac12(dR -
  R{*}dR)= - H df 
\end{equation}
where $H = \bar{\mathcal{H}} N = R \bar{\mathcal{H}}$.  The mean curvature
sphere of the immersion $L=
\begin{pmatrix}
  f \\ 1
\end{pmatrix}\H$ is then
\begin{equation}
  \label{eq:euclidean_form_of_S}
   S = \Ad 
   \begin{pmatrix}
     1 & f \\
     0 & 1
   \end{pmatrix}
   \begin{pmatrix}
     N & 0 \\
     H & -R
   \end{pmatrix}
\end{equation}
and the Hopf fields of $f$ take the form
\begin{equation}
  \label{eq:AQ_in_Euclidean_chart}
  2{*}A = \Ad 
\begin{pmatrix}
  1 & f \\
  0 & 1 
\end{pmatrix}
\begin{pmatrix}
  0 & 0 \\
  w & dR''
\end{pmatrix}
\quad \textrm{ and } \quad 
   2{*}Q = \Ad 
\begin{pmatrix}
  1 & f \\
  0 & 1 
\end{pmatrix}
\begin{pmatrix}
  dN'' & 0 \\
  w+dH & 0
\end{pmatrix}
\end{equation}
with $dN''=\frac12(dN +N{*}dN)$, $dR''=\frac12(dR +R{*}dR)$ and $w=
\tfrac12(-dH - R{*}dH + H{*}dN'')$.

An immersion $f\colon M \rightarrow \R^4= S^4\backslash\{\infty\}$ is
Euclidean minimal if and only if its left and right normal vectors~$N$ and~$R$
with respect to $\infty$ are both anti--holomorphic, that is, if $dN'=0$ and
$dR'=0$, see \eqref{eq:mcvector_and_left_right_normals}.  An immersion
$f\colon M \rightarrow S^4$ is super conformal, i.e., satisfies $A\equiv 0$ or
$Q\equiv 0$, if and only if with respect to one and therefore every point
$\infty$ at infinity its left normal $N$ is holomorphic $dN''=0$ or its right
normal $R$ is holomorphic $dR''=0$, see \eqref{eq:AQ_in_Euclidean_chart} and
Lemma~22 of \cite{BFLPP02}.

The following lemma gives a quaternionic characterization of harmonic maps
into $S^2$.
\begin{Lem}\label{lem:harmonic_map_into_2sphere}
  A map $N\colon M\rightarrow S^2\subset \Im(\H)$ into the 2--sphere is
  harmonic if and only if the 1--form $dN'=\frac12(dN -N{*}dN)$ or,
  equivalently, $dN''=\frac12(dN +N{*}dN)$ is closed.
\end{Lem}

\begin{proof}
  A map $N\colon M \rightarrow S^2$ is harmonic if the 2--form
  \[ d{*}dN=d(N dN' - N dN'') = dN'\wedge dN' + N d(dN') - dN''\wedge dN'' - N
  d(dN''), \] is normal to $S^2$. Both $dN'\wedge dN'$ and $dN''\wedge dN''$
  are normal and one can easily check that $d(dN')$ and $d(dN'')$ are
  tangential.  The tangential part of the form $d{*}dN$ is therefore $N
  d(dN')- N d(dN'')=2 N (dN') = -2N(dN'')$ which proves the statement.
\end{proof}

Using Lemma~\ref{lem:harmonic_map_into_2sphere} we show now that the Gauss map
of an immersion $f\colon M \rightarrow \R^4 =\H$ into Euclidean 4--space has
a harmonic factor if and only if its mean curvature vector $\mathcal{H}$ is
(anti--)holomorphic. More precisely, $f$ has a harmonic left normal $N$ if
and only if its mean curvature vector is a holomorphic section of the normal
bundle, that is, $*\nabla^\perp \mathcal{H}= N \nabla^\perp \mathcal{H}$ with
$\nabla^\perp$ denoting the normal connection of $f$.  Analogously,
harmonicity of the right normal $R$ is equivalent to anti--holomorphicity of
$\mathcal{H}$, that is, to $*\nabla^\perp \mathcal{H}= - N \nabla^\perp
\mathcal{H}$. We only prove the second statement: by
\eqref{eq:mcvector_and_left_right_normals}, $R$ is harmonic if and only if
$*dH=dH N$ or, equivalently, $*d\bar H=-N d\bar H$.  Using again
\eqref{eq:mcvector_and_left_right_normals}, $HN=RH$ implies $dH N -R dH = dR''
H - H dN''$ which shows that $*d\bar H=-N d\bar H$ is equivalent to
$*\nabla^\perp \bar H= - N \nabla^\perp \bar H$. By $\mathcal{H} = N \bar H$
this is equivalent to $*\nabla^\perp \mathcal{H}= - N \nabla^\perp
\mathcal{H}$.

\subsection{M\"obius geometric characterization of the harmonic case}\label{sec:moebius_char}
The following lemma gives a characterization in terms of the Hopf fields $A$
and $Q$ of the property that there is a Euclidean subgeometry in which the
Euclidean Gauss map has a harmonic factor.

\begin{Lem} \label{lem:harmonic} Let $f\colon M \rightarrow \H=
  \HP^1\backslash\{\infty\}$ be a conformal immersion. Then:
  \begin{itemize}
  \item[a)] The right normal vector $R\colon M \rightarrow S^2$ of $f$ with
    respect to the point $\infty$ at infinity is harmonic if and only if $f$
    is constrained Willmore and admits a 1--form $\eta\in\Omega^1(\mathcal R)$
    such that $2{*}A_\circ=2{*}A + \eta$ is closed and vanishes on the line
    corresponding to $\infty$.
  \item[b)] The left normal vector $N\colon M \rightarrow S^2$ of $f$ with
    respect to $\infty$ is harmonic if and only if $f$ is constrained Willmore
    and admits $\eta$ such that the form $2{*}Q_\circ=2{*}Q + \eta$ is closed
    and takes values in the line described by $\infty$.
  \end{itemize}
\end{Lem}

\begin{proof} It is sufficient to prove a), because b) follows by passing to
  the dual immersion $f^\perp$. By Proposition~15 of \cite{BFLPP02}, the
  differential of $2{*}A$ is in $\Omega^2(\mathcal R)$, i.e., vanishes on $L$
  and takes values in $L$, such that
  \begin{equation}
    \label{eq:nabla_A}
     d^\nabla(2{*}A) = \Ad
  \begin{pmatrix}
    1 & f \\ 0 & 1 
  \end{pmatrix}
  \begin{pmatrix}
    0 & 0 \\ dw  & 0
  \end{pmatrix} = 
  \begin{pmatrix}
    f dw & -f dw f  \\ dw & -dw f 
  \end{pmatrix}.
  \end{equation}
  Using \eqref{eq:AQ_in_Euclidean_chart} this implies $w \wedge df+ d(dR'')=0$
  and, by Lemma~\ref{lem:harmonic_map_into_2sphere}, the right normal vector
  $R$ is harmonic if and only if $*w= w N$.

  Every 1-form $\eta \in \Omega^1(\mathcal R)$ can be written as
  \begin{equation}\label{eq:eta}
      \eta = \Ad
  \begin{pmatrix}
    1 & f \\ 0 & 1 
  \end{pmatrix}
  \begin{pmatrix}
    0 & 0 \\ \hat \eta & 0
  \end{pmatrix} 
  \quad \textrm{ and  }  \quad 
  d^\nabla \eta = \Ad
  \begin{pmatrix}
    1 & f \\ 0 & 1 
  \end{pmatrix}
  \begin{pmatrix}
    0 & 0 \\ d\hat \eta  & 0
  \end{pmatrix}\in \Omega^2(\mathcal R)
  \end{equation}
  if and only if $*\hat \eta=-R\hat \eta = \hat \eta N$. In particular, the
  form $2{*}A+\eta$ is closed if and only if $*\hat \eta=-R\hat \eta = \hat
  \eta N$ and $dw +d\hat\eta=0$.

  This proves the lemma: the form $2{*}A+\eta$ vanishes on $\infty=[(1,0)]$ if
  and only if $\hat \eta=-w$ and, because $w$ always satisfies $*w=-Rw$, the
  form $2{*}A+\eta$ is closed if and only if $*w=wN$ which, as we have seen
  above, is equivalent to $R$ being harmonic.
\end{proof}

For a Willmore immersion $f\colon M\rightarrow S^4$ with $\eta\equiv 0$ that
is not super conformal, the following are equivalent (see Section~11.2 of
\cite{BFLPP02} for details):
\begin{itemize}
\item $\check L = \ker(A_\circ)=\ker(A)$ is constant,
\item $\hat L = \im(Q_\circ)=\im(A)$ is constant, and
\item $f$ is Euclidean minimal with planar ends in
  $\R^4=S^4\backslash\{\infty\}$ for some $\infty\in S^4$.
\end{itemize}
In particular, we then have $\infty = \check L = \hat L$. This immediately
follows from the fact that, firstly, $\check L$ and $\hat L$ are invariant
under $S$ if $\eta\equiv 0$ and, secondly, if there is a point~$\infty$
contained in all mean curvature spheres the immersion is Euclidean minimal
with planar ends in $\R^4=S^4\backslash\{\infty\}$ and vice versa.

\subsection{Willmore bundles of rank~1 and the harmonic map spectral
  curve}\label{sec:harmonic_spec}
We review the quaternionic approach to the spectral curve of harmonic maps
from tori to $S^2$ as developed in Sections~6.1 to 6.3 of~\cite{FLPP01}. This
allows a short proof of the prototype result mentioned in the introduction.

A flat connection~$\nabla$ on a quaternionic vector bundle $W$ with complex
structure $J$ over a Riemann surface $M$ is a \emph{Willmore connection} if
the Hopf field $A = \tfrac 14 (J\nabla J + {*}\nabla J)$ satisfies
$d^\nabla(2{*} A)=0$, or equivalently, if for every parameter $\mu\in \C_*$
the complex connection
\begin{equation}
  \label{eq:associated_willmoreconnection}
  \nabla^\mu = \nabla + (\mu-1) A^{(1,0)} + (\mu^{-1}-1) A^{(0,1)}
\end{equation}
on the complex bundle $(W,i)$ is flat.

The Hopf fields $A$ and $Q= \tfrac 14 (J\nabla J - {*}\nabla J)$ of a Willmore
connection $\nabla$ are holomorphic sections $A\in H^0(K\End_-(W))$ and $Q\in
H^0(\End_-(W)K)$, see \cite{FLPP01}.  In case $W$ is a quaternionic line
bundle, the complex line bundle $\End_+(W)$ is canonically trivial and $AQ\in
H^0(K^2)\cong H^0(K^2\End_+(W))$ is a holomorphic quadratic differential.  In
particular, if $M=T^2$ is a torus the holomorphic quadratic differential $AQ$
is either nowhere vanishing or vanishes identically.

\begin{Lem} \label{lem:rank_1_willmore_over_t2} Let $\nabla$ be a trivial
  connection on a quaternionic line bundle $W$ with complex structure $J$ over
  a Riemann surface $M$ and define $N\colon M \rightarrow S^2$ by $J\psi=\psi
  N$ for $\psi$ a non--trivial parallel section. Then $\nabla$ is a Willmore
  connection if and only if $N$ is harmonic.  In case $\nabla$ is Willmore and
  $M=T^2$ is a torus, either
  \begin{itemize}
  \item $AQ$ is nowhere vanishing and $N$ is non--conformal with $\deg(N)=0$
    or
  \item $AQ\equiv 0$ and $N$ is holomorphic or anti--holomorphic depending on
    whether $A\equiv 0$ or $Q\equiv 0$. In particular $\deg(N)\neq 0$ unless
    $N$ is constant.
  \end{itemize}
\end{Lem}
\begin{proof}
  That $\nabla$ is Willmore if and only if $N$ is harmonic follows from
  $(2{*}A)\psi = -\psi dN'$ and Lemma~\ref{lem:harmonic_map_into_2sphere}.  If
  $M=T^2$, the holomorphic quadratic differential $AQ$ either has no zeros at
  all or vanishes identically. Because $A\psi=\psi\frac12 N dN'$ and
  $Q\psi=\psi\frac12 N dN''$, the latter is equivalent to $N$ being conformal.
  The statement about the degree $\deg(N)$ of $N$ in the non--conformal case
  holds, because $A$ and $Q$ are then nowhere vanishing holomorphic sections
  of the bundles $K\End_-(W)$ and $\End_-(W)K$, respectively, and
  $\deg(\End_-(W))=2\deg(W)=2\deg(N)$.  In the conformal case the statement
  about the degree follows from $\deg(N)=\deg(W)=\frac1{2\pi}\int_M A\wedge *A
  - Q\wedge *Q$.
\end{proof}

The following lemma implies the existence of the harmonic map spectral curve
for non--conformal harmonic maps $N\colon T^2\rightarrow S^2$.
\begin{Lem} \label{lem:holonomy_not_trivial} Let $W$ be a quaternionic line
  bundle with complex structure $J$ and Willmore connection $\nabla$ over a
  torus $T^2$. In case $\nabla$ has non--trivial Hopf fields $A\not\equiv0$
  and $Q\not\equiv0$, there are only finitely many spectral parameters $\mu \in
  \C_*$ for which all holonomies of the flat connection $\nabla^\mu$ in
  \eqref{eq:associated_willmoreconnection} have eigenvalue $1$.
\end{Lem}
\begin{proof}
  Assume the statement was not true. Because the space $H^0(W)$ of
  $\nabla''$--holomorphic sections with trivial monodromy is finite
  dimensional, there are then $\mu_0,...,\mu_n\in \C_*$ and $\psi_{\mu_0},
  ...,\psi_{\mu_n}\in H^0(W)$ with $\nabla^{\mu_l}\psi_{\mu_l}=0$ for
  $l=0,...,n$ such that $\psi_{\mu_1},...,\psi_{\mu_n}$ are linearly
  independent over $\C$ while $\psi_{\mu_0},...,\psi_{\mu_n}$ are linearly
  dependent.  Because $Q\not\equiv 0$, it is impossible that a holomorphic
  section of $W$ is contained in the $\pm i$ eigenspaces of~$J$ such that both
  $(A\psi_{\mu_1})^{(1,0)},...,(A\psi_{\mu_n})^{(1,0)}$ and
  $(A\psi_{\mu_1})^{(0,1)},...,(A\psi_{\mu_n})^{(0,1)}$ are also linearly
  independent over~$\C$.  Moreover, there are $\lambda_l\in \C$ such that
  $\psi_{\mu_0} = \sum_{l=0}^n \psi_{\mu_l} \lambda_l$ and, using $\nabla
  \psi_{\mu_l} = (1-\mu_l) (A\psi_{\mu_l}))^{(1,0)}+(1-\mu^{-1}_l)
  (A\psi_{\mu_l})^{(0,1)}$, we obtain
  \[ \sum_{l=0}^n (A\psi_{\mu_l})^{(1,0)} (\mu_l-\mu_0) \lambda_l +
  \sum_{l=0}^n (A\psi_{\mu_l})^{(0,1)} (\mu_l^{-1}-\mu_0^{-1}) \lambda_l =
  0
  \] which is a contradiction, because $\mu_0\neq \mu_l$ for all $l=1$,...,$n$
  and $\lambda_l \neq 0$ for some $l$.
\end{proof}
  
The \emph{harmonic map spectral curve} of a non--conformal harmonic map
$N\colon T^2 \rightarrow S^2$ is the hyper--elliptic Riemann surface
$\mu\colon \Sigma_{harm} \rightarrow \C_*$ parametrizing the holonomy
eigenlines of the associated family \eqref{eq:associated_willmoreconnection}
of the corresponding trivial Willmore connection.  For the spectral curve to
be well defined one has to make sure that Lemma~\ref{lem:holonomy_curve} can
be applied to the holonomies of $\nabla^\mu$ around non--trivial loops, i.e.,
that for generic $\mu$ the holonomies have two different eigenvalues.  If this
was not the case, all holonomies had a double eigenvalue which had to be $1$
because $\nabla^\mu$ is a family of $\SL(2,\C)$--connections and
$\nabla^{\mu=1}$ is trivial.  But this is impossible by
Lemma~\ref{lem:holonomy_not_trivial}.  The same lemma shows that the harmonic
map spectral curve $\Sigma_{harm}$ has finite genus, because the number of
branch points of the projection $\mu\colon \Sigma_{harm} \rightarrow \C_*$ is
finite: the holonomies corresponding to branch points have $\pm 1$ as double
eigenvalues such that every branch point gives rise to a
$\nabla''$--holomorphic section with trivial monodromy defined on a 4--fold
cover of $T^2$.

\subsection{Proof of Theorem~\ref{the:mainB}}\label{sec:proof_the_b}
The first part of Theorem~\ref{the:mainB} immediately follows from
Lemma~\ref{lem:harmonic}.  In Section~\ref{sec:euclidean} we have seen that an
immersion $f\colon M\rightarrow \H$ with conformal left or right normal $N$ or
$R$ is super conformal or Euclidean minimal with planar ends.  Thus, an
immersions $f\colon T^2\rightarrow S^4$ of a torus whose left or right normal
with respect to some point $\infty\in S^4$ is conformal belongs to Case~III of
Theorem~\ref{the:mainA}. It therefore remains to show that $f$ belongs to
Case~II of Theorem~\ref{the:mainA} if there is $\infty \in S^4$ for which one
of the Euclidean normals is harmonic but non--conformal. This is proven
by Proposition~\ref{pro:cases_harmonic} below.

Firstly, we show that if a constrained Willmore immersion $f\colon T^2
\rightarrow S^4=\HP^1$ is neither super conformal nor Euclidean minimal with
planar ends and has constant $\hat L=\im(Q_\circ)$, the quaternionic
holomorphic line bundle $W=V/L$ carries a Willmore connection $\nabla$ which
is induced by $\infty=\hat L=\im(Q_\circ)$ and compatible $D=\nabla''$ with
the holomorphic structure $D$ on $V/L$.  (Passing to the dual constrained
Willmore immersion $f^\perp$ shows that, if $\check L = \ker(A_\circ)$ is
constant, the same is true for the holomorphic line bundle
$L^{-1}=V^*/L^\perp$.)

For an arbitrary conformal immersion $f\colon M\rightarrow \HP^1$, the
canonical projection $\pi$ to $V/L$ projects $v=(1,0) \in V\cong\H^2$ to a
holomorphic section $\varphi= \pi(v)\in H^0(V/L)$ that vanishes on the
discrete set $Z$ of points at which $f$ goes through $\infty = [(1,0)]$. Away
from $Z$, the complex structure $\tilde J\in \Gamma(\End(V/L))$ satisfies
$\tilde J\varphi=\varphi N$, where $N$ is the left normal of $f\colon
M\backslash Z \rightarrow \H= \HP^1\backslash\{\infty\}$.  The point $\infty$
induces a compatible connection $\nabla$ defined over $M\backslash Z$ by
setting $\nabla\varphi=0$.  This connection is Willmore if and only if $N$ is
harmonic.  By Lemma~\ref{lem:harmonic} this is equivalent to $f$ being
constrained Willmore with $\im( Q_\circ)=\infty$ for some Lagrange multiplier
$\eta$.

The following lemma shows that if $M=T^2$ is a torus and $N$ is
non--conformal, the set $Z$ is empty and the Willmore connection $\nabla$ is
globally defined.

\begin{Lem}\label{lem:no_ends}
  Let $f\colon T^2 \rightarrow S^4=\HP^1$ be a conformal immersion and denote
  $N$ and $R$ the left and right normals of $f$ seen as an immersion into
  $\H=\HP^1\backslash\{\infty\}$ for a point $\infty\in \HP^1$ at infinity.
  In case $N$ or $R$ is harmonic, it smoothly extends through the points of
  $T^2$ at which $f$ goes through $\infty$ to a harmonic map $T^2\rightarrow
  S^2$.  If this harmonic map is non--conformal, its degree is zero and $f$
  does not go through $\infty$.
\end{Lem}

In case $N$ or $R$ is conformal and non--constant, its extension
$T^2\rightarrow S^2$ has non--zero degree and $f$ can go through $\infty$.  In
the anti--holomorphic case when $f$ is minimal with planar ends in $\R^4\cong
S^4\backslash\{\infty\}$ it goes through~$\infty$ at the ends of the surface.
The degree of the extended left and right normals is then $\deg(N)=d/2-e$ and
$\deg(R)= -e- d/2$, where $d=\deg(\perp_f)$ denotes the degree of the normal
bundle and $e$ the number of ends.

\begin{proof}
  It is sufficient to prove the statement in the case that $N$ is harmonic
  (the case that $R$ is harmonic immediately follows by passing to the dual
  surface).  As above, let $\varphi= \pi(v)$ be the holomorphic section of
  $V/L$ obtained by projection of $v=(1,0)$. The left normal vector $N$ of
  $f\colon T^2\backslash Z \rightarrow \H=\HP^1\backslash\{\infty\}$ then
  satisfies $\tilde J\varphi=\varphi N$ and therefore, by
  Lemma~\ref{lem:singularity}, continuously extends through the set $Z$ of
  points at which $f$ goes through $\infty=[(1,0)]$. Because $N$ is continuous
  on $T^2$ and, by assumption, harmonic on $T^2\backslash Z$, it is a
  continuous solution to an elliptic equation and therefore smooth on all of
  $T^2$, see e.g.~\cite{He04}.  In particular, the second part of
  Lemma~\ref{lem:singularity} implies that $dN''_p=0$ for all~$p\in Z$.
  
  It remains to check that $Z=\emptyset$ if $N$ is non--conformal.  For this
  we equip the trivial bundle $E=M\times \H$ with a Willmore connection by
  setting $\nabla\psi=0$ and $J\psi=\psi N$ for $\psi(p)=(p,1)$.  Since $N$ is
  non--conformal, Lemma~\ref{lem:rank_1_willmore_over_t2} implies that $AQ$ is
  nowhere vanishing and $\deg(N)=0$.  In particular, $Q\psi=\psi\frac12 N
  dN''$ implies the set $Z$ is empty, because $dN''_p=0$ for all~$p\in Z$ but
  $Q$ is nowhere vanishing.
\end{proof}

The following lemma shows that, if $\im(Q_\circ)$ is constant, the
$\SL(2,\C)$--connections in the associated family
\eqref{eq:associated_willmoreconnection} of the induced Willmore connection on
$V/L$ are gauge equivalent to an invariant subbundle of the
$\SL(4,\C)$--connections in the constrained Willmore associated family
\eqref{eq:nabla_mu} on $V$.

\begin{Lem} \label{lem:2by2_to_4by4}Let $f\colon M \rightarrow S^4$ be a
  constrained Willmore immersion with constant $\im(Q_\circ)$.  For every
  spectral parameter $\mu\in \C_*$, the prolongation of a parallel section of
  the connection \eqref{eq:associated_willmoreconnection} in the associated
  family of the Willmore connection on $V/L$ induced by $\infty=\im(Q_\circ)$
  is parallel with respect to the connection~\eqref{eq:nabla_mu} on $V$.
\end{Lem}
\begin{proof} 
  As above, denote by $\varphi$ the holomorphic section of $V/L$ defined by
  projection $\pi(v)$ of $v=(1,0)$ with $\infty=[(1,0)]=\im(Q_\circ)$.  The
  prolongation $\psi$ of a holomorphic section $\tilde \psi=\varphi g$ of
  $V/L$ is then
  \[ \psi =
  \begin{pmatrix}
    1 \\ 0 
  \end{pmatrix} g + 
  \begin{pmatrix}
    f \\ 1 
  \end{pmatrix} \chi\] with $\chi$ defined by $dg+df\chi=0$.
  
  A section $\tilde \psi=\varphi g$ of $V/L$ is $\nabla^\mu$--parallel with
  respect to the connection \eqref{eq:associated_willmoreconnection} if
  \[ dg + \pi^{(1,0)}_N(N dN' g) \frac{\mu-1}2 + \pi^{(0,1)}_N(N dN' g)
  \frac{\mu^{-1}-1}2=0,\] where $\pi^{(1,0)}_N(v) = \tfrac12(v-Nv i)$ and
  $\pi^{(0,1)}_N(v) = \tfrac12(v+Nv i)$. Now $dN'=-df H$ yields $NdN'=df R H$
  and, because $dg+df \chi=0$, the function $\chi$ is given by
  \begin{equation}
    \label{eq:chi}
    \chi = \pi^{(0,1)}_R(RH g)\frac{\mu-1}2+ \pi^{(1,0)}_R(RH
    g)\frac{\mu^{-1}-1}2.
  \end{equation}

  By \eqref{eq:AQ_in_Euclidean_chart} and \eqref{eq:eta}, the fact that
  $Q_\circ$ takes values in the line corresponding to $\infty$ implies
  \[ 2{*}Q_\circ = \Ad
  \begin{pmatrix}
    1 & f \\ 0 & 1 
  \end{pmatrix}
  \begin{pmatrix}
    dN'' & 0 \\ 0 & 0 
  \end{pmatrix}.  \quad \textrm{ Therefore } \quad 2{*}A_\circ = \Ad
  \begin{pmatrix}
    1 & f \\ 0 & 1 
  \end{pmatrix}
  \begin{pmatrix}
    0 & 0 \\ -dH & dR'' 
  \end{pmatrix}
  \]
  and 
  \begin{equation}
  \label{eq:A_2by2_to_4by4_proof}
   A_\circ = \frac 12 \Ad
  \begin{pmatrix}
    1 & f \\ 0 & 1 
  \end{pmatrix}
  \begin{pmatrix}
    0 & 0 \\ -R dH & R dR'' 
  \end{pmatrix}.
  \end{equation}
  The derivative of the prolongation $\psi$ of a holomorphic section $\tilde
  \psi=\varphi g$ with respect to the connection \eqref{eq:nabla_mu} on $V$ is
  \begin{multline}
    \nabla^\mu \psi =
  \begin{pmatrix}
    f \\ 1 
  \end{pmatrix}d\chi +  
  \begin{pmatrix}
    f \\ 1 
  \end{pmatrix} \Big(\pi^{(0,1)}_R(-R dH g + R dR'' \chi)\frac{\mu-1}2+\\
  \pi^{(1,0)}_R(-R dH g + R dR''\chi)\frac{\mu^{-1}-1}2\Big).
  \end{multline}
  In case that $\tilde \psi$ is parallel with respect to the connection
  \eqref{eq:associated_willmoreconnection}, using $RH dg = -R H df \chi = R
  dR'\chi= R dR \chi - R dR''\chi$, differentiation of \eqref{eq:chi} yields
  \begin{multline}\label{eq:differential_chi}
    d\chi = \pi^{(0,1)}_R(R dH g - R dR'' \chi)\frac{\mu-1}2+
  \pi^{(1,0)}_R(R dH g - R dR''\chi)\frac{\mu^{-1}-1}2 + \\
  \pi^{(0,1)}_R(R dR \chi)\frac{\mu-1}2+
  \pi^{(1,0)}_R(R dR \chi)\frac{\mu^{-1}-1}2 + 
  (dR H g)\frac{\mu-1}4+
  (dR H g)\frac{\mu^{-1}-1}2.  
  \end{multline}
  For proving that the prolongation $\psi$ is a parallel section of $V$ with
  respect to \eqref{eq:nabla_mu} if the section $\tilde \psi$ of $V/L$ is
  parallel with respect to \eqref{eq:associated_willmoreconnection}, it
  therefore remains to check that the second line of
  \eqref{eq:differential_chi} vanishes.  By \eqref{eq:chi}
  \[ R dR \chi = \pi^{(1,0)}_R(dR H g)\frac{\mu-1}2+ \pi^{(0,1)}_R(dR H
  g)\frac{\mu^{-1}-1}2 \] we have
  \[  \pi^{(0,1)}_R(R dR \chi)\frac{\mu-1}2 = \pi^{(0,1)}_R(dR H
  g)\frac{\mu-1}2 \frac{\mu^{-1}-1}2 \] 
  \[ 
  \pi^{(1,0)}_R(R dR \chi)\frac{\mu^{-1}-1}2 = \pi^{(1,0)}_R(dR H
  g)\frac{\mu-1}2 \frac{\mu^{-1}-1}2 \] such that indeed, by $\frac{\mu-1}2
  \frac{\mu^{-1}-1}2 = \frac14(2-\mu-\mu^{-1})$, the second line of
  \eqref{eq:differential_chi} vanishes.
\end{proof}

\begin{Pro} \label{pro:cases_harmonic} Let $f\colon T^2 \rightarrow S^4$ be a
  constrained Willmore torus which is neither super conformal nor Euclidean
  minimal with planar ends and has the property that $\check L=\ker(A_\circ)$
  or $\hat L=\im(Q_\circ)$ is a constant point $\infty \in S^4$. Then the
  immersion $f$ belongs to Case~II of Theorem~\ref{the:mainA} and the harmonic
  map spectral curve of the harmonic left or right normal of $f\colon T^2
  \rightarrow \R^4=S^4\backslash \{\infty\}$ coincides with the constrained
  Willmore spectral curve of $f$.
\end{Pro}

\begin{proof}
  Passing to the dual immersion $f^\perp$ interchanges the property that
  $\ker(A_\circ)$ is constant with the property that $\im(Q_\circ)$ is
  constant. Moreover, $f$ and $f^\perp$ belong to the same case in the list of
  Theorem~\ref{the:mainA}, because the holonomy representations of the
  associated family $(\nabla^\perp)^\mu$ belonging to the dual constrained
  Willmore immersion $f^\perp$ are equivalent to the dual representations of
  the holonomy representations of $\nabla^\mu$, see \eqref{eq:gauge_trafo}.

  Assuming that $\ker(A_\circ)$ is constant we obtain that we are not in
  Case~I of Theorem~\ref{the:mainA}, because a constant kernel $\ker(A_\circ)$
  gives rise to a complex 2--dimensional space of sections with trivial
  monodromy of $V$ which are parallel for all $\nabla^\mu$.

  On the other hand, assuming that $\im(Q_\circ)$ is constant shows that we
  are in Case~II of Theorem~\ref{the:mainA}: because the immersion $f$ is
  neither super conformal nor Euclidean minimal with planar ends, its left
  normal with respect to $\infty=\im(Q_\circ)$ is harmonic and non--conformal
  (see Lemma~\ref{lem:harmonic} and Section~\ref{sec:euclidean}).
  Lemma~\ref{lem:no_ends} then implies that $f$ does not go through
  $\infty=\im(Q_\circ)$ such that $\infty$ induces a compatible Willmore
  connection on the bundle $V/L$. By Lemma~\ref{lem:holonomy_not_trivial}, for
  generic $\mu\in \C_*$ the holonomy of the associated family
  \eqref{eq:associated_willmoreconnection} of this Willmore connection has
  non--trivial eigenvalues. Together with Lemma~\ref{lem:2by2_to_4by4} this
  shows that we are in Case~II, that is, over generic $\mu$ the holonomy of
  the constrained Willmore associated family \eqref{eq:nabla_mu} has two
  non--trivial eigenvalues in addition to the two trivial eigenvalues.

  In particular, all three spectral curves arising in our situation are
  canonically isomorphic
  \[ \Sigma_{harm} = \Sigma_{hol} = \Sigma_{mult}. \] The second equality
  always holds for constrained Willmore tori of Case~II. The first equality
  holds, because by Lemma~\ref{lem:2by2_to_4by4} the Riemann surfaces
  $\Sigma_{harm}$ and $\Sigma_{hol}$ describing the non--trivial eigenlines of
  the $\SL(2,\C)$-- and $\SL(4,\C)$--holonomies, respectively, can be defined
  as normalizations of the same algebraic sets which describe the non--trivial
  eigenvalues of the holonomies (cf.\ the proof of
  Lemma~\ref{lem:holonomy_curve}).
\end{proof}

\subsection{CMC surfaces in Euclidean 3--space $\R^3$}\label{sec:cmc_r3}
The left and right normal vectors of a conformal immersion $f\colon M
\rightarrow \Im \H=\R^3$ coincide and $*df = N df = - df N$.

\begin{Lem} \label{lem:cmc_in_r3_iff_gauss_harmonic} A conformal immersion
  $f\colon M \rightarrow \Im \H=\R^3$ has constant mean curvature if and only
  if its Gauss map $N\colon M \rightarrow S^2$ is harmonic.
\end{Lem}

\begin{proof}
  By Lemma~\ref{lem:harmonic_map_into_2sphere}, harmonicity of $N$ is
  equivalent to closedness of $dN'$ which, by
  \eqref{eq:mcvector_and_left_right_normals}, is equivalent to $H$ being
  constant (note that $H$ for surfaces in $\R^3$ is the real function $H =
  -\frac12 \Tr \<df,dN\>$).
\end{proof}

Because CMC immersions are constrained Willmore and isothermic, the form
$\eta$ for which $d^\nabla(2{*}A + \eta)=0$ is not unique: for all
$\rho\in \R$, the forms $2{*}A_\circ^\rho = 2{*} A + \eta_0 + \rho \omega$ and
$2{*} Q^\rho_\circ = 2{*} Q + \eta_0 + \rho \omega$ are closed, where by
\eqref{eq:AQ_in_Euclidean_chart}
\[ 2{*} A = 
\Ad
\begin{pmatrix}
  1 & f \\
  0 & 1 
\end{pmatrix}
\begin{pmatrix}
  0 & 0 \\
  \frac12 H {*}dN'' & dN''
\end{pmatrix} \qquad\textrm{ and } \qquad \omega := \Ad
\begin{pmatrix}
  1 & f \\
  0 & 1 
\end{pmatrix}
\begin{pmatrix}
  0 & 0 \\
   dN'' & 0
\end{pmatrix}
\]
and $\eta_0:=-\frac12 H{*}\omega$. Because
\begin{equation*}
  2{*}A_\circ^\rho  =  \Ad
  \begin{pmatrix}
    1 & f \\
    0 & 1 
  \end{pmatrix}
  \begin{pmatrix}
    0 & 0 \\
    \rho dN'' & dN''
  \end{pmatrix} \quad \textrm{ and } \quad 
  2{*}Q_\circ^\rho  = \Ad
  \begin{pmatrix}
    1 & f \\
    0 & 1 
  \end{pmatrix}
  \begin{pmatrix}
    dN'' & 0 \\
    \rho dN'' & 0
  \end{pmatrix},
\end{equation*}
the 2--step forward and backward B\"acklund transforms for different
parameters $\rho$ are
\[ \check L_\rho = \ker(A_\circ^\rho) =
\begin{pmatrix}
  f \rho -1 \\ \rho
\end{pmatrix} \qquad \textrm{ and } \qquad \hat L_\rho = \im(Q_\circ^\rho) =
\begin{pmatrix}
  1 + f \rho \\ \rho
\end{pmatrix}
\]
and the only parameter for which one of these B\"acklund transforms is
constant is $\rho=0$ when $\check L_0=\hat L_0 = \infty$.

\subsection{CMC surfaces in the metrical 3--sphere $S^3$} \label{sec:cmc_s3}
Let $f\colon M \rightarrow S^3\subset \H$ be a conformal immersion with $*df=
N df = - df R$.  The imaginary 1--forms
\begin{align}
  \label{eq:s3_forms}
  \alpha= f^{-1} df\qquad \textrm{ and }\qquad \beta= df f^{-1}
\end{align}
satisfy $*\alpha = R \alpha = - \alpha R$ and $\beta = N \beta = -\beta N$.
Their Maurer--Cartan equations are
\[ d\alpha + \alpha\wedge \alpha = 0\qquad \textrm{ and }\qquad d\beta = \beta
\wedge \beta. \] Denote by $n$ the positive oriented normal vector of $f$ as a
surface in $S^3$, that is, $f$, $n$ is a positive orthonormal basis of the
normal bundle of $f$ seen as an immersion into $\H$. Because the complex
structure on the normal bundle is given by left multiplication by $N$ and
right multiplication by $R$ we have $n = N f=fR$.  The second fundamental form
of $f$ as an immersion into $\H$ is $\sff = -\<df,df\>f -\< df,dn\>n$ and its
mean curvature vector is $\mathcal{H} = \frac 12 \Tr\sff = \HS n -f$, where
$\HS$ denotes the scalar mean curvature of $f$ as an immersion into $S^3$.
Hence $H=\bar{\mathcal{H}} N = R \bar{\mathcal{H}}$ satisfies
\begin{align}
  \label{eq:s3_H}
  H = (\HS-R)f^{-1} = f^{-1} (\HS -N).  
\end{align} A straightforward computation (using \eqref{eq:s3_forms} together
with \eqref{eq:mcvector_and_left_right_normals} and \eqref{eq:s3_H}) shows
\begin{equation}
  \label{eq:mc_cmc_s3}
  d{*}\alpha + \HS \alpha\wedge \alpha = 0\qquad \textrm{ and }\qquad
  d{*}\beta + \HS \beta \wedge \beta=0.
\end{equation}

\begin{Lem} \label{lem:cmc_in_s3_iff_gauss_harmonic} A conformal immersion
  $f\colon M \rightarrow S^3\subset \H$ with $*df = N df = -df R$ has
  constant mean curvature in $S^3$ if and only if $N\colon M \rightarrow
  S^2$ is harmonic or, equivalently, $R\colon M \rightarrow S^2$ is
  harmonic.

\end{Lem}
\begin{proof}
  By Lemma \ref{lem:harmonic_map_into_2sphere}, the map $R$ is harmonic if and
  only if the form $dR'$ is closed. Because $dR'= - (\HS -R)f^{-1} df$, this
  is equivalent to
  \[ 0=d(dR') = -d\HS\wedge \alpha - \HS d\alpha + d{*}\alpha = -d\HS\wedge
  \alpha\] such that $f$ is CMC in $S^3$ if and only if $R$ is harmonic. The
  proof for $N$ is analogous.
\end{proof}

As for CMC surfaces in $\R^3$, the form $\eta$ with $d^\nabla(2{*}A+ \eta)=0$
is not unique: for all $\rho\in \R$, the forms $2{*}A_\circ^\rho = 2{*} A +
\eta_0 + \rho \omega$ and $2{*}Q_\circ^\rho = 2{*} Q + \eta_0 + \rho \omega$
are closed, where
\[ 2{*}A= \Ad
\begin{pmatrix}
    1 & f \\
    0 & 1
  \end{pmatrix}
  \begin{pmatrix}
    0 & 0 \\
    \frac12(1-\HS R) dR'' f^{-1} & dR''
  \end{pmatrix}, \qquad \omega := \Ad
  \begin{pmatrix}
    1 & f \\
    0 & 1
  \end{pmatrix}
  \begin{pmatrix}
    0 & 0 \\
    dH & 0
  \end{pmatrix}
\]
and $\eta_0 := \frac12 \HS S \omega$, because by
\begin{multline}
  \label{H_is_dual_for_CMC_in_sphere}
  dH = -dR f^{-1} - (\HS -R) f^{-1} df f^{-1} =\\
  -(-H df+ dR'')f^{-1} - H df f^{-1} = - dR'' f^{-1}
\end{multline}
and $N=fRf^{-1}$ which implies $dN= \Ad(f)(dR+2\alpha R)$ and therefore
$dN''=\Ad(f)(dR'')$, the form $w= - \frac12(dH+R{*}dH + H NdN'')$ occurring in
\eqref{eq:AQ_in_Euclidean_chart} is $w = \frac12(1-\HS R) dR'' f^{-1}.$ Thus
\[
2{*}A_\circ^\rho = \Ad
\begin{pmatrix}
  1 & f \\
  0 & 1
\end{pmatrix}
\begin{pmatrix}
  0 & 0 \\
  (\rho -\frac12) dH & dR''
\end{pmatrix} \quad \textrm{ and } \quad 2{*}Q_\circ^\rho = \Ad
\begin{pmatrix}
  1 & f \\
  0 & 1
\end{pmatrix}
\begin{pmatrix}
  dN'' & 0 \\
  (\rho +\frac12) dH & 0
\end{pmatrix}
\]
and the 2--step forward and backward B\"acklund transforms for different
$\rho$ are
\[
\check L_\rho = \ker(2{*}A_\circ^\rho) =
\begin{pmatrix}
  (\rho+\tfrac12)f \\ (\rho-\tfrac12)
\end{pmatrix}\H \quad \textrm{ and } \quad
\hat L_\rho = \im(2{*}Q_\circ^\rho) =
\begin{pmatrix}
  (\rho-\tfrac12)f \\ (\rho+\tfrac12)
\end{pmatrix}\H,
\]
because $dH = -dR'' f^{-1}= -f^{-1} dN''$.  In particularly, for $\rho=0$ and
$\rho = \pm 1/2$ we obtain
\begin{align}
  \label{eq:cmc_S3_B_constant}
  & \check L_0 = \hat L_0 =
  \begin{pmatrix}
    -f \\ 1
  \end{pmatrix}\H  \\
  & \check L_{1/2} =
  \begin{pmatrix}
    1 \\ 0
  \end{pmatrix}\H \qquad \textrm{ and } \qquad
  \hat L_{1/2} =
  \begin{pmatrix}
    0 \\ 1
  \end{pmatrix}\H \\
  & \check L_{-1/2} =
  \begin{pmatrix}
    0 \\ 1
  \end{pmatrix}\H \qquad \textrm{ and } \qquad
  \hat L_{-1/2} =
  \begin{pmatrix}
    1 \\ 0
  \end{pmatrix}\H.
\end{align}
The parameters $\rho=\pm \tfrac 12$ are the only parameters for which the
2--step B\"acklund transforms $\check L_\rho$ and $\hat L_\rho$ are constant.

Minimal surfaces in $S^3$ are examples of isothermic surfaces that, for
different choices of $\eta\in \Omega^1(\mathcal{R})$ in the Euler--Lagrange
equation \eqref{eq:def_constrained-willmore}, belong to different cases of
Theorem~\ref{the:mainA}: for $\rho=0$ their holonomy spectral curve belongs to
Case~I, see Corollary~\ref{cor:willmore}, while for $\rho = \pm \frac12$ it
belongs to Case~II, because then $\ker(A_\circ^\rho)$ and $\im( Q^\rho_\circ)$
are constant, see Proposition~\ref{pro:cases_harmonic}.

\subsection{M\"obius geometric characterization of CMC surfaces in $\R^3$ and
  $S^3$}
As we have seen in Sections~\ref{sec:cmc_r3} and \ref{sec:cmc_s3}, CMC
surfaces in $\R^3$ and $S^3$ are examples of constrained Willmore surfaces for
which the 1--form $\eta\in \Omega^1(\mathcal{R})$ in
\eqref{eq:def_constrained-willmore} can be chosen such that both
$\ker(A_\circ)$ and $\im(Q_\circ)$ are constant.  The following
characterization of constrained Willmore surfaces with this property is
readily verified by combining equations \eqref{eq:AQ_in_Euclidean_chart},
\eqref{eq:eta} and \eqref{eq:mcvector_and_left_right_normals}.

\begin{Lem}\label{lem:both_constant}
  Let $f\colon M \rightarrow S^4=\HP^1$ be a constrained Willmore immersion.
  Then:
  \begin{itemize}
  \item[a)] The immersion $f$ admits $\eta\in \Omega^1(\mathcal{R})$ such that
    $\ker(A_\circ) =\im(Q_\circ)=\infty$ if and only if $f\colon M \rightarrow
    \H=\HP^1\backslash\{\infty\}$ has the property that $H$ is constant, that
    is, $f$ is CMC in a 3--dimensional plane in $\H$ of minimal in Euclidean
    4--space $\H$.
  \item[b)] The immersion $f$ admits $\eta\in \Omega^1(\mathcal{R})$ thus that
    $\ker(A_\circ)=0$ and $\im(Q_\circ)=\infty$ if and only if $f\colon M
    \rightarrow \H=\HP^1\backslash\{\infty\}$ satisfies $Hf+R=c$ for some
    constant $c\in \H$. Similarly, $\ker(A_\circ)=\infty$ and $\im(Q_\circ)=0$
    is equivalent to $fH+N = c$ for $c\in \H$. The constant $c\in \H$ is real
    if and only if the surface is contained in a concentric 3--sphere in $\H$.
    In particular, $f$ is then CMC in that 3--sphere.
  \end{itemize}
\end{Lem}

This lemma directly implies the following characterization of CMC surfaces in
$\R^3$ and $S^3$.

\begin{Cor}
  A constrained Willmore immersion $f\colon M \rightarrow S^4=\HP^1$ is CMC
  with respect to a 3--dimensional Euclidean or spherical subgeometry if and
  only if it is contained in a totally umbilic 3--sphere and admits a 1--form
  $\eta\in \Omega^1(\mathcal{R})$ satisfying
  \eqref{eq:def_constrained-willmore} such that the 2--step B\"acklund
  transformation $\check L=\ker(A_\circ)$ is constant. The Euclidean case is
  then characterized by the fact that the point $\check L$ is contained in the
  totally umbilic 3--sphere.
\end{Cor}
\begin{proof}
  We have already seen in Lemma~\ref{lem:both_constant} that CMC surfaces in
  $\R^3$ and $S^3$ have the given properties.  To prove the converse, assume
  that $f$ is contained in a totally umbilic 3--sphere. This can be
  represented as the null lines of an indefinite quaternionic hermitian form
  $\<\, . \,\>$, see e.g.\ Section~10.1 of \cite{BFLPP02}. Then $Q_\circ = -
  A_\circ^*$ with $*$ denoting the adjoint with respect to $\<\, . \,\>$. In
  particular, for surfaces contained in a totally umbilic 3--sphere,
  $\ker(A_\circ)$ being constant is equivalent to $\im(Q_\circ)$ being
  constant and the corollary follows from Lemma~\ref{lem:both_constant}.
\end{proof}

The following lemma gives another characterization of CMC surfaces in $\R^3$
or $S^3$.

\begin{Lem}
  A conformal immersion $f\colon M \rightarrow S^4$ admits a point $\infty \in
  S^4$ at infinity such that both the left and right normal vectors $N$ and
  $R$ of $f$ seen as an immersion into $\R^4= S^4\backslash\{\infty\}$ are
  harmonic if and only if $f\colon M \rightarrow \R^4\cong
  S^4\backslash\{\infty\}$ is CMC in a 3--plane or a round 3--sphere in $\R^4$
  or minimal in Euclidean 4--space $\R^4$.
\end{Lem}
\begin{proof}
  By Lemma~\ref{lem:harmonic_map_into_2sphere} and
  \eqref{eq:mcvector_and_left_right_normals}, harmonicity of $N$ and $R$ is
  equivalent to $*dH=-R dH=dH N$ which again is equivalent to the mean
  curvature vector $\mathcal{H}$ of $f$ being a parallel section of the normal
  bundle of $f$. It is well know that a surface in $\H=\R^4$ with parallel
  mean curvature vector is either Euclidean minimal or CMC in a 3--plane or
  3--sphere.
\end{proof}

\subsection{Hamiltonian stationary Lagrangian tori and Lagrangian tori with
  conformal Maslov form in $\C^2$}
\label{sec:lagrangian}
We discuss two classes of examples of constrained Willmore surfaces which are
related to harmonic maps to $S^2$ but in general not CMC in $\R^3$ or $S^3$.

We identify $\C^2$ with $\H$ via $(z_1,z_2)\mapsto z_1+ j z_2$ and equip it
with the standard symplectic form $\omega$ defined by $\omega(x,y)= \< x
i,y\>$, where $\<x,y\> =\Re(\bar x y)$ is the usual Euclidean product on
$\H=\R^4$. An immersion $f\colon M \rightarrow \H$ is Lagrangian with respect
to $\omega$ if and only if its tangent and normal bundles $T_fM$ and
$\perp_f\!M$ are related via $\perp_f\!M=(T_f M) i$ or, equivalently, if its
right normal vector is of the form $R= j\exp(i \beta)$ for a $\R/2\pi
\Z$--valued function $\beta$ called the Lagrangian angle.

It is shown in \cite{HR1} that $f$ is Hamiltonian stationary Lagrangian (i.e.,
a Lagrangian immersion that is stationary for the area functional under all
Hamiltonian variations) if and only if $\beta$ is harmonic.  Because $2dR'=j
\exp(i \beta)i d\beta + i {*} d\beta$, harmonicity of $\beta$ is equivalent to
harmonicity of $R$, see Lemma~\ref{lem:harmonic_map_into_2sphere}. In
particular, by Lemma~\ref{lem:harmonic}, Hamiltonian stationary Lagrangian
surfaces are constrained Willmore and admit a form $\eta$ such that $\im(
Q_\circ)=\infty$ for $2{*}Q_\circ=2{*}Q+\eta$. The above formula for $dR'$
together with \eqref{eq:mcvector_and_left_right_normals} immediately implies
that the special case of Lagrangian surfaces that are minimal is characterized
by the property that $R$ is constant, that is, such surfaces are complex
holomorphic with respect to the complex structure given by right
multiplication with $-R$.

While Lagrangian minimal surfaces can never be compact, there are compact
Hamiltonian stationary Lagrangian surfaces. All Hamiltonian stationary
Lagrangian tori are explicitly described in \cite{HR2}. They are examples of
constrained Willmore tori with trivial normal bundle and spectral genus $g=0$
(the latter, because $R$ is a harmonic map into $S^2$ which takes values in a
great circle).

It is shown in \cite{CU} that Lagrangian surfaces with conformal Maslov form
are characterized by the property that the left normal vector $N$ is harmonic
while the right normal vector $R$ takes values in the great circle
perpendicular to $i$. All Lagrangian tori with conformal Maslov form are
explicitly described in \cite{CU}. They are examples of constrained Willmore
tori with trivial normal bundle and spectral genus $g\leq 1$ (the latter,
because the harmonic map $N$ into $S^2$ is equivariant).

\appendix

\section{}

The following lemma is needed in the proof of Lemma~\ref{lem:no_ends}.

\begin{Lem} \label{lem:singularity} Let $\varphi\in H^0(L)$ be a holomorphic
  section of a quaternionic holomorphic line bundle $L$ and denote by $N$ the
  $C^{\infty}$--map with values in $S^2$ that is defined away from the zeros
  of $\varphi$ by $J\varphi=\varphi N$. This map $N$ continuously extends
  through the zeros of $\varphi$. Moreover, in case $N$ is $C^1$ at a zero $p$
  of $\varphi$ the Hopf field $Q$ of $L$ has to vanish at that point $p$.
\end{Lem}

\begin{proof}[Proof of Lemma~\ref{lem:singularity}]
  Let $z=x+iy$ be a chart centered at a zero $p$ of $\varphi$. Then locally
  there is a nowhere vanishing section $\psi$ such that
  \[ \varphi = \psi((x+yR)^n \lambda_n + O(n+1)),\] where $J\psi=\psi R$ and
  $\lambda_n\in \H\backslash\{0\}$, cf.\ \cite{PP}. In particular,
  $\varphi=\psi (x+yR)^n \lambda$ for a continuous function $\lambda$ with
  $\lambda(p)\neq 0$. Now $N=\lambda^{-1} R \lambda$ implies that $N$
  continuously extends through the zero $p$. Moreover, away from $p$ the Hopf
  field is given by $Q\varphi = \varphi \tfrac12 N dN''$, hence
  \[ Q\psi = \psi \tfrac12 (x+Ry)^n (x-Ry)^{-n} \lambda N dN'' \lambda^{-1}.\]
  Because the left hand side is well defined and continuous at $p$ while
  $(z/\bar z)^n$ is bounded but not continuous at zero, the form $dN''$ has to
  vanish at $p$ in case it is continuous. This implies that $Q_p=0$.
\end{proof}

\subsection*{Acknowledgments}

The author thanks Franz Pedit and Ulrich Pinkall for helpful discussions.



\begin{thebibliography}{10}

\bibitem{Bob} A.~Bobenko, All constant mean curvature tori in $R^3$,
  $S^3$, $H^3$ in terms of theta functions. \emph{Math.\ Ann.} \textbf{290}
  (1991), 209--245.

\bibitem{Boh03} C.~Bohle, \emph{M{\"o}bius Invariant Flows of Tori in $S^4$}.
  Thesis, TU--Berlin, 2003.

\bibitem{BLPP} C.~Bohle, K.~Leschke, F.~Pedit, and U.~Pinkall, Conformal maps
  from a 2--torus to the 4--sphere.  \url{http://xxx.arxiv.org/abs/0712.2311}.

\bibitem{BPP2} C.~Bohle, F.~Pedit, and U.~Pinkall, Spectral curves of
  quaternionic holomorphic line bundles over 2--tori.  Preprint.

\bibitem{BPP1} C.~Bohle, G.~P.~Peters, and U.~Pinkall, Constrained Willmore
  surfaces. To appear in \emph{Calc.\ Var.\ Partial Differential Equations}.


\bibitem{BFPP} F.~Burstall, D.~Ferus, F.~Pedit, and U.~Pinkall,
  Harmonic tori in symmetric spaces and commuting Hamiltonian systems on loop
  algebras.  \emph{Ann.\ of Math.}  \textbf{138} (1993), 173--212.

\bibitem{BFLPP02} F.~Burstall, D.~Ferus, K.~Leschke, F.~Pedit, and U.~Pinkall,
  \emph{Conformal Geometry of Surfaces in $S\sp 4$ and Quaternions}. Lecture
  Notes in Mathematics, 1772, Springer--Verlag, Berlin, 2002.
    

\bibitem{BPP02} F.~Burstall, F.~Pedit, and U.~Pinkall, Schwarzian
  derivatives and flows of surfaces. \emph{Contemp.\ Math.}  \textbf{308}
  (2002), 39--61. 

\bibitem{Bu07} F.~Burstall, private communication. (See also: Twistors,
  4--symmetric spaces and integrable systems, in \emph{Progress in Surface
    Theory}, Oberwolfach, April 29--May 5, 2007.  To appear in
  \emph{Oberwolfach Rep.})

\bibitem{CU} I.~Castro and F.~Urbano, Lagrangian surfaces in the complex
  Euclidean plane with conformal Maslov form.  \emph{Tohoku Math. J.}
  \textbf{45} (1993), 565--582.

\bibitem{EW83} J.~Eells and J.~C.~Wood, Harmonic maps from surfaces into
  projective spaces. \emph{Adv.\ Math.} \textbf{49} (1983), 217--263.

\bibitem{FPPS} D.~Ferus, F.~Pedit, U.~Pinkall, and I.~Sterling,
  Minimal tori in $S\sp 4$.  \emph{J.\ Reine Angew.\ Math.} \textbf{429}
  (1992), 1--47.

\bibitem{FLPP01} D.~Ferus, K.~Leschke, F.~Pedit, and U.~Pinkall, Quaternionic
  holomorphic geometry: Pl{\"u}cker formula, Dirac eigenvalue estimates and
  energy estimates of harmonic $2$--tori.  \emph{Invent.\ Math.}  {\bf 146}
  (2001), 507--59.
  

\bibitem{GS98} P.~G.~Grinevich and M.\ U.\ Schmidt, Conformal invariant
  functionals of immersions of tori into~$R\sp 3$.  \emph{J.\ Geom.\ Phys.}
  \textbf{26} (1998), 51--78.

\bibitem{He04} 
F.~Helein, 
Removability of singularities of harmonic maps into pseudo--Riemannian
manifolds. \emph{Ann.~Fac.~Sci.~Toulouse Math.} \textbf{XIII} (2004), 45--71.

%
\bibitem{HR1}
F.~Helein and P.~Romon,
Hamiltonian stationary Lagrangian surfaces in $\C^2$.
\emph{Comm.~Anal.~Geom.} \textbf{10} (2002), 79--126. 

\bibitem{HR2}
F.~Helein and P.~Romon,
Weierstrass representation of Lagrangian surfaces in four--dimensional space
using spinors and quaternions. 
\emph{Comment.~Math.~Helv.} \textbf{75} (2000), 668--680.

\bibitem{HJ03}
U.~Hertrich--Jeromin,
\emph{Introduction to M\"obius Differential Geometry}.
Cambridge University Press, 2003.

\bibitem{Hi} N.~Hitchin, Harmonic maps from a $2$--torus to the $3$--sphere.
  \emph{J.\ Differential Geom.}  \textbf{31} (1990), 627--710.

\bibitem{LPP05} K.~Leschke, F.~Pedit, and U.~Pinkall, Willmore tori
  in the 4--sphere with nontrivial normal bundle.  \emph{Math.~Ann.}
  \textbf{332} (2005), 381--394.

\bibitem{PP} F.~Pedit and U.~Pinkall, Quaternionic analysis on Riemann
  surfaces and differential geometry.  \emph{Doc.\ Math.\ J.\ DMV}, Extra
  Volume ICM Berlin 1998, Vol.\ II, 389--400.

\bibitem{Pi} U.~Pinkall, Hopf tori in $S\sp 3$.  \emph{Invent. Math.}
  \textbf{81} (1985), 379--386.

\bibitem{PS87} U.~Pinkall and I.~Sterling, Willmore surfaces.
  \emph{Math.\  Intelligencer} \textbf{9} (1987), 38--43.

\bibitem{PS89}
U.~Pinkall and I.~Sterling,
On the classification of constant mean curvature tori.
\emph{Ann.\ of Math.} \textbf{130} (1989), 407--451.


\bibitem{S02} M.~U.~Schmidt, A proof of the Willmore conjecture.
  \eprint{arxiv.org/math.DG/0203224}.

\bibitem{Ta98} I.~A.~Taimanov, The Weierstrass representation of
  closed surfaces in $R^3$.  \emph{Funct.\ Anal.\ Appl.}  \textbf{32}
  (1998), 49--62.

\bibitem{W} T.~J.~Willmore, \emph{Riemannian Geometry}. Oxford University
  Press, Oxford, New York, 1993.

\end{thebibliography}
\end{document}